\documentclass[a4paper, leqno]{article}
\usepackage{yhmath}
\usepackage{tgpagella}
\usepackage{xcolor}
\usepackage{euler}
\usepackage[mathscr]{euscript}
\usepackage[T1]{fontenc}
\usepackage{yfonts}
\usepackage[a4paper, total={6in, 8in}]{geometry}
\usepackage[utf8]{inputenc}
\usepackage{latexsym}
\usepackage{stmaryrd}
\usepackage{makeidx}
\usepackage{xr}
\usepackage{mathrsfs}
\usepackage{mathtools}
\usepackage{tikz-cd}
\usepackage{tikz}
\usetikzlibrary{positioning}
\usepackage{amscd}
\usepackage{syntonly}
\usepackage{amssymb}
\usepackage{amsfonts}
\usepackage{amsmath}
\usepackage{amsthm}
\usepackage[english]{babel}
\usepackage{chngcntr}
\usepackage{xcolor}
\usepackage[numbers]{natbib}
\usepackage{adjustbox}
\usepackage{quiver}
\usepackage{enumitem}
\usepackage{hyperref}
\usepackage{cleveref}
\usepackage{comment}
\usepackage{breqn}
\counterwithin{equation}{section}
\newtheorem{theorem}{Theorem}[section]
\newtheorem{proposition}{Proposition}[section]
\newtheorem{definition}{Definition}[section]

\newtheorem{lemma}{Lemma}[section]
\newtheorem*{note*}{Note}
\newtheorem*{lemma*}{Lemma}
\newtheorem*{examples*}{example}
\newtheorem*{example*}{example}
\newtheorem*{corollary*}{Corollary}
\newtheorem*{theorem*}{Theorem}
\newtheorem*{Reminder*}{Reminder}
\newtheorem*{Acknowledgement*}{Acknowledgment}
\newtheorem*{NotationConvention*}{Notation Convention}

\def\dotminus{\mathbin{\ooalign{\hss\raise1ex\hbox{.}\hss\cr
  \mathsurround=0pt$-$}}}
  
\author{Ali Hamad \footnote{Affiliation: University of Ottawa, Email: ahama099@uottawa.ca  ORCID: 0009-0005-5472-6776}}
\title{Bundles of metric structures as left ultrafunctors}

\date{June 2026}
\begin{document}
\oddsidemargin 0in
\evensidemargin 0in
\topmargin=0in
\textwidth=6in
\textheight=8in
\maketitle
\begin{abstract}
We pursue the study of ultracategories initiated by Makkai, and more recently
Lurie, by looking at properties of ultracategories of complete metric structures,
i.e. coming from continuous model theory, instead of ultracategories of
models of first-order theories. Our main result is that for any continuous
theory $\mathbb{T}$, there is an equivalence between the category of left
ultrafunctors from a compact Hausdorff space $X$ to the category of
$\mathbb{T}$-models and a notion of bundle of $\mathbb{T}$-models over
$X$. The notion of bundle of $\mathbb{T}$-models is new but recovers many
classical notions, like bundles of Banach spaces, bundles
of Hilbert spaces and (semi-)continuous fields of $\mathrm{C}^{*}$-algebras.
\end{abstract}
\tableofcontents

%%*************** Text entry area ******************%%

\section*{Introduction}

 Ultracategories are a categorical axiomatisation of the idea of a ``category
with ultraproduct functor'':
\begin{equation*}
\int _{X} \bullet d\mu : A^{X} \to A,
\end{equation*}
for each set $X$ and ultrafilter $\mu $ on $X$, which are meant to capture
the idea of ultraproduct of a family of objects. The notion has been originally
introduced by Makkai in \cite{makkai88strong,MAKKAI198797}; Makkai's ultracategories
were studied further in \cite{ZawadowskiPhD}, where they were used to answer
questions regarding the class of formulas preserved by an interpretation
in coherent logic, and in \cite{marmolejo} where the idea of continuous
families of models was first introduced. A simplified (non-equivalent)
version of the definition was given more recently by Lurie in
\cite{lurie2018ultracategories}, which is the notion we use in the present
paper. The primary goal of the introduction of Ultracategories was to study
the ultraproduct construction in model theory, and to show a conceptual
completeness theorem stating (in Lurie's case) that for any coherent theory
$\mathbb{T}$, we have an equivalence of categories between
$\mathrm{Lult}(\mathbb{T}\text{-}\mathrm{Models},\mathsf{Set})$, and the
classifying topos of $\mathbb{T}$. Here, $\mathrm{Lult}$ is short for Left
ultrafunctors, which are ultraproduct respecting functors in a lax sense.

It is also possible to have a geometric ``flavour'' of ultracategories.
A classic result by Manes establishes that compact Hausdorff spaces are
exactly algebras, for the ultrafilter monad. Decoding this, the algebra
structure can be viewed as a function that associates to every ultrafilter
on a compact Hausdorff space its unique limit, and this completely determines
the topology of the space. It turns out that Manes' theorem can be expressed
in an ``ultracategoric'' way. Namely, the convergence of ultrafilters data
defines an ultrastructure on the underlying set of the compact Hausdorff
space (viewed as a discrete category). Moreover, we can upgrade this into
an equivalence of categories between \emph{ultrasets} (ultracategories for which
there is no non-identity morphisms), and \emph{Compact Hausdorff spaces}
\cite[theorem 3.1.5]{lurie2018ultracategories}.

A fundamental result by Lurie~\cite[3.4]{lurie2018ultracategories} is the
statement that a \emph{left ultrafunctor} from a compact Hausdorff space $X$
regarded as an ultraset to $\mathsf{Set}$ is the same thing as a sheaf
over $X$. This theorem captures the idea that a sheaf is nothing but a
continuous family of sets depending on a parameter from a set $X$. This
point of view is captured by another classic result stating an equivalence
between $Sh(X)$ and \'{e}tale bundles over $X$.

This paper presents two goals; the first is to study new examples of ultracategories
namely those of metric structures (like Banach spaces, Hilbert spaces,
$\mathrm{C}^*$-algebras, tracial von Neumann algebras) and the second
is to extend the left ultrafunctor-\'{e}tale bundles equivalence of Lurie
to this class of new examples. In other words, we want to show that for
any ultraset (compact Hausdorff space) left ultrafunctors from this ultraset
to the category of Banach spaces (Hilbert spaces, $\mathrm{C}^{*}$-algebras
etc.) are equivalent to the already known notions of Banach (Hilbert, $
\mathrm{C}^{*}$  etc.) bundles.

It turned out that a very good framework to deal with these metric structures
is continuous model theory. In our work, we used
\cite{farah2021model} as a reference for this. The major advantage of continuous
model theory is that it allows us to consider a natural notion of ultraproduct
on the category of models, and hence makes this category an ultracategory.
This confirms that we are not deviating from the logic point of view when
studying ultracategories, and in spirit this confirms also that ultracategories
are a very natural setting when doing model theory, i.e. when constructing
the category of models, one should care about objects, morphisms, and ultraproducts.

Although the simplest structure studied by continuous model theory is a
complete metric space bounded by a certain constant, it is capable of axiomatising
many structures like Banach spaces, Hilbert spaces, $\mathrm{C}^{*}$-algebras,
preduals of von Neumann algebras, and von Neumann algebras with a faithful
normal state (usually called sigma-finite von Neumann algebras, or
$\mathrm{W}^{*}$ probability spaces). For $\mathrm{C}^{*}$-algebras see
\cite[page 11 example 2.2.1]{farah2021model}, for sigma-finite von Neumann
algebras see \cite{dabrowski2019continuous}, or for a different approach
in \cite{arulseelan2025totally}, for tracial von Neumann algebras see
\cite{goldbring2022survey} and for a general survey on continuous model
theory see \cite{hart2023introduction}.

One reason to expect this class of ultracategories to be fundamentally
different from the class of ultracategories captured by coherent logic
(this is the intersection between first-order logic and the logic we study
in topos theoretic setting namely geometric logic) is the fact that, for
most relevant structures captured by continuous first-order logic, the
category of left ultrafunctors to $\mathsf{Set}$ turns out to be trivial, this is
shown in the subsection \ref{not_coherent_VTEX1}. On the other hand this category
turns out to be none other than the classifying topos of the theory for
coherent theories by Makkai-Lurie Conceptual completeness.

Bundle theory plays an important role in functional analysis. They are, as their name suggests, families of metric structures depending on a parameter.
Algebras of operators fields (this means bundles whose fibres are operator
algebras) were first introduced in the works of Fell and Dixmier
\cite{fell_structure_1961,dixmier1982c}, to study non-commutative
$\mathrm{C}^{*}$-algebras. The $\mathrm{C}^{*}$-algebraic flavour of these
usually goes by the name of \emph{continuous fields of $\mathrm{C}^{*}$-algebras},
which were introduced in the references stated before
\cite{fell_structure_1961,dixmier1982c} (for other various equivalent definitions
and results regarding these see
\cite{dadarlat2009continuous,dupre1974hilbert,Niels,Williams2007CrossedPO}).
Continuous fields of $\mathrm{C}^{*}$-algebras come in upper semi-continuous
and continuous versions, with the semi-continuous version being more prominent.
On the other hand the study of the von Neumann algebra version of these
is much more recent and due to the work of Ozawa
\cite{ozawa2013dixmier}, who introduced \emph{tracially continuous $\mathrm{W}^*$-bundles}
and these were studied extensively in the works of Evington and Pennig
\cite{evington2016locally}. These can serve as a tool to study the trace
space of a von Neumann algebra, when all fibres have the same underlying
von Neumann algebra.

Our main result turns out to be the following: for a compact Hausdorff
space, there is an equivalence between left ultrafunctors from $X$ to the
category of models of a continuous theory and the notion that we introduce,
which we called ``\emph{bundles of models of a continuous first-order
theory}'':

\begin{theorem*} \ref{third_theorem_VTEX1}
Let $X$ be a compact Hausdorff space then there is
an equivalence of categories between the category of left ultrafunctors
from $X$ to the category of models of a continuous first order theory
$\mathbb{T}$ and the category of bundles of models of $\mathbb{T}$.
\end{theorem*}

and
\begin{theorem*} \ref{functoriality_in_the_compact_space_VTEX1} This equivalence depends functorially
on the compact Hausdorff space $X$.
\end{theorem*}

These bundles of models recover already studied notions like continuous
and semi-continuous bundles of Banach spaces (see
\cite{hofmann1977bundles} for semi-continuous Banach bundles or
\cite{Fell1977} for the continuous version), bundles of Hilbert spaces,
continuous fields of $\mathrm{C}^{*}$-algebras and $\mathrm{W}^*$-bundles (bundles
of tracial von Neumann algebras) which we already discussed, this uncovers
that left ultrafunctors to categories of metric structures are the adequate
notion of continuous families of this metric structure. This touches on
another important point, which is the fact that the ultraproduct construction
can be regarded as a generalised topology on a specific category. Another
advantage of this new notion of bundles of models, is that it may help
uncover new notions of bundles in functional analysis (not studied in this
paper), like a non-commutative counterparts of tracially continuous
$\mathrm{W}^*$-bundles.

This idea of continuous family of models was already present in the work
of Marmolejo \cite{marmolejo}. But in his case the attention was only towards
coherent first-order theories. In his definition, for $X$ a topological
space, and $\mathcal{P}$ a pretopos. A continuous family of models is nothing
but a pretopos morphism from $\mathcal{P}$ to $Sh(X)$. These are also going
to be equivalent to left ultrafunctors from $X$ to
$\mathrm{Mod}(\mathcal{P})$, if $X$ is compact Hausdorff
\cite[Corollary 2.2.7]{lurie2018ultracategories}.

\subsection*{Outline of results and methodology}

Sections~\ref{1} and \ref{2} provide an overview of the literature on Ultracategories
and on the category of complete metric spaces. The goal is only to introduce
the definitions and results we will need throughout the paper.

\subsubsection*{Establishing the equivalence for bounded complete metric spaces}

In sections~\ref{3} and \ref{4}, we study the case
$\mathcal{M}=\mathsf{k}\text{-}\mathrm{CompMet}$ the category of complete
metric spaces where the distance function is bounded by a certain
$k$ and with \emph{contractions} as morphisms, that is morphisms satisfying
$d(f(x),f(y)) \leqslant d(x,y)$, or $1$-Lipschitz maps. In section~\ref{3}, we defined what we mean by the category of bundles of complete
metric spaces. The next step is to define an assignment that gives a bundle
of bounded metric spaces, for each left ultrafunctor from $X$ to
$\mathsf{k}\text{-}\mathrm{CompMet}$. In section~\ref{4}, we constructed
an inverse process to the previous one, which leads us to our first important
theorem:

\begin{theorem*}~\ref{first_theorem_VTEX1} There exists an equivalence of categories between
$\mathrm{Lult}(X,\mathsf{k\text{-}CompMet})$ and
$\mathrm{Bun}(\mathsf{k\text{-}CompMet},\allowbreak X)$.
\end{theorem*}
\noindent
Here $\mathrm{Bun}(\mathsf{k\text{-}CompMet},X)$ are bundles of bounded
complete metric spaces bounded by $k$ over $X$.

\subsubsection*{Extending the equivalence to structures of continuous model theory}

First, we introduced continuous model theory, an extension
of classic first-order logic that allows the axiomatisation of structures
of metric nature. Continuous models are going to be our main building block;
they come equipped with a natural notion of ultraproduct turning them into
ultracategories.

This is an entirely new class of ultracategories. In particular we have
the following theorem
\begin{theorem*}~\ref{not_coherent_theorem_VTEX1} Let $\mathcal{M}$ be an ultracategory of models
of a continuous first order theory, that satisfies the following two conditions:
\begin{enumerate}
\item For any $A ,B \in \mathcal{M}$, $\mathrm{Hom}(A,B)$ is connected
with the topology of pointwise convergence.
\item The category $\mathcal{M}$ has a zero object.
\end{enumerate}
then the only left ultrafunctors $\mathcal{M} \to \mathsf{Set}$ are constants.
\end{theorem*}

In particular this theorem shows that these categories of models are not
models of coherent first order logic, since, if they were, then the category
$\mathrm{Lult}(\mathcal{M},\mathsf{Set})$ should be the classifying topos
of this theory.

In the context of model theory, one can think of a signature as a theory
with an empty set of axioms, in the sense that $\Sigma $-structures are
the models of the empty theory in the signature $\Sigma $, models
of such theory turn out to be exactly the complete bounded metric spaces.
In section~\ref{5}, we generalise the theorem above to the continuous version
of such empty theories: We extended our work from just defining bundles
of bounded complete metric spaces to defining bundles of structures of
continuous model theory. And we show that our previous result can be extended
to:

\begin{theorem*}~\ref{second_theorem_VTEX1} Let $X$ be a compact Hausdorff space then there
is an equivalence of categories between
\newline
$\mathrm{Lult}(X,\mathrm{CompMet}_{\mathfrak{L}})$ and the category
$\mathrm{Bun}(\mathrm{CompMet}_{\mathfrak{L}},X)$.
\end{theorem*}

\subsubsection*{Extending the equivalence to models of continuous model theory}

Finally in section~\ref{6}, we show that the equivalence above is restricted
to one between ``bundles of models of a continuous theory'', i.e. the bundles
of structures in which every fibre is a model of the theory and left ultrafunctors
taking value in the category of models. In particular:

\begin{theorem*}~\ref{third_theorem_VTEX1} Let $X$ be a compact Hausdorff space then there is
an equivalence of categories between
\newline
$\mathrm{Lult}(X,\mathrm{CompMet}_{\mathfrak{L},\mathbb{T}})$ and the category
$\mathrm{Bun}(\mathrm{CompMet}_{\mathfrak{L},\mathbb{T}},X)$.
\end{theorem*}

\subsubsection*{Showing that the construction is natural in the compact Hausdorff space}

Let $Y \xrightarrow{f} X$ be a continuous map between compact Hausdorff
spaces, and suppose we have a left ultrafunctor $\mathcal{F}$ from
$X$ to $\mathcal{M}$ where $\mathcal{M}$ is the category of models of some
continuous model theory (could be as simple as complete metric spaces bounded
by a certain $k$ or more complex like Banach spaces). Since continuous maps
between compact Hausdorff spaces are a particular example of left ultrafunctors, then the composition $\mathcal{F} \circ f $ gives a left ultrafunctor
$Y \to \mathcal{M}$, and this allows the construction of the category
$\mathsf{CompHaus}^{\mathsf{o}}_{\mathcal{M}}$ in which the objects are left ultrafunctors
from some compact Hausdorff space $X$ to $\mathcal{M}$ and as morphisms
between $ \mathcal{F} :X \rightarrow \mathcal{M}$ and
$ \mathcal{G}: Y \rightarrow \mathcal{M}$ consists of a pair
$(f,\alpha )$ where $f $ is a continuous map from $X$ to $Y$ and
$\alpha $ is a natural transformation of left ultrafunctors from
$\mathcal{F}$ to $\mathcal{G} \circ f$, and this category is fibred over
$\mathsf{CompHaus}$.

In section~\ref{7}, we show that the bundle over $Y$ resulting from the
composition $\mathcal{F} \circ f$ is the pullback along
$Y \xrightarrow[]{f} X$ in $\mathsf{Top}$, and we extended the equivalence
between $\mathrm{Bun}(\mathcal{M},X)$ and
$\mathrm{Lult}(X,\mathcal{M})$ to an equivalence between
$\mathrm{CompMet}_{X}$ and $\mathrm{Bun}$ the category of bundles of
$\mathcal{M}$ over any compact Hausdorff space.

\subsubsection*{Examples}

Section~\ref{8} is dedicated to showing examples where our notion of bundles
agrees with previously existing notions of bundles of metric structures
used in functional analysis. And we showed that this notion corresponds
exactly to our notion of bundles. We establish that the two slightly different
notions of bundles of Banach spaces  (namely semi-continuous
and continuous bundles), already in the literature (see
\cite{hofmann1977bundles} for semi-continuous bundles and
\cite{Fell1977} for continuous ones) already correspond to two slightly
different continuous theories of Banach spaces whose categories
of models are, respectively, Banach spaces with linear contractions and
Banach spaces with linear isometries. After that, we show that bundles of
$\mathrm{C}^{*}$ algebras \cite[Appendix C]{Williams2007CrossedPO}, also
called (semi)-continuous fields of $\mathrm{C}^{*}$-algebras, is the notion
of bundle that corresponds to the continuous theory of
$\mathrm{C}^{*}$-algebras, we also show that $\mathrm{W}^{*}$-bundles (see
\cite[section 5]{ozawa2013dixmier} or
\cite[subsection 3.1]{bosa2019covering} or
\cite[definition 2.1]{evington2016locally}) is the notion of bundles which
corresponds to the continuous model theory of tracial von Neumann algebras
\cite{goldbring2022survey}.

\subsubsection*{Giving an alternative proof of Lurie's result}

As mentioned above, in \cite{lurie2018ultracategories} Lurie shows that
for a compact Hausdorff space $X$, there is an equivalence of categories
between $\mathrm{Sh}(\mathcal{O}(X))$ and left ultrafunctors from
$X$ to $\mathsf{Set}$. On the other hand, it is known that there is an
equivalence between $\mathrm{Sh}(\mathcal{O}(X))$ and the category of
\'{e}tale bundles over $X$. The category of sets is equivalent to the category
of discrete metric spaces, which is axiomatisable using continuous
model theory. In section~\ref{9} we show that bundles of discrete metric
spaces are equivalent to \'{e}tale bundles, which allows us to write the
following chain of equivalences for any compact Hausdorff space:%
\begin{align*}
\mathrm{Sh}(\mathcal{O}(X)) &\simeq \{
\textbf{\'{e}tale bundles over } \ X \}\\
& \simeq \{
\textbf{bundles of discrete metric spaces over} \ X\} \\
& \simeq \{
\mathbf{Left \ ultrafunctors}(X,\mathsf{Set}) \}.
\end{align*}

This allows to reobtain the result shown by Lurie \cite[3.4.4]{lurie2018ultracategories},
while giving it an entirely different proof. Our construction relies on
the \'{e}tale space description of sheaves, while the one given by Lurie
uses, more or less, the functorial description.

\subsubsection*{A non example}

Finally in section~\ref{10}, we study a new notion of ``bundles of pointed
complete spaces over $X$'', and we show that when the space $X$ is compact Hausdorff such concept is equivalent to left ultrafunctors from $X$ to the category
of pointed complete metric spaces. The reason this does not fit the framework
of bundles of models is that we do not know if it is possible
to have a continuous first-order axiomatisation of complete pointed metric
spaces.

%s1 #&#
\section{Preliminary constructions}
%%LEAP%%%\label{sec1}
\label{1}

\subsection*{Definition of an ultracategory}

Following \cite{lurie2018ultracategories}
\label{definition_of_an_ultracategory_VTEX1}
%
%d1.1 #&#
\begin{definition}
\label{defn1.1}
An ultrastructure on a category $A$ consists of the following data:
\begin{enumerate}
\item For every set $X$, and every ultrafilter $\mu$ on $X$, a functor from $A^{X}$ to $A$, which we are going
to call the ultraproduct functor and we denote it by
\begin{equation*}
\int _{X} \bullet \ d \mu .
\end{equation*}
\item Given a set $X$, a family of ultrafilters on $X$
$(\nu _{s})_{s \in S}$ and an ultrafilter $\mu $ on $S$, we require the
existence of a morphism
$\Delta _{\mu ,\nu _{\bullet}}: \int _{X} M_{x} d(\int _{S} \nu _{s} d
\mu ) \rightarrow \int _{S} (\int _{X} M_{x} d \nu_{s} ) d \mu $, which is
natural in the family $(M_{x})_{x \in X}$. The map
$\Delta _{\mu ,\nu _{\bullet}}$ is called \emph{the categorical Fubini transform}.
\item For every principal ultrafilter $\delta _{x_{0}}$ on a set $X$, we
require a natural family of isomorphisms $\epsilon _{X,x_{0}}$ from
$\int _{X} M_{x} d \delta _{x_{0}}$ to $M_{x_{0}}$.
\end{enumerate}
This data is required to satisfy the following axioms:
\begin{enumerate}[label=\Alph{enumi}]
\item Given a family of ultrafilters $(\nu _{s})_{s
\in S}$ on a set $X$, and a family of objects of $A$,
$(M_{x})_{x \in X}$ then the map
$\Delta _{ \delta _{s_{0}},\nu _{\bullet}}: \  \int _{X} M_{x} d
\int \nu _{s} d \delta _{s_{0}} \rightarrow \int _{S} \int _{X} M_{x} d
\nu _{s} d \delta _{s_{0}}$, is the inverse of the map
$\epsilon _{S,s_{0}}$ from
$\int _{S} \int _{X} M_{x} d \nu _{s} d \delta _{s_{0}}$ to
$\int _{X} M_{x} d \nu _{s_{0}}$.
\item Suppose that we have a monomorphism of sets (injective
function) $f: \  X \rightarrow Y $ then the categorical Fubini transform
from
$\int _{Y} M_{y} df \mu =\int _{Y} M_{y} d \int _{X} \delta _{f(x)} d
\mu $ to $\int _{X} \int _{Y} M_{y} d \delta _{f(x)} d
\mu \simeq \int _{X}M_{f(x)}d\mu $ is an isomorphism.
\item Suppose that we have a set $R$ and an ultrafilter $\lambda $ on it,
and suppose we have $(\mu _{r})_{r \in R}$ a family of ultrafilters on
a set $S$, and $(\nu _{s})_{s \in S}$ is a family of ultrafilters on some
set $T$, then the following diagram commutes:
\begin{equation*}
\begin{tikzcd}
\int _{T} M_{t} d \rho
\arrow[rrr, "{\Delta _{\lambda , \int _{S} \nu _{s}d\mu _{\bullet}}}"]
\arrow[dd, "{\Delta _{\int _{R} \mu _{r} d \lambda ,\nu _{\bullet}}}"]
& & & \int _{R} (\int _{T} M_{t} d \int _{S} \nu _{s} d \mu _{r}) d
\lambda
\arrow[dd, "{\int _{R} \Delta _{\mu _{r},\nu _{\bullet}}d \lambda}"]
\\
& & &
\\
\int _{S} \int _{T} M_{t} d \nu _{s} d \int _{R} \mu _{r} d \lambda
\arrow[rrr, "{\Delta _{\lambda ,\mu _{\bullet}}}"] & & & \int _{R} (
\int _{S} (\int _{T} M_{t} d \nu _{s})d \mu _{r}) d \lambda
\end{tikzcd}
\end{equation*}
where
$\rho =\int _{R} (\int _{S} \nu _{s} d\mu _{r}) d\lambda  =\int _{S}
\nu _{s} d (\int _{R} \mu _{r} d \lambda)$ (here
$\int _{S} \nu _{s} d \mu $ is defined by \newline
$B \in \int _{S} \nu _{s} d \mu \iff \{s \in S \mid \ B \in \nu _{s}
\} \in \mu $).
\end{enumerate}
\end{definition}

Now we define an ultracategory to be a category with an ultrastructure.

\subsection*{Left ultrafunctors}

Suppose that $\mathsf{{M}}$ and $\mathsf{{N}}$ are two ultracategories, we define a
left ultrafunctor from $\mathsf{M}$ to $\mathsf{N}$ to be a functor equipped
with a left ultrastructure.

%d1.2 #&#
\begin{definition}%
\label{defn1.2}
A left ultrastructure on a functor consists of the following: for every
ultrafilter $\mu $ on a set $X$ and every family of objects
$(M_{x})_{x \in X}$ of $\mathsf{{M}}$, we have a family of morphisms in
$\mathsf{{N}}$-we call all of them $\sigma _{\mu}$ by abuse of language-from
$F(\int _{X} M_{x} d\mu )$ to $\int _{X} F(M_{x})d \mu $.

Such that they satisfy the following axioms:
\begin{enumerate}
\item[0.] The following diagram commutes for every family of morphisms
$(\psi _{x})_{x \in X}$ from $M_{x}$ to $N_{x}$ in $\mathsf{{M}}$:
\begin{equation*}
\begin{tikzcd}
F(\int _{X} M_{x} d \mu ) \arrow[rr, "\sigma _{\mu}"]
\arrow[dd, "F(\int _{X} \psi _{x} d \mu )"] & & \int _{X}F(M_{x})d
\mu \arrow[dd, "\int _{X}F(\psi _{x}) d \mu "]
\\
& &
\\
F(\int _{X} N_{x} d \mu ) \arrow[rr, "\sigma _{\mu}"] & & \int _{X}F(N_{x})d
\mu
\end{tikzcd}
\end{equation*}
\item[1.] For every principal ultrafilter $\delta _{x_{0}}$ the following diagram
commutes:
\begin{equation*}
\begin{tikzcd}
F(\int _{X} M_{x} d \delta _{x_{0}}) \arrow[rr, "\sigma _{\delta_{x_{0}}}"]
\arrow[rd, "{F(\epsilon _{X,x_{0}})}"] & & \int _{X} F(M_{x}) d
\delta _{{x_{0}}} \arrow[ld, "{\epsilon _{X,x_{0}}}"]
\\
& F(M_{x_{0}}) &
\end{tikzcd}
\end{equation*}
\item[2.] For any sets $S$ and $T$, any ultrafilter $\mu $ on $S$ and any family
of ultrafilters $(\nu _{s})_{s \in S}$ on $T$ indexed by $S$, the following
diagram commutes:
\[
\begin{tikzcd}
F(\int _{T} M_{t} d (\int _{S} \nu _{s} d \mu ))
\arrow[rrrr, "\sigma _{\int _{S}\nu _{s} d \mu}"]
\arrow[dd, "{F(\Delta _{\mu ,\nu _{\bullet}})}"] & & & & \int _{T} F(M_{t})
d \int _{S} \nu _{s} d\mu
\arrow[dd, "{\Delta _{\mu ,\nu _{\bullet}}}"]
\\
& & & &
\\
F(\int _{S} (\int _{T} M_{t} d\nu _{s}) d\mu )
\arrow[rr, "\sigma _{\mu}"] & & \int _{S}F(\int _{T}M_{t}d\nu _{s})d
\mu \arrow[rr, "\int _{S} \sigma _{\nu _{s}} d \mu "] & & \int _{S}
\int _{T} F(M_{t}) d \nu _{s} d \mu
\end{tikzcd}
\]%
\end{enumerate}
\end{definition}

\begin{note*}
\normalfont The dual notion is a right ultrafunctor in which the comparison
maps go in the other direction, we omit writing the axioms which can be
found in~\cite{lurie2018ultracategories}.
\end{note*}
%
%d1.3 #&#
\begin{definition}
\label{defn1.3}
An ultrafunctor is a left ultrafunctor for which all the comparison maps
are isomorphisms.
\end{definition}

\subsection*{Natural transformations of left ultrafunctors}

Suppose that $\mathsf{{M}}$ and $\mathsf{{N}}$ are two ultracategories, and let
$F$, $G$ be left ultrafunctors between $\mathsf{{M}}$ and $\mathsf{{N} }$, a natural
transformation of left ultrafunctors from $F$ to $G$, is a natural transformation
$\phi $ satisfying the additional condition:
\noindent
For every family $(M_{i})$ of objects in $\mathsf{{M}}$ and for every ultrafilter
$\mu $ on $I$ the following diagram commutes:
\begin{equation*}
\begin{tikzcd}
{F(\int _{I} M_{i} d\mu )} &&& {\int _{I}F(M_{i})d\mu}
\\
\\
{G(\int _{I} M_{i} d\mu )} &&& {\int _{I}G(M_{i})d\mu}
\arrow["{\phi _{\int _{I}M_{i}d\mu}}", from=1-1, to=3-1]
\arrow["{\sigma _{\mu}}"', from=1-1, to=1-4]
\arrow["{\sigma ^{\prime}_{\mu}}", from=3-1, to=3-4]
\arrow["{\int _{I}\phi _{M_{i}}d\mu}"', from=1-4, to=3-4]
\end{tikzcd}
\end{equation*}
A natural transformation of right ultrafunctors is defined similarly.

%s1.1 #&#
\subsection{Some ultracategories
constructions}
\label{sec1.1}

\subsubsection*{Ultrasets}

You may have noticed that at this point we are using the notation
$\int \nu _{s} d\mu $ to denote the ultrafilter defined by
$A \in \int \nu _{s} d\mu \ \text{iff} \ \{s \mid \ A \in \nu _{s}\}
\in \mu $. This notation is not a coincidence, as this is a special case
of ultracategories.

%d1.4 #&#
\begin{definition}
\label{defn1.4}
An ultraset is a small ultracategory with no non-identity morphisms.
\end{definition}
Now the next theorem is due to Lurie~\cite[theorem 3.1.5]{lurie2018ultracategories}
%
%t1.1 #&#
\begin{theorem}
\label{thm1.1}
There is an equivalence of categories between ultrasets (with either left
ultrafunctors or ultrafunctors, they are the same in this case), and the
category of compact Hausdorff spaces with continuous maps.
\end{theorem}
Suppose that $S$ is a compact Hausdorff space and let
$(a_{x})_{x \in X}$ be a family of points of $S$ indexed by $X$ and $\mu $ an ultrafilter on
$X$. Then this equivalence is exhibited by defining
$\int _{X} a_{x} d \mu $ as the unique limit of the pushforward of the ultrafilter
$\mu $ by the map $x \mapsto a_{x}$.

Now, a particular case of the former is $\beta X$ the set of ultrafilters
on $X$ which is a Stone Space (has a totally separated compact Hausdorff
topology), thus an ultraset, which justifies the notation
$\int \nu _{s} d\mu $.

 Before continuing we show a very useful result regarding
ultracategories that we will be using later:

%l1.1 #&#
\begin{lemma}
%%LEAP%%%\label{lem1.1}
\label{ultresult}
Suppose that we have a map of sets $f$ from $Y$ to $X$, where $X$ is a
compact Hausdorff space (seen as an ultraset), and suppose that we
have a left ultrafunctor $\mathcal{F}$ from $X$ to $\mathcal{M}$, here
$\mathcal{M}$ is an arbitrary ultracategory. Let $\mu $ be an ultrafilter
on $Y$ such that $f\mu$ converges to $x_{0}$. Then the following diagram commutes:
\begin{equation*}
\begin{tikzcd}
\mathcal{F}(x_{0})=\mathcal{F}(\int _{X} x d \int _{Y}\delta _{f(y)}d
\mu )=\mathcal{F}(\int _{Y} f(y) d\mu )
\arrow[rrrrr, "\sigma _{f \mu}"] \arrow[rrrrrdd, "\sigma _{\mu}"] & & & & &
\int _{X} \mathcal{F}(x) df\mu \arrow[dd, "{\Delta _{\mu , f}}"]
\\
& & & & &
\\
& & & & & \int _{Y} \mathcal{F}(f(y))d\mu
\end{tikzcd}
\end{equation*}
\end{lemma}

\begin{proof}
We use the following diagram:
\[
\fontsize{8}{9}\selectfont
\begin{tikzcd}[column sep=20,row sep=30]
{\mathcal{F}(\int _{X} x d \int _{Y}\delta _{f(y)}d\mu )} &&&&&& {
\int _{X} \mathcal{F}(x) df\mu}
\\
\\
{\mathcal{F}(\int _{Y} \int _{X} xd \delta _{f(y)}d\mu )=\mathcal{F}(
\int _{Y} f(y) d\mu )} && {\int _{Y}\mathcal{F}(\int _{X}xd\delta _{f(y)})d
\mu} &&&& {\int _{Y} \int _{X} \mathcal{F}(x)d \delta _{f(y)} d \mu}
\\
&& {\mathcal{}}
\\
\\
&&&&&& {\int _{Y} \mathcal{F}(f(y))d\mu}
\arrow["{\sigma _{f \mu}}"', from=1-1, to=1-7]
\arrow["{\mathrm{id}}", from=1-1, to=3-1]
\arrow["{\Delta _{\delta _{f(\bullet )},\mu}}"', from=1-7, to=3-7]
\arrow["{\Delta _{f,\mu}}", shift left=5, curve={height=-35pt}, from=1-7, to=6-7]
\arrow["{\sigma _{\mu}}", from=3-1, to=3-3]
\arrow["{\int _{Y} \sigma _{\delta _{f(y)}} d\mu}"{description}, from=3-3, to=3-7]
\arrow["{\int _{Y} \mathcal{F}(\epsilon _{X,f(y)}) d\mu =\mathrm{id}}"', from=3-3, to=6-7]
\arrow["{\int _{Y} \epsilon _{X,f(y)} d \mu}"', from=3-7, to=6-7]
\end{tikzcd}
\]
The upper diagram commutes by axiom $(2)$ of
\cite[definition 1.4.1]{lurie2018ultracategories}, and the lower diagram
commutes by axiom $(1)$ of
\cite[definition 1.4.1]{lurie2018ultracategories}, and hence, the outermost
diagram commutes which is exactly what we wanted to show.
\end{proof}

In the case of categories of models of continuous logic (we are going to
introduce these later); if $g \in \mathcal{F}(x_0) $, then
$\sigma _{f \mu}(g)=(b_{x})_{x \in X}$ implies that
$\sigma _{\mu}(g)=(b_{f(y)})_{y \in Y}$.

\medskip
\noindent\textbf{Ultrasets corresponding to compact subspaces of $\mathbb{R}$}
One particular case of compact Hausdorff space is compact subsets of the
real line $\mathbb{R}$, in this section, we are going to give a nice characterisation
of the ultraproduct functor for such sets, which will come in handy when
studying the ultraproduct of metric spaces.

Let $X$ be a set and let $\mu $ be an ultrafilter on $X$ and suppose that
$\phi $ is a function taking values in $M$ where $M$ is a compact subset
of $\mathbb{R}$ (we can take $M$ for simplicity to be a closed interval).
Now take the ultraproduct $\int _{X}\phi (x)d\mu $. This is the limit of
the ultrafilter $\phi \mu $ (the pushforward of $\mu $ by $\phi $), which
translates to the fact that $\phi \mu $ contains the neighbourhood filter
of $\int _{X}\phi (x)d\mu $. In other words, for arbitrarily small
$\epsilon $ the set
$\{z \in X \mid  |\phi (z)-\int _{X}\phi (x)d\mu |<\epsilon \} \in \mu $ (such
ultraproduct is what is usually referred to as an ultralimit and usually
denoted by $\lim _{\mu}\phi (x)$).

Now, we claim the following:
\begin{lemma*}
$\int _{X}\phi (x)d\mu =\mathrm{Inf}_{U \in \mu}\mathrm{Sup}_{x \in U}
\phi (x)$.
\end{lemma*}
\begin{proof}%
Let us call $m=\int _{X}\phi (x)d\mu $. First, let us prove that $m$ is
a lower bound for
$\{ \mathrm{Sup}_{x \in U}\phi (x) \mid U \in \mu \}$. To do this, suppose
by contradiction that there exists some $U \in \mu $ such that
$m>\mathrm{Sup}_{x \in U}\phi (x)$, let us call
$\epsilon =m-\mathrm{Sup}_{x \in U}\phi (x)$, then the set
$\{z \in X \mid |\phi (z) -m|<\epsilon \} \in \mu $ but
$\{z \in X \mid |\phi (z) -m|<\epsilon \}\subseteq \{z \in X \mid m-
\epsilon <\phi (z) \}$. But this would imply that
$V= \{z \in X \mid \mathrm{Sup}_{x \in U}\phi (x)<\phi (z) \} \in
\mu $, but $V \bigcap U =\varnothing $ on one hand, and on the
other hand $V \bigcap U \in \mu $, hence a contradiction ($
\varnothing \in \mu $).%
\newline So $m$ is a lower bound for
$\{ \mathrm{Sup}_{x \in U}\phi (x) \mid U \in \mu \}$. To prove it is the greatest
lower bound, notice that for any $\epsilon >0$ the set
$ V_{\epsilon}=\{x \in X \mid \phi (x) <m+ \epsilon \}\in \mu $ thus
$\mathrm{Sup}_{x \in V_{\epsilon}} \phi (x) \leq m+\epsilon $ so
$\mathrm{Inf}_{U \in \mu}\mathrm{Sup}_{x \in U}\phi (x) \leq m +
\epsilon $, and since $\epsilon $ was arbitrary then we get that
$\mathrm{Inf}_{U \in \mu}\mathrm{Sup}_{x \in U}\phi (x)\leq m $, and thus
since $m$ is a lower bound, we get that
$m=\mathrm{Inf}_{U \in \mu}\mathrm{Sup}_{x \in U}\phi (x)$.%
\end{proof}

\begin{note*}
The dual statement $\int _{X}\phi (x)d\mu =\mathrm{Sup}_{U \in \mu}\mathrm{Inf}_{x \in U}
\phi (x)$ can be shown similarly.
\end{note*}

\subsubsection*{Ultracategories arising from directed colimits}

%t1.2 #&#
\begin{theorem}
%%LEAP%%%\label{thm1.2}
\label{categorical_ultraproduct_VTEX1}
Suppose that we have a category $\mathcal{M} $ that has products
and directed colimits, then in this case it has an ultrastructure given
by:
\begin{equation*}
\int _{X} M_{x} d\mu =\varinjlim _{ U \in \mu}(\prod _{x \in U}M_{x}).
\end{equation*}
\end{theorem}
Here we consider the set of sets of $\mu $ as a directed
set by reverse inclusion.
%
%l1.2 #&#
\begin{lemma}
%%LEAP%%%\label{lem1.2}
\label{categorical_ultraproduct_lemma_VTEX1}
Suppose that $\mathcal{M}$ is a full subcategory of an ultracategory
$\mathcal{M}^{+}$ which is closed under the ultraproduct functor, then
it is an ultracategory with such restriction of the ultraproduct functor.
\end{lemma}

Theorem~\ref{categorical_ultraproduct_VTEX1} and Lemma~\ref{categorical_ultraproduct_lemma_VTEX1} are just restating proposition 1.3.7
of \cite{lurie2018ultracategories}, and a proof can be found there.

The Lemma~\ref{categorical_ultraproduct_lemma_VTEX1} allows us not only to consider
categories having directed colimits and products, but full subcategories
of those closed under the ultraproduct construction given by such directed
colimit of products. The main example of such ultracategories is the ultracategory
of models of a first-order theory, which is a full subcategory of the category
of structures of the same signature (similarity type).

We highlight this construction in the case of $\mathsf{Set}$, and this extends to
all first-order theories. The construction is an application of \ref{categorical_ultraproduct_VTEX1}:%
\newline First, we define the ultraproduct of non-empty sets by
\begin{equation*}
\int _{I} M_{i} d\mu = \prod _{I} M_{i}/\sim .
\end{equation*}
Here $\sim $ identifies tuples that agree on any set of the ultrafilter,
and you can notice that this is just a direct limit of products in
$\mathsf{{Set}}$. Now in the case where some sets of $(M_{i})$ are empty we have
two cases, either the set
$\{i \in I \mid M_{i} =\varnothing \} \in \mu $, in this case we define
$\int _{I} M_{i} d\mu =\varnothing $, otherwise the set
$I^{\prime}= \{ i \in I \mid M_{i} \neq \varnothing \} \in \mu $ so we define%
\begin{equation*}
\int _{I} M_{i} d\mu = \prod _{I^{\prime}} M_{i}/\sim .
\end{equation*}
\noindent
In other words, we restrict our attention to a set of the ultrafilter for
which the sets $M_{i}$ are non-empty. So in what follows,
we are going to denote the elements of the ultraproduct
$\int _{I} M_{i} d \mu $ by $(a_{i})_{i \in J}$ where $J \in \mu $.

Now suppose that we have a first-order theory with signature
$\langle \mathcal{S}_{1},\dots ,\mathcal{S}_{n},\mathcal{R}_{1},
\dots ,\mathcal{R}_{n'},f_{1},\dots ,f_{n''} \rangle $ and a set of axioms
$\mathscr{A}$. The category of structures has ultraproducts, resulting
from applying Theorem~\ref{categorical_ultraproduct_VTEX1} and Lemma~\ref{categorical_ultraproduct_lemma_VTEX1}, which are constructed as follows:
Suppose we have structures $(E_{i})_{i \in I}$, in what follows we are
going to denote by $E_{i}^{\mathcal{S}_{j}}$ the set of sort
$\mathcal{S}_{j}$ corresponding to $E_{i}$. Now let
$\mu $ be some ultrafilter on $I$ and define
$\int _{I} E_{i} d \mu $ as follows: For each sort $\mathcal{S}_{j}$,
$(\int _{I} E_{i} d \mu )^{\mathcal{S}_{j}}=\int _{I} E_{i}^{\mathcal{S}_{j}}
d \mu $.

Now for a relation symbol $\mathcal{R}$ with formal domain
$S_{1} \times \dots \times S_{l} $, here each
$S_{m} \in \{\mathcal{S}_{j}\}_{j=1}^{n}$, we define
$ \mathcal{R}^{\int _{I} E_{i} d\mu}$ by
$((a^{m}_{i})_{i \in I})_{1 \leq m \leq l}\in \mathcal{R}^{\int _{I} E_{i} d\mu}$
iff $\{ i \in I \ : (a^{m}_{i})_{1 \leq m \leq l} \in \mathcal{R}^{E_{i}} \} \in \mu $.

Next, for a function symbol $f$ with formal domain $S_{1}
\times \dots \times S_{l} $ and formal range $S^{\prime}$, we define
$f^{\int _{I} E_{i} d \mu }((a^{m}_{i}) _{i \in I})_{1 \leq m \leq l}$ by
$(f^{E_i}(a^{m}_{i})_{1 \leq m \leq l})_{i \in I}$. 

Then we can regard the category of models in $\mathsf{{Set}}$ of
$\mathscr{A}$ as a full subcategory of the category of structures of similarity
type
$\langle \mathcal{S}_{1},\dots ,\mathcal{S}_{n},\mathcal{R}_{1},
\dots ,\mathcal{R}_{n'},f_{1},\dots ,f_{n''} \rangle $.%
\newline As a result of {\L}os theorem, this subcategory is closed under the categorical
ultraproduct of the category of structures, which allows the application
of Lemma~\ref{categorical_ultraproduct_lemma_VTEX1}. Notice that we assumed finitely many sorts, function, and relation symbols but this is not necessary at all.

%s2 #&#
\section{The ultracategory $\mathsf{k}\text{-}\mathrm{CompMet}$}
%%LEAP%%%\label{sec2}
\label{2}

Given $k$ a positive real number, we denote by
$\mathsf{k\text{-}CompMet}$ the category of $k$-bounded complete metric
spaces, with contractions ($1$-Lipschitz functions) as morphisms. More
precisely, the objects are the complete metric spaces satisfying
$d(x,y) \leqslant k $ for all $x,y$, and the morphisms are functions satisfying
$d(f(x),f(y)) \leq d(x,y)$ for all $x,y$.

%p2.1 #&#
\begin{proposition}
\label{prop2.1}
The category $\mathsf{k\text{-}CompMet}$ has all products and all directed
colimits.
\end{proposition}

We give a sketch of this typical construction:

The product of a family $(B_{i})_{i \in I} $ of $k$-bounded complete metric
spaces, is computed by taking the products of their underlying sets, and
equipping it with the distance:
\begin{equation*}
d((b_{i}),(c_{i}))=\mathrm{Sup}_{I}(d(b_{i},c_{i})).
\end{equation*}

It should be noted that if we were working with unbounded metric spaces,
without allowing for the possibility that $d(x,y)$ can be infinite, then
this construction would not work and the resulting category would not have
all products. This is the main reason why we work with this specific category
$\mathsf{k\text{-}CompMet}$.

For directed colimits,  we first compute the colimits inside
the category $\mathsf{k\text{-}PsMet}$ of $k$-bounded pseudo-metric spaces
with contractions as morphisms. That is, we remove the requirement that
$d(x,y)=0 \Rightarrow x = y$, as well as the completeness requirement.

The category $\mathsf{k\text{-}CompMet}$ is reflective in
$\mathsf{k\text{-}PsMet}$: to each $k$-bounded pseudo-metric space, one
can associate a metric space by quotienting it by the relation
$x \sim y$ if $d(x,y)=0$, and take the completion of the resulting metric
space. Hence colimits in $\mathsf{k\text{-}CompMet}$ can be obtained by
first taking the colimit in the category $\mathsf{k\text{-}PsMet}$ and
then applying this quotient-completion construction (left
adjoint).

Finally, directed colimits in
$\mathsf{k\text{-}PsMet}$ are computed as follows: Let $I$ be a directed
set viewed as a category and let $B$ be a functor from this directed set
to the category of $k$ bounded pseudo-metric spaces. Then one first takes
the colimit of the underlying sets:
$\varinjlim _{I}B_{i} \simeq \coprod _{i \in I}B_{i}/\approx $ where
$\coprod $ denotes the disjoint union and the equivalence relation is the
relation generated by: $f \approx g$ iff if $f \in B_{i_{1}}$ and
$g \in B_{i_{2}}$ and $i_{1} \leq i_{2}$, then
$\epsilon _{i_{1},i_{2}}(f)=g$ (here $\epsilon _{i_{1},i_{2}}$ is the image
by the functor $B$ of the morphism between $i_{1}$ and $i_{2}$ in the directed
set viewed as a category). And we equip it with the following pseudo-metric:
if $f\in B_{i}$ and $g \in B_{j}$ then
$d(f,g)=\mathrm{Inf}_{i,j\leq l}d_{l}(f,g)$. In particular, it should be
noted that if $f \in B_{i}$ and $g \in B_{j}$, then in the
colimits in $\mathsf{k\text{-}CompMet}$, we have that $f=g$ iff
$\forall \epsilon >0$, there exists $l \geqslant i,j$ such that
$d_{l}(f,g) < \epsilon $.

We can deduce from this the following by virtue of Theorem~\ref{categorical_ultraproduct_VTEX1}:

%p2.2 #&#
\begin{proposition}
\label{prop2.2}
The category $\mathsf{k\text{-}CompMet}$ has an ultrastructure, where the
ultraproduct functors are given by:
\begin{equation*}
\int _{S}M_{s} d\mu =\varinjlim _{U \in \mu} \left ( \prod _{s \in U}M_{s}
\right ).
\end{equation*}
where $\mu $ is seen as a category with an arrow $A \to B$ if
$B \subseteq A$.
\end{proposition}

We can however give a slightly more explicit description of this ultraproduct
construction. We fix $S$ a set and $\mu $ an ultrafilter on $S$; suppose
we have a family of non-empty complete $k$-bounded metric spaces
$(M_{s})_{s \in S}$ (that is an object in
$\mathsf{k\text{-}CompMet}^{S}$).

We endow the set-theoretic product $\prod _{s \in S}M_{s}$ with the equivalence
relation defined by $(f_{s}) \sim (g_{s})$ iff for every
$\epsilon >0$ the set
$\{s \in S \mid \ d_{s}(f_{s},g_{s})<\epsilon \} \in \mu $, and the distance
given by:
\begin{equation*}
d((f_{s})_{s \in S},(g_{s})_{s \in S})= \lim _{\mu} d(f_{s},g_{s}) =
\operatorname{\mathrm{Inf}}_{ M \in \mu}\mathrm{Sup} _{s \in M}d(f_{s},g_{s}).
\end{equation*}

%t2.1 #&#
\begin{theorem}%
\label{thm2.1}
The distance defined above makes ($\prod _{s \in S}M_{s}/\sim $) a complete
metric space, which identifies up to canonical isometry with the ultraproduct
$\varinjlim _{U \in \mu} \left ( \prod _{s \in U}M_{s} \right )$.
\end{theorem}

For a proof see \cite[Ultraproduct of metric spaces]{yaacov2008model}.

\begin{note*}
\normalfont We should be more precise that the construction above would
work if the family of metric spaces $(M_{s})_{s \in S}$ are all non-empty,
if some $M_{s}$ are empty we can do the same trick as in the case of
$\mathsf{Set}$ and looking whether
$S^{\prime}=\{s \in S \mid M_{s} = \varnothing \} $ is in the ultrafilter or
not.

Although we are going to write proofs assuming that no metric space is
empty, this trick can always be used so our proofs also encompass the case
where some metric spaces are allowed to be empty.
\end{note*}

%s2.1 #&#
\subsection{Description of the categorical Fubini transform $\Delta $ in the category $\mathsf{k\text{-}CompMet}$}
%%LEAP%%%\label{sec2.1}
\label{description}

%t2.2 #&#
\begin{theorem}
\label{thm2.2}
Let $\mathsf{k\text{-}CompMet}$ denote the category of complete metric
spaces bounded by a certain $k$ with contractions as morphisms, and let
$S, T$ be sets and let $\nu _{\bullet}=(\nu _{s})_{s \in S}$ be a collection
of ultrafilters on $T$ and let $(M_{t})_{t \in T} $ be a collection of
complete metric spaces indexed by $T$, then we have:
\begin{equation*}
\Delta _{\mu ,\nu _{\bullet}}((b_{t})_{t\in T})=((b_{t})_{t\in T})_{s \in S}.
\end{equation*}
\end{theorem}
\begin{proof}
We need to make sure that the map
$ (b_{t})_{t \in T} \mapsto ((b_{t})_{t \in T})_{s \in S} $ is well-defined:
Suppose that
\begin{equation*}
(b_{t})_{t\in T}=(b^{\prime}_{t})_{t\in T}.
\end{equation*}
Now we know that for any $\epsilon >0$ the set
\begin{equation*}
\{ t \in T \mid d_{t}(b_{t},b^{\prime}_{t})<\epsilon \} \in \int _{S}
\nu _{s} d\mu ,
\end{equation*}
which translates to the fact that for any $\epsilon >0$ the set
\begin{equation*}
\{s \in S \mid \{ t \in T \mid d_{t}(b_{t},b^{\prime}_{t})<\epsilon \}
\in \nu _{s} \}\in \mu ,
\end{equation*}
so, we get that for any $\epsilon >0$%
\begin{equation*}
\{s \in S \mid \ d_{\nu_s}((b_{t})_{t\in T},(b^{\prime}_{t})_{t\in T})\leq
\epsilon \} \in \mu ,
\end{equation*}
so
\begin{equation*}
((b^{\prime}_{t})_{t\in T})_{s \in S}=((b_{t})_{t\in T})_{s \in S}.
\end{equation*}

So the map defined the way above is well-defined, and we can see that it
is a contraction.

Now the map $\Delta _{\mu ,\nu _{\bullet}}$ for the ultrastructure on
$\mathsf{k\text{-}CompMet}$ is the unique map that makes the following
diagram commute for every set $S_{0}\subseteq S$ satisfying
$S_{0} \in \mu $ and every set $T_{0} \subseteq T$ such that
$T_{0} \in \nu _{s} \ \forall s \in S$
\cite[proposition 1.2.8]{lurie2018ultracategories}
\begin{equation*}
\begin{tikzcd}
{\prod _{t \in T_{0}}M_{t}} && {\prod _{s \in S_{0}}\int _{T}M_{t}d\nu _{s}}
\\
\\
{\int _{T}M_{t}d(\int _{S}\nu _{s} d\mu )} && {\int _{S}(\int _{T}M_{t}d
\nu _{s})d\mu}
\arrow["{(q^{T_{0}}_{\nu _{s}})_{s \in S}}", from=1-1, to=1-3]
\arrow["{q^{T_{0}}_{\int _{S}\nu _{s}d\mu}}"', from=1-1, to=3-1]
\arrow["{\Delta _{\mu ,\nu _{\bullet}}}"', from=3-1, to=3-3]
\arrow["{q^{S_{0}}_{\mu}}", from=1-3, to=3-3]
\end{tikzcd}
\end{equation*}

Now clearly the map
$ (b_{t})_{t \in T} \mapsto ((b_{t})_{t \in T})_{s \in S} $ makes the diagram
above commutative, for every set $S_{0}\subseteq S$ satisfying
$S_{0} \in \mu $, and every set $T_{0} \subseteq T$ such that
$T_{0} \in \nu _{s}$ for all $s \in S$.
\end{proof}

One particular case which is important to consider is when we have a map
of sets $p$ from $S$ to $T$ and then we consider the family
$\delta _{\bullet}=(\delta _{p(s)})_{s \in S}$,  which is
the family of all the principal ultrafilters of the points in the image
of $p$. Then in this case we get the map
\begin{equation}
\Delta _{\mu ,\delta _{\bullet}}((b_{t})_{t \in T})=(b_{p(s)})_{s \in S}.
\label{eq:1}
\end{equation}
Here, $\int _{T}M_{t} d{\delta _{p(s)}}$ was identified with
$M_{p(s)}$ (more precisely, without this identification, the above is a
description of the ultraproduct diagonal map as defined in
\cite[Notation 1.3.3]{lurie2018ultracategories}).

%s3 #&#
\section{The bundle (the first functor)}
%%LEAP%%%\label{sec3}
\label{3}

In this section, we are going to define the category of bundles of complete
metric spaces bounded by some constant $k$ over some compact Hausdorff
space $X$, which we are going to denote by
\newline $\mathsf{Bun}(\mathsf{k\text{-}CompMet},X)$ or alternatively
$\mathsf{Bun}(\mathsf{k\text{-}CompMet})/X$, and construct a functor from
the category
\newline $\text{Lult}(X,\mathsf{k\text{-}CompMet})$, to the category
$\mathsf{Bun}(\mathsf{k\text{-}CompMet},X)$.

%s3.1 #&#
\subsection{Bundles of complete metric spaces}
\label{sec3.1}

%d3.1 #&#
\begin{definition}
\label{defn3.1}
A function $f$ from a topological space $E$ to
$\mathbb{R}\bigcup \{-\infty ,\infty \}$ is said to be upper semi-continuous
(respectively lower semi-continuous) at a point $a$ iff
for every $y>f(a)$ (respectively $y<f(a)$) there exists a neighbourhood
$V$ of $a$ such that for every $x \in V \ f(x)<y$ (respectively
$f(x)>y$).

A function $f$ from a topological space $E$ with values in
$\mathbb{R} \bigcup \{-\infty ,\infty \}$ is upper semi-continuous (respectively
lower semi-continuous) iff it is upper semi-continuous (respectively lower
semi-continuous) at every point of its domain.
\end{definition}

\begin{note*}
\normalfont It is easy to see that being upper semi-continuous is equivalent
to being continuous when equipping $[-\infty ,+\infty ]$ with the topology
generated by $\{ [-\infty ,b), \ b \in (-\infty ,\infty ]  \}$ which is
called the left order topology. The subspace topology of the left order
topology of $[0,+\infty ]$ is generated by sets of the form $[0,b)$ with
$b \in (0,+\infty ]$.
\end{note*}
%
%d3.2 #&#
\begin{definition}
%%LEAP%%%\label{defn3.2}
\label{V}
Let $E$ be a topological space and let $\pi $ be a surjection from
$E$ to $X$, such that for each $x \in X$ $\pi ^{-1}(x)$ is a metric space
with distance $d_{x}$, and let $V$ be an open set then we
define
\begin{equation*}
V_{\epsilon}=\{ f \in E \mid \ \exists g \in V \ \pi (f) =\pi (g) \
\text{and} \ d_{\pi (f)}(f,g)<\epsilon \}.
\end{equation*}
\end{definition}

%d3.3 #&#
\begin{definition}
\label{defn3.3}
In the same context as Definition~\ref{V}, let $V,W$ be
open sets in $E$. We say that $V \subseteq _{\epsilon} W $ if  $V \subseteq V_{\epsilon} \subseteq W$.
\end{definition}

Let $A,B,C$ be topological spaces and let $f$ (respectively $g$) be a continuous
map from $A$ to $C$ (respectively from $B$ to $C$). We define the fibre
product space $A \times _{C} B$ to be the space
$\{ \ (x,y) \in A \times B \ \mid f(x)=g(y) \ \}$ with the subspace topology
of $A \times B$, this space is the pullback of $f,g$ in the category of
topological spaces.

Now we need to give an adequate definition of a continuous family of complete
metric spaces bounded by some constant $k$:
%
%d3.4 #&#
\begin{definition}
%%LEAP%%%\label{defn3.4}
\label{bundle_definition_VTEX1}
A bundle of complete metric spaces bounded by $k$ is defined
to be a triple $(E,X,\pi ) $ with $ \pi \  : E \rightarrow X$ a surjection
such that for every $x \in X$ $\pi ^{-1}(x)$ is a complete metric space
bounded by $k$, if it satisfies the axioms:
\begin{itemize}
\item Axiom(1): The global distance function defined from
$E \times _{X} E$ to $[0,k]$ is upper semi-continuous.
\item Axiom(2): $\pi $ is continuous and open.
\item Axiom(3): For every open set $W$ and every $f \in W$, there exists
an open neighbourhood $V$ of $f$ and $\epsilon >0$ such that
$V \subseteq _{\epsilon} W$.
\end{itemize}
\end{definition}

%d3.5 #&#
\begin{definition}
\label{defn3.5}
If $(E,X,\pi )$ and $(E^{\prime},X,\pi ^{\prime})$ are two bundles with base space
$X$, we define a map of bundles $\psi $ to be a continuous map from
$E$ to $E^{\prime}$ such that the following diagram commutes:
\begin{equation*}
\begin{tikzcd}
E && {E^{\prime}}
\\
& X \arrow["\pi "', from=1-1, to=2-2]
\arrow["\psi ", from=1-1, to=1-3]
\arrow["{\pi ^{\prime}}", from=1-3, to=2-2]
\end{tikzcd}
\end{equation*}
and such that for each $x \in X$ the map $\psi |_{\pi ^{-1}(x)}$ is a contraction.
\end{definition}
This makes bundles with a fixed base space a category. The case where the
base space is allowed to vary will be treated in section~\ref{7}.

%s3.2 #&#
\subsection{The Bundle's topology}
\label{sec3.2}

Given a compact Hausdorff space $X$, and a left ultrafunctor
$\mathcal{F}$ from $X$ to $\mathsf{k}\text{-}\mathrm{CompMet}$ we want to
endow the space $\coprod _{x \in X}\mathcal{F}(x)$ with a canonical topology
making it a bundle as in our Definition~\ref{bundle_definition_VTEX1}. A common
idea usually used in the definition of bundles, is that the bundle space
is some sort of section space (or germs of section space) to the projection
map, the definition that we gave starts from the realisation that the image
by the left-ultrastructure maps of a point in the base space can be regarded
as some sort of ``generalised'' local section maps at this point, and hence
one can use these to define a topology on the space
$\coprod _{x \in X}\mathcal{F}(x)$ similar in spirit to
\cite[13.18]{fell1988representations} (constructing a bundle from a family
of sections is abundant in functional analysis literature, the same kind
of idea can be seen for example in \cite{evington2016locally} or
\cite{Niels,Williams2007CrossedPO}).%
\label{tau}

Let us call $\mathcal{L}$ the assignment that we are going to define, which
gives a bundle for each left ultrafunctor from a compact Hausdorff set
$X$ (ultraset) to $\mathsf{k}\text{-}\mathrm{CompMet}$.
%
%t3.1 #&#
\begin{theorem}
%%LEAP%%%\label{thm3.1}
\label{topology}
Let $X$ be an ultraset, and let $\mathcal{F}$ be a left ultrafunctor from
$X$ to $\mathsf{k}\text{-}\mathrm{CompMet}$, let
$E=\coprod _{x \in X}\mathcal{F}(x)$ then there is a unique topology
$\tau $ on $E=\coprod _{x \in X}\mathcal{F}(x)$ such that an ultrafilter
$\eta $ converges to a point $f \in E$ iff:
\begin{enumerate}[leftmargin=2em]
\item[$C_{1}$] : $\pi \eta $ converges to $\pi f$.
\item[$C_{2}$] : for any $\epsilon > 0$ if
$\sigma _{\pi \eta }(f)=(b_{x})_{x \in X}$ then
$\coprod _{x \in X}B(b_{x},\epsilon )\in \eta $.
\end{enumerate}
 And this topology is characterised by a set $U$ being open iff for any
ultrafilter $\eta $ converging to a point $f \in U$ then
$U \in \eta $.
\end{theorem}
\begin{proof}
\begin{lemma*}
The condition $C_{2}$ is well-defined, that is, it does not depend on the
representative of the equivalence class of $(b_{x})_{x \in X}$.
\end{lemma*}
\begin{proof}
Suppose
$\sigma _{\pi \eta}(f)=(b_{x})_{x \in X}=(b^{\prime}_{x})_{x \in X}$, and suppose
that for any
$\epsilon > 0 \ \coprod _{x \in X}B(b_{x},\epsilon )\in \eta $. Let
$\epsilon >0$ then:
\begin{equation*}
S =\{x \in X \mid d_{\pi (x)}(b_{x},b^{\prime}_{x})<\epsilon /2\} \in \pi
\eta ,
\end{equation*}
so
\begin{equation*}
\coprod _{x \in S}\mathcal{F}(x) \in \eta ,
\end{equation*}
so
\begin{equation*}
\coprod _{x \in S}\mathcal{F}(x) \bigcap \coprod _{x \in X}B(b_{x},
\epsilon /2)= \coprod _{x \in S}B(b_{x},\epsilon /2) \in \eta .
\end{equation*}
Now let $g \in \coprod _{x \in S}B(b_{x},\epsilon /2)$ then
\begin{equation*}
d(g,b^{\prime}_{x})\leq d(g,b_{x})+d(b_{x},b^{\prime}_{x})<\epsilon /2+
\epsilon /2=\epsilon ,
\end{equation*}
hence
\begin{equation*}
\coprod _{x \in S}B(b_{x},\epsilon /2) \subseteq \coprod _{x \in X }B(b^{\prime}_{x},
\epsilon ),
\end{equation*}
and this implies that
\begin{equation*}
\coprod _{x \in X}B(b^{\prime}_{x},\epsilon ) \in \eta .\qedhere
\end{equation*}
\end{proof}
Now back to the proof of the theorem, the proof relies on
\cite[Theorem 4.4]{wyler1996convergence}, namely every relation satisfying
conditions UQ1 and UQ4 of \cite{wyler1996convergence} defines a topology
characterised by this relation being the convergence relation on ultrafilters,
we are going to summarise these conditions in the following theorem:

%t3.2 #&#
\begin{theorem}
%%LEAP%%%\label{thm3.2}
\label{Wyler}
Let $X$ be a set and let $\beta X$ be the set of all ultrafilters on
$X$ (the Stone-\v{C}ech compactification of its discrete structure), let
$q$ be a relation on $\beta X \times X$ satisfying the following:
\begin{itemize}
\item UQ1: $\forall x \in X$ $\delta _{x} q x$, here $\delta _{x}$ is the
principal ultrafilter at $x$.
\item UQ4: If $t : S \rightarrow X$ and $u:S \rightarrow \beta X$ are maps
such that $u(s)\ q\ t(s)$ for every $s \in S$, and if
$t \phi \ q \ x $ for an ultrafilter $\phi $ on $S$, then
$\int _{S} u(s) d\phi \ q \ x$.
\end{itemize}
Here the ultrafilter $\int _{S} u(s) d\phi $ is the ultrafilter defined
in \ref{definition_of_an_ultracategory_VTEX1}.

Then there exists a topology on $X$ characterised by being the unique topology
such that the ultrafilter $\phi $ converges to $x \in X$ in the usual sense
iff $\phi q x$.
\end{theorem}
\noindent
\textbf{Property UQ1}
In what follows, let us say that for an ultrafilter $\mu $ on $E$ and a
point $f\in E$, $\mu q f$ if $(\mu ,f)$ satisfies conditions $C_{1}$ and
$C_{2}$ of \ref{topology}. We need to prove that $\delta _{f}qf$. First,
it is obvious that $\pi (\delta _{f})=\delta _{\pi (f)}$ so it converges
to $\pi (f)$, so $\delta _{f}$ has property $C_{1}$ of \ref{topology}. Now if $\sigma _{\delta_{\pi (f)}}(f)=(b_{x})_{x \in X}$ then the
equivalence class is completely determined by $b_{\pi (f)}=f$ (using
\cite[definition 1.4.1(1)]{lurie2018ultracategories}), so
$\delta _{f}$ satisfies property $C_{2}$ of \ref{topology}, so
$\delta _{f}qf$.

\medskip
\noindent
\textbf{Property UQ4}
Now let us prove that the convergence relation defined by
the two properties above ($C_{1}$ and $C_{2}$) satisfies the second condition
of \ref{Wyler}: let $S$ be a set and let $t$ be a map of sets from $S$
to $E$, let $u$ be a map from $S$ to $\beta E$ (where $\beta E$ is the
set of all ultrafilters on E) such that each $u(s)qt(s)$ (that means satisfies
the conditions $C_{1}$ and $C_{2}$), and suppose $t \mu q f \in E$, we
need to show that $\int _{S}u(s)d\mu q f$. In what follows, we will be
calling $\int _{S}u(s)d\mu $, $\alpha $ to make writing easier.

Now showing that $\pi \alpha $ converges to $\pi f$ (in the usual sense
of converging in a topological space) is trivial. To see why we know that
we have a morphism $\Delta _{\mu , \pi u \bullet }$ from
$\int _{X} x d \pi \alpha $ to
$\int _{S}(\int _{X} x \ d \pi u(s))d\mu =\int _{S} \pi (t(s)) d
\mu $ the latter can be shown to be equal to
$\int _{X} x \ d \pi t \mu =\pi (f)$ (this follows from
the fact that $t \mu $ satisfies the two conditions). Since
the only morphisms in ultrasets are identities, this proves that
$\pi \alpha $ converges to $\pi (f)$. Thus $\alpha $ satisfies property
$C_{1}$ of \ref{topology}.

Now it remains to show that for any $\epsilon >0$, and supposing
$\sigma _{\pi \alpha}(f)=(b_{x})_{x \in X}$ then
$\coprod _{x \in X}B(b_{x},\epsilon ) \in \alpha $. 
Now we state the following lemma:

\begin{lemma*}%
Suppose that $\sigma _{\mu}(f)=(q^{\prime}_{s})_{s \in S}$ and
$\sigma _{t \mu}(f)=(q_{e})_{e \in E}$ and $\sigma_{\pi t \mu}(f)=(q''_{x})_{x \in X}$ then
$(q^{\prime}_{s})_{s \in S}=(q_{t(s)})_{s \in S}=(q''_{\pi(t(s))})_{s \in S}=(t(s))_{s\in S}$.%
\end{lemma*}

\begin{proof}
Since $t\mu $ converges to $f$, condition $C_{1}$ gives that
$\pi t \mu = \pi (t \mu )$ converges to $\pi (f)$. So, applying
Lemma~\ref{ultresult} to the map $\pi t : S \rightarrow X$
and the ultrafilter $\mu $ on $S$, we get
$\sigma _{\mu}(f)=(q''_{\pi t(s)})_{s \in S}$. Similarly we can show that $(q_{e})_{e \in E}=(q''_{\pi(e)})_{e \in E}$ by applying Lemma~\ref{ultresult} to the map $\pi: E \to X$ and the ultrafilter $t\mu$ on $E$, and hence $(q^{\prime}_{s})_{s \in S}=(q_{t(s)})_{s \in S}=(q''_{\pi(t(s))})_{s \in S}$.

Now let $\epsilon>0$, by the fact $t\mu $ converges to $f$, we get that
\begin{equation*}
\coprod _{x \in X}B(q''_{x},\epsilon )\in t \mu ,
\end{equation*}
and this implies that
\begin{equation*}
t^{-1}(\coprod _{x \in X}B(q''_{x},\epsilon )) \in \mu ,
\end{equation*}

Any element of this set satisfies
\begin{equation*}
d_{\pi(t(s))}(t(s), q''_{\pi(t(s))})< \epsilon
\end{equation*}

and so, using the fact that $(q_{t(s)})_{s \in S}=(q''_{\pi(t(s))})_{s \in S}$ we get%
\begin{equation*}
\{s \in S \ \mid d_{\pi(t(s))}(q_{t(s)},t(s))< \epsilon \} \in \mu ,
\end{equation*}
and hence 
\begin{equation*}
(t(s))_{s \in S}=(q_{t(s)})_{s \in S}
\end{equation*}
which completes the proof of the lemma.%
\end{proof}

Now let $\epsilon >0$, let us start by writing the diagram
\cite[definition 1.4.1(2)]{lurie2018ultracategories} for the family
$(\pi (u(s)))_{s \in S}$; as the reader can verify easily, $\pi \alpha = \pi \int_S u(s) d \mu = \int_S \pi u(s) d\mu$:
\[
\begin{tikzcd}
F(\pi (f)) \arrow[dd, "\sigma _{\mu}"']
\arrow[rrrrr, "\sigma _{\pi \alpha}"] & & & & & \int _{X}F(x)d \pi
\alpha \arrow[dd, "{\Delta _{\mu ,\pi u_{\bullet}}}"]
\\
& & & & &
\\
\int _{S}F(\pi t(s))d\mu
\arrow[rrrrr, "\int _{S}\sigma _{\pi (u(s))}d\mu "'] & & & & & \int _{S}(
\int _{X}F(x)d\pi u(s))d\mu
\end{tikzcd}
\]
since $f \in F(\pi (f))$ and since
$\sigma _{\mu}(f)=(t(s))_{s\in S}$ then the commutativity of the diagram
tells us that if $\sigma _{u(s)}(t(s))= (a^{s}_{x})_{x \in X}$ then
\begin{equation*}
((a^{s}_{x})_{x \in X})_{s \in S} =((b_{x})_{x \in X})_{s \in S},
\end{equation*}
which translates to the fact that for any $\epsilon ^{\prime}$ we have that
\begin{equation*}
\{s \in S \ \mid \  d_{\pi u(s)}((a^{s}_{x})_{x \in X},(b_{x})_{x \in X})<
\epsilon ^{\prime}\}\in \mu ,
\end{equation*}
in particular
\begin{equation*}
\{s \in S \ \mid \  d_{\pi (u(s))}((a^{s}_{x})_{x \in X},(b_{x})_{x \in X})<
\epsilon /2\}\in \mu .
\end{equation*}
We also know that
\begin{equation*}
\coprod _{x \in X}B(b_{x},\epsilon /2) \in t \mu ,
\end{equation*}
which implies that
\begin{equation*}
t^{-1}(\coprod _{x \in X}B(b_{x},\epsilon /2))\in \mu .
\end{equation*}
It follows that their intersection
\begin{equation*}
\{s \in S \ \mid \ d_{\pi (u(s))}((a^{s}_{x})_{x \in X},(b_{x})_{x \in X})<
\epsilon /2\} \bigcap t^{-1}(\coprod _{x \in X}B(b_{x},\epsilon /2))
\in \mu .
\end{equation*}
Now our goal is to show that
\begin{equation*}
\{s \in S \ \mid \ d_{\pi (u(s))}((a^{s}_{x})_{x \in X},(b_{x})_{x \in X})<
\epsilon /2\}\bigcap t^{-1}(\coprod _{x \in X}B(b_{x},\epsilon /2))
\subseteq \{s \in S \ \mid  \coprod _{x \in X}B(b_{x},\epsilon ) \in u(s)\},
\end{equation*}
to do so consider any $s$ in the intersection, we have for such $s$
\begin{equation*}
t(s) \in B(b_{\pi t(s)}, \epsilon /2),
\end{equation*}
and on the other hand
\begin{equation*}
d_{\pi u(s)}((a^{s}_{x})_{x \in X},(b_{x})_{x \in X})<\epsilon /2,
\end{equation*}
which implies that the set
\begin{equation*}
\{x \in X \ \mid  d_{x}(b_{x},a^{s}_{x})<\epsilon /2\}\in \pi u(s),
\end{equation*}
which is equivalent to saying that
\begin{equation*}
\pi ^{-1}(\{x \in X \mid d_{x}(b_{x},a^{s}_{x})<\epsilon /2\}) \in u(s).
\end{equation*}
Now we already know that
\begin{equation*}
\coprod _{x \in X}B(a^{s}_{x},\epsilon /2) \in u(s),
\end{equation*}
thus
\begin{equation*}
\pi ^{-1}(\{x \in X \mid d_{x}(b_{x},a^{s}_{x})<\epsilon /2\}) \bigcap
\coprod _{x \in X}B(a^{s}_{x},\epsilon /2) \in u(s),
\end{equation*}
So it remains to show that this is a subset of
$\coprod _{x \in X}B(b_{x},\epsilon )$,
\newline
to do so let
$h \in \pi ^{-1}(\{x \in X \ d_{x}(b_{x},a^{s}_{x})<\epsilon /2\})
\bigcap \coprod _{x \in X}B(a^{s}_{x},\epsilon /2)$, we have that
\begin{equation*}
d_{\pi (h)}(h,b_{\pi (h)})\leq d_{\pi (h)}(h,a^{s}_{\pi (h)})+d_{\pi (h)}(a^{s}_{\pi (h)},b_{\pi (h)})<
\epsilon /2 +\epsilon /2=\epsilon ,
\end{equation*}
so
\begin{equation*}
h \in \coprod _{x \in X}B(b_{x},\epsilon ),
\end{equation*}
and this proves that
\begin{equation*}
\pi ^{-1}(\{x \in X \ d_{x}(b_{x},a^{s}_{x})<\epsilon /2\}) \bigcap
\coprod _{x \in X}B(a^{s}_{x},\epsilon /2) \subseteq \coprod _{x \in X}B(b_{x},
\epsilon ),
\end{equation*}
hence
\begin{equation*}
\coprod _{x \in X}B(b_{x},\epsilon ) \in u(s),
\end{equation*}
and thus
\begin{equation*}
\{s \in S \ \mid \ d_{\pi (u(s))}((a^{s}_{x})_{x \in X},(b_{x})_{x \in X})<
\epsilon /2\}\bigcap t^{-1}(\coprod _{x \in X}B(b_{x},\epsilon /2))
\subseteq \{s \in S \ \mid  \coprod _{x \in X}B(b_{x},\epsilon ) \in u(s)\},
\end{equation*}

so as a result $\coprod _{x \in X}B(b_{x},\epsilon ) \in \alpha $, and
this is true for any $\epsilon > 0$. So $\alpha $ satisfies property
$C_{2}$ of \ref{topology}, so we may deduce that $\alpha q f$.
\end{proof}
\begin{note*}%
\normalfont By definition, the topology characterised by properties
$C_{1}$ and $C_{2}$ of \ref{topology} makes $\pi $ continuous, since the
condition $C_{1}$ implies that if $\eta $ converges to $f$, then
$\pi \eta $ converges to $\pi (f)$.
\end{note*}
%

%s3.3 #&#
\subsection{Characterisation of the topology}
%%LEAP%%%\label{sec3.3}
\label{charecterisation}

%t3.3 #&#
\begin{theorem}%
%%LEAP%%%\label{thm3.3}
\label{condition}
\label{charecterisation_theorem_VTEX1}
Let $X$ be a compact Hausdorff space and let
$\mathcal{F}: \ X \rightarrow \mathsf{k}\text{-}\mathrm{CompMet}$ be a left
ultrafunctor. A set
$C=\coprod _{x \in X} U(x) \subseteq \coprod _{x \in X} \mathcal{F}(x)$
is open in the topology $\tau $ defined in \ref{tau}, iff it satisfies
the following condition:%
\newline For every ultrafilter $\mu $ on $X$ converging to a point
$\bar{x} \in \pi (\coprod _{x \in X} U(x))$ and $\forall g \in U(\bar{x})$ if
$\sigma _{\mu}(g)=(b(x))_{x \in X}$, then $ \exists W \in \mu $ and
$\epsilon > 0$, such that $B_{x}(b_{x},\epsilon ) \subseteq U(x)$ for any
$x \in W$ (in other words
$\coprod _{x \in W}B(b_{x},\epsilon ) \subseteq C $). %

\begin{note*}
\normalfont Although $\epsilon $ and $W$ depend on the representative of
the class of $\sigma _{\mu}(g)$, their existence does not depend on the
representative, so this condition is well-defined.
\end{note*}
\end{theorem}

\subsubsection*{Proof of the {``if''} direction}

Let $V=\coprod _{x \in X}U(x)$ be a set satisfying the condition
of Theorem~\ref{condition}. Our goal is to show that $V$ is open, by showing
that for any ultrafilter $\mu $ converging to $f \in V$, we have
$V \in \mu $. Let $\mu $ be such an ultrafilter converging to some
$f \in V$, by definition of convergence relation on $E$, we get that
$\pi \mu $ converges to $\pi (f)$. Now, we note that
$f \in U(\pi (f))$. So if $\sigma _{ \pi \mu}(f)=(b_{x})_{x \in X}$, then using the condition of the theorem
$\exists W \in \pi \mu \ \text{and} \ \epsilon >0$ such that for any
$ x \in W \ B(b_{x},\epsilon ) \subseteq U(x)$, so
$\coprod _{x \in W}B(b_{x},\epsilon )\subseteq V$. Now
$\coprod _{x \in X}B(b_{x},\epsilon ) \in \mu $ by property $C_{2}$, but
also $ \pi ^{-1}(W) \in \mu $ (since $\pi $ is continuous),
so their intersection
$ \pi ^{-1}(W) \bigcap \coprod _{x \in X}B(b_{x},\epsilon ) \in \mu $,
but now we see that:
\begin{equation*}
\pi ^{-1}(W) \bigcap \coprod _{x \in X}B(b_{x},\epsilon) =
\coprod _{x \in W} B(b_{x},\epsilon )\subseteq V.
\end{equation*}
thus $V \in \mu $, so $V$ is open.

\subsubsection*{Proof of the {``only if'' } direction}

We recall classical results regarding ultrafilters:

%l3.1 #&#
\begin{lemma}
%%LEAP%%%\label{lem3.1}
\label{finite_intersection_VTEX1}
Let $E$ be a set and let $ U \subseteq \mathcal{P}(E)$ (the powerset of
$E$) be a set of non-empty subsets of $E$, satisfying the finite intersection
property, then there exists an ultrafilter $\nu $ such that
$A \in \nu $ for all $A \in U$.
\end{lemma}
\begin{proof}
Define the filter $\alpha $ by setting
$A \in \alpha \ \text{iff} \ A \supseteq A^{\prime}$ for some
$A^{\prime} = \bigcap U^{\prime}$ such that $U^{\prime}$ is finite and  $U^{\prime} \subseteq U $, now it can be checked that $\alpha $ is a filter and thus
must be contained in some ultrafilter $\nu $.
\end{proof}
%
%l3.3 #&#
\begin{lemma}
%%LEAP%%%\label{lem3.3}
\label{important_lemma_VTEX1}
Let $E$ be a set and let $ U \subseteq \mathcal{P}(E)$ be a family of subsets,
let $A \subseteq E$ such that $A$ satisfies the following:
\begin{equation*}
\forall \mu \in \beta E, (U \subseteq \mu \implies A \in \mu ),
\end{equation*}
then $\exists \ B_{1},\dots ,B_{n} \in U$ such that
$B_{1} \bigcap \dots \bigcap B_{n} \subseteq A$.
\end{lemma}
\begin{proof}
In the case where for some finite family $U^{\prime} \subseteq U$,
$\bigcap U^{\prime}$ is empty, we can say that $\varnothing \subseteq A$. Otherwise,
assume by contradiction that no such finite family exists, apply Lemma~\ref{finite_intersection_VTEX1} to the family $U \bigcup \{A^{c} \} $ to get a contradiction,
thus such finite family must exist.%
\end{proof}

Now let $V=\coprod _{x \in X}U(x)$ be an open set, and let $\eta $ be an
ultrafilter on $X$ converging to a point $\bar{x} \in \pi (V)$. Take
$f \in U(\bar{x}) $, and let $\sigma _{\eta}(f)=(b_{x})_{x \in X}$. Let
$\mu $ be an ultrafilter on $E$ and suppose that:
%
%e* #&#
\begin{equation}
\tag{*}
\forall S \in \eta , \ \forall \epsilon >0 \coprod _{x \in S}B(b_{x},
\epsilon )\in \mu .
\label{*}
%%LEAP%%%\label{eq*}
\end{equation}
We can show that in this case $\eta =\pi \mu $ and since we assumed the
condition \eqref{*}, then $\mu $ converges to $f$ and thus since
$V$ is open, we may deduce that $V \in \mu $. So we have proved that
\begin{equation*}
\forall \mu \in \beta E \ ((\forall S \in \eta ,\forall \epsilon >0, \coprod _{x \in S} \ B(b_{x},
\epsilon) \in \mu \implies V \in \mu ).%
\end{equation*}
Hence we may deduce by Lemma~\ref{important_lemma_VTEX1} that:
\begin{equation*}
\exists \epsilon >0 \ \text{and} \  W \in \eta \ \text{such that} \
\coprod _{x \in W}B(b_{x},\epsilon )\subseteq V.
\end{equation*}
So the set $V$ satisfies the condition of Theorem~\ref{charecterisation_theorem_VTEX1}.

%s3.4 #&#
\subsection{The topology definition satisfies the Axioms of ~\ref{bundle_definition_VTEX1}}
\label{sec3.4}

Now we want to prove that our definition of the topology on
$\coprod _{x \in X}\mathcal{F}(x)$ from Theorem~\ref{topology} gives rise
to a bundle of complete metric spaces so we will check that our definition
satisfies the three axioms of Definition~\ref{bundle_definition_VTEX1}.

\subsubsection*{Axiom 1}
\label{axiom_1_VTEX1}

To prove that the distance function is upper semi-continuous. Let
$\pi : \coprod _{x \in X}\mathcal{F}(x)\rightarrow X$ be the projection
map, suppose that $\mu $ is an ultrafilter on $E \times _{X} E$ such that
$\mu $ converges to $(f,g)$. Now let $r>d(f,g)$. First notice that
$\pi \circ \pi _{1}=\pi \circ \pi _{2}$. Now we have that
$\pi _{1} \mu $ converges to $f$ and $\pi _{2} \mu $ converges to
$g$ (simply because projections are continuous).

Suppose that $\sigma _{\pi \pi _{1} \mu}(f)=(b_{x})_{x \in X}$ and
$\sigma _{\pi \pi _{2} \mu}(g)=(c_{x})_{x \in X}$. Take
$\epsilon _{1}$, $\epsilon _{2}$ and $\epsilon_3$ such that they satisfy
$r-(\epsilon _{1}+\epsilon _{2} +\epsilon_3)>d(f,g)$, we know that
$\coprod _{x \in X}B(b_{x},\epsilon _{1}) \in \pi _{1}\mu $ and that
$\coprod _{x \in X}B(c_{x},\epsilon _{2}) \in \pi _{2}\mu $. We also know
that
$d_{\pi \pi _{1} \mu}((b_{x})_{x \in X},(c_{x})_{x \in X}) \leq d(f,g)$
so that means that there exists some set
$L \in \pi \pi _{1} \mu =\pi \pi _{2} \mu $ such that
$\forall x \in L$ $d(b_{x},c_{x}) < d(f,g) + \epsilon_3$. So we deduce that
$\coprod _{x \in L}B(b_{x},\epsilon _{1}) \in \pi _{1} \mu $ and that
$\coprod _{x \in L}B(c_{x},\epsilon _{2}) \in \pi _{2}\mu $ (by intersecting
with $\pi ^{-1}(L)$).

Now let
$(h,l) \in \pi _{1}^{-1}(\coprod _{x \in L}B(b_{x},\epsilon _{1}))
\bigcap \pi _{2}^{-1}(\coprod _{x \in L}B(c_{x},\epsilon _{2})) $, then
\begin{equation*}
d(h,l) < \epsilon _{1} +\epsilon _{2} +d(b_{\pi (h)},c_{\pi (h)}) <
\epsilon _{1} +\epsilon _{2} +\epsilon_3+ d(f,g) < r .
\end{equation*}
And thus
\begin{equation*}
d^{-1}([0,r))\supseteq \pi _{1}^{-1}(\coprod _{x \in L}B(b_{x},
\epsilon _{1})) \bigcap \pi _{2}^{-1}(\coprod _{x \in L}B(c_{x},
\epsilon _{2})) \in \mu .
\end{equation*}
So $d\mu $ converges to $d(f,g)$ (if we equip $[0,k]$ with the left order
topology of course).

\subsubsection*{Axiom 2}

By definition of the topology on $E$, $\pi $ is continuous. Now we prove
that $\pi $ is open. Let $V=\coprod _{x \in \pi (V)}U(x)$ be a non-empty
open set of $E$, and let $\eta $ be an ultrafilter on $X$ converging to
$x \in \pi (V)$. We wish to show that $\pi (V) \in \eta $. Since
$x \in \pi (V)$ then $U(x)$ is non-empty then let $a \in U(x)$. If
$\sigma _{\eta}(a)=(b_{x} )_{x \in X}$ then $ \exists W \in \eta $ and
$\epsilon >0$ such that
$\coprod _{x \in W}B(b_{x},\epsilon )\subseteq V$, thus
$\pi (\coprod _{x \in W}B(b_{x},\epsilon ))=W\subseteq \pi (V)$.  But
since $W \in \eta $ thus $\pi (V) \in \eta $. So $\pi (V)$ is open.

\subsubsection*{Axiom 3}
\label{axiom_3_VTEX1}

We  want to prove that our definition of topology associated
to left ultrafunctors satisfies axiom (3) of the definition of bundle topology,
which informally means that if $V \subseteq W$ are open sets in
$E =\coprod _{x \in X}\mathcal{F}(x)$ where $\mathcal{F}$ is a left ultrafunctor
from $X$ to the category $\mathsf{k\text{-}CompMet}$, we can``enlarge''
$V$ by some $\epsilon $ and remain inside $W$ (we have of course given
a more formal statement).

Before showing that our construction satisfies axiom $(3)$, we give two
lemmas, which are true for every pair of topological spaces $(E,X)$, satisfying
that there exists a surjection $\pi $ from $E$ to $X$, such that for every
$x \in X$, $\pi ^{-1}(x)$ is a metric space and such that the distance
function from $E \times _{X} E$ to $[0,\infty)$ is upper semi-continuous.

%d3.6 #&#
\begin{definition}
\label{defn3.6}
Let $E,X$ be topological spaces and let $\pi $ be a surjection from
$E$ to $X$, such that for every $x \in X$ $\pi ^{-1}(x)$ is a complete metric
space bounded by some $k$, we call $V\subseteq E$ $\epsilon $-thin iff
for every $f,f^{\prime} \in V$, if $\pi (f)=\pi (f^{\prime})$ then
$d_{\pi (f)}(f,f^{\prime})<\epsilon $.
\end{definition}

%l3.4 #&#
\begin{lemma}
%%LEAP%%%\label{lem3.4}
\label{thin_lemma_VTEX1}
Let $E,X$ be topological spaces, and let $\pi $ be a surjection from
$E$ to $X$, such that for every $x \in X$ $\pi ^{-1}(x)$ is a metric space
and such that the distance function from $E \times _{X} E$ to $[0,\infty)$ is upper
semi-continuous,  then for any $\epsilon >0 $, $E$ has a
basis consisting of $\epsilon $-thin neighbourhoods.
\end{lemma}

\begin{proof}
The distance $E \times _{X} E$ is upper semi-continuous which implies that
for any $ \epsilon >0 $ the set
$\{ (v,v^{\prime}) \in E\times _{X} E \mid d_{\pi (v)}(v,v^{\prime}) <\epsilon \
\}$ is open. This implies that the sets of form
$U_{i}\times _{X} V_{i}$, where $U_{i}$ and $V_{i}$ are open sets such
that:
$U_{i}\times _{X} V_{i}\subseteq \{(v,v^{\prime}) \in E\times _{X} E \mid d_{\pi (v)}(v,v^{\prime})<
\epsilon \}$ form a basis for the subspace topology of
$\{ (v,v^{\prime}) \in E\times _{X} E \mid d_{\pi (v)}(v,v^{\prime}) <\epsilon \}$.
Now the subspace topology of the diagonal is generated by the intersection
of the diagonal with these basic open neighbourhoods. So by
applying the projection map (which is a homeomorphism between the diagonal
of $E \times _{X} E$ and $E$), we may deduce that for any
$\epsilon >0$, $E$ has a topology generated by open sets
$W_{i}=\Delta ^{-1}(U_{i} \times _{X} V_{i})$ where
$\Delta : E \rightarrow E \times _{X} E $ is the diagonal map. Now by construction,
each set of these satisfies the following:
\begin{equation*}
\forall g,g^{\prime} \in W_{i} \ \text{if} \  \pi (g)=\pi (g^{\prime}) \
\text{then} \ d_{\pi (g)}(g,g^{\prime})<\epsilon .\qedhere
\end{equation*}
\end{proof}

%l3.5 #&#
\begin{lemma}
\label{lem3.5}
Let $E,X$ be topological spaces and let $\pi $ be an open surjection from
$E$ to $X$, such that for every $x \in X$, $\pi ^{-1}(x)$ is a metric space
and such that the distance function from $E \times _{X} E$ to $[0,\infty)$ is upper
semi-continuous then the sets $V_{\epsilon}$ as defined in \ref{V} are
open in the topology of $E$.
\end{lemma}

\begin{proof}
Take the distance map from $E \times _{X} V$ to $[0,\infty )$, by upper
semi-continuity, each $(d|_{E \times _{X} V})^{-1}([0,\epsilon ))$ is open
in $E \times _{X} V$. Now, since $E \times _{X} V$ is open in
$E \times _{X} E$, then $d|_{E \times _{X} V}^{-1}([0,\epsilon ))$ is open
in the topology of $E \times_{X} E$, so we can apply the first projection (open
map) to $d|_{E \times _{X} V}^{-1}([0,\epsilon ))$ to get the open set
$V_{\epsilon}$.
\end{proof}

\begin{note*}
The fact that the projection maps of the fibre product are open follows from the fact that pullback along an open map is open.
\end{note*}

Now let $\mathcal{F}$ be a left ultrafunctor from $X$ to the category of
metric spaces bounded by a certain $k$, and let
$E =\coprod _{x \in X}\mathcal{F}(x)$ equipped with the topology defined
in \ref{tau}.  And let $W$ be an open set of $E$,
$f \in W$, and let $\mathcal{N}_{f}$ denote the set of open neighbourhoods
of $f$. We want to show that there exists an open neighbourhood $V$ of
$f$ such that $V \subseteq _{\epsilon} W$ using Lemma~\ref{important_lemma_VTEX1}. Take the family of sets
$\{V_{\epsilon} \mid \epsilon >0, V \in \mathcal{N}_{f} \}$, let
$\mu $ be an ultrafilter on $E$ and suppose that
$\{V_{\epsilon} \mid \epsilon >0, V \in \mathcal{N}_{f}\} \subseteq
\mu $, We want to show that $\mu $ converges to $f$, this will allow us
to use Lemma~\ref{important_lemma_VTEX1}.

First, to prove that $\pi \mu $ converges to $\pi f$, notice that for any
neighbourhood $S$ of $\pi (f)$,  if we take any
$\epsilon >0$, then $\pi ^{-1}(S)_{\epsilon} =\pi ^{-1}(S)$, so
$S \in \mu $, thus $\pi \mu $ converges to $\pi f$. Also
suppose that we have chosen a representative $(b_{x})_{x \in X}$ of the
class of $\sigma _{\pi \mu}(f)$ (in other words
$\sigma _{\pi \mu}(f)=(b_{x})_{x \in X}$). Now we regard $b$ as a map from
$X$ to $E$. We will show that $b \pi \mu $ converges to $f$, first notice
that $\pi b \pi \mu $ converges to $\pi (f)$ since
$\pi b =\mathrm{id}$. Now we prove that for any $\epsilon $,
$\coprod _{x \in X}B(b_{x},\epsilon ) \in b \pi \mu $, this follows from
the fact that
$b^{-1}(\coprod _{x \in X}B(b_{x},\epsilon )) =X\ \in \pi \mu $, so
$b \pi \mu $ converges to $f$.

We have already shown that the distance on $E \times _{X} E$ is upper semi-continuous
(see the subsection \ref{sec3.4}), which implies by Lemma~\ref{thin_lemma_VTEX1} that for any $r>0$, $E$ has a basis consisting of
$r$-thin neighbourhoods, thus take an open neighbourhood $V$ of $f$ such
that this neighbourhood is $\epsilon /2$-thin and take the set
$V_{\epsilon /2}$. Since the set $V$ is an open neighbourhood of $f$, and $b \pi \mu$ converges to $f$, then $V \in b\pi\mu$, and hence  $\pi ^{-1}(\{x \in X \mid b_{x} \in V\}) \in  \mu $.

Let us prove that
$\pi ^{-1}(\{x \mid b_{x} \in V \})\bigcap V_{\epsilon /2} \subseteq
\coprod _{x \in X}B(b_{x},\epsilon )$. Let
$g \in \pi ^{-1}(\{x \mid b_{x} \in V\})\bigcap V_{\epsilon /2}$, then
$\exists h \in V $ such that $\pi (g)=\pi (h)$ and such that
$d(g,h)<\epsilon /2$. Then we have
$d(g,b_{\pi (g)})\leq d(g,h)+d(h,b_{\pi (g)})<\epsilon /2 +\epsilon /2=
\epsilon $. So $\mu $ converges to $f$, thus $W \in \mu $. So we can apply
Lemma~\ref{important_lemma_VTEX1}.

Thus there exist some neighbourhoods
$V^{1}, \dots V^{n}$ of $f$ and some
$\epsilon _{1}, \dots \epsilon _{n}>0$ such that
$\bigcap _{i=1}^{n} V^{i} \subseteq \bigcap _{i=1}^{n} V_{\epsilon _{i}}^{i}
\subseteq W$ (using Lemma~\ref{important_lemma_VTEX1}). Now if we call
$\epsilon = \min_{i=1}^{n} \epsilon _{i}$ and
$V=\bigcap _{i=1}^{n} V^{i}$ (these are not $V$ and $\epsilon $ of the
previous paragraph), then
$V \subseteq V_{\epsilon} \subseteq \bigcap _{i=1}^{n} V_{\epsilon _{i}}^{i}
\subseteq W$.

\begin{note*}
\normalfont We can also conclude that the sets $V_{\epsilon}$, for
a neighbourhood $V$ of $f$  form a basis for the neighbourhood system
at $f$.
\end{note*}

%s3.5 #&#
\subsection{$\mathcal{L}$ is a functor}
\label{sec3.5}

We have described the way $\mathcal{L}$ acts on objects. Now let us describe
how it acts on morphisms:

Let $X$ be a compact Hausdorff space and $\mathcal{F}$ and
$\mathcal{F^{\prime}}$ be two left ultrafunctors and let $\nu $ be a natural
transformation of left ultrafunctors between $\mathcal{F}$ and
$\mathcal{F^{\prime}}$. Then the induced map of bundles is
$\psi =\mathcal{L}(\nu )$ defined by
\begin{equation*}
\psi |_{\pi ^{-1}(x)}=\nu _{x}.
\end{equation*}
Now we need to prove that this induced map $\psi $ is indeed a map of bundles.
Let $E=\mathcal{L}(\mathcal{F})$ and
$E^{\prime}=\mathcal{L}(\mathcal{F}^{\prime})$, we already have the commutativity
of the diagram:
\begin{equation*}
\begin{tikzcd}
E && {E^{\prime}}
\\
& X \arrow["{\pi ^{\prime}}"', from=1-3, to=2-2]
\arrow["\pi ", from=1-1, to=2-2] \arrow["\psi ", from=1-1, to=1-3]
\end{tikzcd}
\end{equation*}
  We also have that $\psi |_{\pi ^{-1}(x)}$ is a contraction
by definition.

It remains to show that $\psi $ is continuous, to do so let $\mu $ be an
ultrafilter on $E$ that converges to $f \in E$. Let us show that
$\psi \mu $ converges to $\nu _{\pi (f)}(f)$. We start by proving that
$ \pi ^{\prime}\psi \mu $ converges to
$\pi (f)=\pi ^{\prime}(\nu _{\pi (f)}(f))$ but this follows from the commutativity
of the following diagram:
\begin{equation*}
\begin{tikzcd}
E && {E^{\prime}}
\\
& X \arrow["{\pi ^{\prime}}"', from=1-3, to=2-2]
\arrow["\pi ", from=1-1, to=2-2] \arrow["\psi ", from=1-1, to=1-3]
\end{tikzcd}
\end{equation*}
 Now suppose that
$\sigma ^{\prime}_{\pi ^{\prime}\psi \mu}(\nu _{\pi (f)}(f))=\sigma ^{\prime}_{\pi \mu}(
\nu _{\pi (f)}(f))=(b^{\prime}_{x})_{x \in X}$ and
$\sigma _{\pi \mu}(f)=(b_{x})_{x \in X}$. We know that since $\nu $ is
a natural transformation of left ultrafunctors then the following diagram
commutes:
\begin{equation*}
\begin{tikzcd}
{\mathcal{F}(\pi (f))} && {\int _{X}\mathcal{F}(x)d\pi \mu}
\\
{}
\\
{\mathcal{F^{\prime}}(\pi (f))} && {\int _{X}\mathcal{F^{\prime}}(x)d\pi \mu}
\arrow["{\sigma _{\pi \mu}}", from=1-1, to=1-3]
\arrow["{\nu _{\pi (f)}}"', from=1-1, to=3-1]
\arrow["{\sigma ^{\prime}_{\pi \mu}}"', from=3-1, to=3-3]
\arrow["{\int _{X}\nu _{x}d\mu}", from=1-3, to=3-3]
\end{tikzcd}
\end{equation*}
Which implies that
$(b^{\prime}_{x})_{x\in X}=(\nu _{x}(b_{x}))_{x \in X}$, so for simplicity,
we are going to take $(\nu _{x}(b_{x}))_{x \in X}$ as a
representative of the class. Now, let $\epsilon >0$ we are going to show
that:
\begin{equation*}
\coprod _{x \in X}B(b_{x},\epsilon ) \subseteq \psi ^{-1}(\coprod _{x \in X}B(
\nu _{x}(b_{x}),\epsilon ))
\end{equation*}
Remember that $\psi $ is the map such that
$\psi |_{\mathcal{F}(x)}=\nu _{x}$. Now let $x \in X$, take
$g \in \mathcal{F}(x)$ such that $d(g,b_{x})<\epsilon $, since
$\nu _{x}$ is a contraction we get
$d(\nu _{x}(g),\nu _{x}(b_{x}))<\epsilon $. So
\begin{equation*}
g \in \psi ^{-1}(\coprod _{x \in X}B(\nu _{x}(b_{x}),\epsilon )),
\end{equation*}
and this proves that:
\begin{equation*}
\coprod _{x \in X}B(b_{x},\epsilon ) \subseteq \psi ^{-1}(\coprod _{x \in X}B(
\nu _{x}(b_{x}),\epsilon )),
\end{equation*}
and since $\coprod _{x \in X}B(b_{x},\epsilon ) \in \mu $, then:
\begin{equation*}
\coprod _{x \in X}B(\nu _{x}(b_{x}),\epsilon ) \in \psi \mu .
\end{equation*}
So $\psi \mu $ converges to $\psi (f)$, which completes the functoriality
proof.

%s4 #&#
\section{The inverse functor {from bundles to left ultrafunctors}}
%%LEAP%%%\label{sec4}
\label{4}

The first process we defined is a functorial assignment from the category
of left ultrafunctors from a fixed ultraset $X$ to the adequate category
$\mathsf{k} \text{-} \mathrm{CompMet}$, to the category of bundles of metric
spaces bounded by $k$.

Now we want to define an inverse process, a functorial assignment
$\mathcal{R}$ that sends a bundle with base space $X$ to a left ultrafunctor,
moreover, we claim that the pair $(\mathcal{L},\mathcal{R})$ is an equivalence
of categories between left ultrafunctors and bundles.

But first, let us define $\mathcal{R}$:

%s4.1 #&#
\subsection{The inverse functor construction}
%%LEAP%%%\label{sec4.1}
\label{sigma}

Let $(E,X,\pi )$ be a bundle of complete metric spaces (bounded by some
$k$) and let $X$ be a compact Hausdorff space. Let $\mu $ be an ultrafilter
that converges to $x$. Our goal is to define a left ultrafunctor from
$X$ to $\mathsf{k}\text{-}\mathrm{CompMet}$.

Defining $\mathcal{R}(E)$ as a functor is straightforward: we send each
$x$ to the fibre at $x$, and this is a functor since $X$ has no non-identity
morphism. Now we search for an adequate left-ultrastructure on this functor,
i.e. we wish to construct for every ultrafilter $\mu $ on some set
$S$ and every map of sets $f$ from $S$ to $X$ a morphism
$\sigma _{\mu}$ from $F(\int _{S} f(s) d \mu )$ to
$\int _{S}F(f(s))d\mu $, which satisfies certain axioms indicated in
\cite{lurie2018ultracategories}. We will see soon that it's enough to take
the case $S=X$ and $f =\mathrm{id}$.

Now we turn to the construction:
%
%d4.1 #&#
\begin{definition}
\label{defn4.1}
Let $\mu $ be an ultrafilter on $X$ converging to $x$.

For every $W$ neighbourhood of $f \in \pi ^{-1}(x)$ define the following:
\begin{equation*}
A_{W}=\{ \ (b_{y})_{y \in X} \mid \exists U \in \mu ,\epsilon >0 \
\coprod _{y \in U}B(b_{y},\epsilon ) \subseteq W \ \} .
\end{equation*}
\end{definition}
\noindent
The condition ``$\exists U \in \mu ,\epsilon >0 \ \coprod _{y \in U}B(b_{y},
\epsilon ) \subseteq W$'' is well-defined in the sense that it's independent
of the representative of the class of $(b_{y})_{y \in X}$. Let us call
as usual $\mathcal{N}_{f}$ the set of open neighbourhoods of $f$.
\begin{lemma*}
The family $(A_{W})_{W \in \mathcal{N}_{f}}$ is a filter basis.
\end{lemma*}
\begin{proof}
We need to check that each set of this family is non-empty, and that the
intersection of any two contains a third. Let $W$ be a neighbourhood of
$f$, we want to show that $A_{W}$ is non-empty:
\newline We know that there exists $\epsilon >0$ and an open set $V$ such that
$V \subseteq _{\epsilon} W$ using facts we showed in \ref{sec3.4}. Take
any family $(b_{y})_{y \in \pi (V)}$ such that each $b_{y} \in V$. We already
know that $\pi (V) \in \mu $ since $\mu $ converges to $x$ and the map
$\pi $ is open. Now let us take
$\coprod _{x \in \pi (V)}B(b_{x},\epsilon )$ then by our assumption that
$V \subseteq _{\epsilon} W$, we conclude that
$\coprod _{x \in \pi (V)}B(b_{x},\epsilon ) \subseteq W$ thus the equivalence
class of the family $(b_{y})_{ y \in \pi (V) }$ is in $A_{W}$ thus
$A_{W}$ is non-empty. Finally, it's also clear that
$A_{W} \bigcap A_{W^{\prime}}= A_{W \bigcap W^{\prime}}$
\end{proof}
%
%t4.1 #&#
\begin{theorem}
\label{thm4.1}
The filter
$\{ \ B \ \mid B\supseteq A_{W}, \ W \in \mathcal{N}_{f} \}$ is a Cauchy
filter.%
\end{theorem}

\begin{proof}
Since $E$ is a bundle of complete metric spaces bounded by $k$, then the
hypothesis of Lemma~\ref{thin_lemma_VTEX1} is satisfied, meaning that there exists
an $\epsilon $-thin neighbourhood of $f$. Now take $L_{\epsilon}$ to be
an $\epsilon $-thin neighbourhood of $f$ then we can directly see that
the set $A_{L_{\epsilon}}$ is of diameter $\leq \epsilon $.
\end{proof}
\noindent
Now since the ultraproduct $\int _{X} \pi ^{-1}(x) d\mu $ is complete the
constructed Cauchy filter converges to some element which we are going
to denote by $(b^{f}_{y})_{y \in X}$.

%t4.2 #&#
\begin{theorem}%
\label{thm4.2}
Define $\sigma _{\mu}(f)=(b^{f}_{y})_{y \in X}$ the limit of the Cauchy
filter generated by the family $(A_{W})_{W \in \mathcal{N}_f}$, then the
map $\sigma _{\mu}$ is a contraction from $\pi ^{-1}(x) $ to
$\int _{X} \pi ^{-1}(x) d \mu $.
\end{theorem}

\begin{proof}
Suppose that $f$ and $f^{\prime} \in \pi ^{-1}(x)$ are such that
$d(f,f^{\prime})=\epsilon $. Take $\epsilon' >0$, since the distance map is upper semi-continuous, there exists a neighbourhood $W_{1}$ of $f$ and a neighbourhood
$W_{2}$ of $f^{\prime}$ such that $d(g,g^{\prime}) < \epsilon +\epsilon'/2 $ for every
$(g,g^{\prime}) \in W_{1} \times _{X} W_{2}$. Now, $B((b^{f}_{y}),\epsilon ^{\prime}/4)$ intersects any element of the Cauchy filter in particular $A_{W_{1}}$. Same thing
$B((b^{f^{\prime}}_{y}),\epsilon ^{\prime}/4)$ intersects any element of the second
Cauchy filter converging to $(b^{f^{\prime}}_{y})_{y \in X}$ in particular
$A_{W_{2}}$.

%l4.1 #&#
\begin{lemma}
%%LEAP%%%\label{lem4.1}
\label{Important_Lemma_2_VTEX1}
Let $f \in E$ and let $\mu $ be an ultrafilter on $X$ converging to
$\pi (f)$, suppose $\sigma _{\mu}(f)=(b_{y}^{f})_{y \in X}$ then for any
$\epsilon >0$ and any $W$ neighbourhood of $f$ if
$(g_{y})_{y \in X} \in A_{W } \bigcap B((b^{f}_{y}),\epsilon )$, there
exists some $U \in \mu $ such that $U \subseteq \pi ( W)$ and such that
$\forall y \in U$ $g_{y} \in W$ and $d(g_{y},b^{f}_{y})<\epsilon $.
\end{lemma}

\begin{proof}
Take $(g_{y})_{y \in X} \in A_{ W}$, then there exists $U_{1}$ such that
$U_{1} \subseteq \pi (W)$ and $U_{1} \in \mu $ and $\exists \  r>0$ such
that $\coprod _{y \in U_{1}}B(g_{y},r)\subseteq W $. Now since
$(g_{y})_{y \in X} \in B((b^{f}_{y}),\epsilon )$ then there exists
$U_{2} \in \mu $ such that $\forall y \in U_2, d(b^{f}_{y},g_{y})<\epsilon $ thus
$U=U_{1} \bigcap U_{2} $ will satisfy the requirements above.
\end{proof}

Using the Lemma~\ref{Important_Lemma_2_VTEX1}, if
$(g_{y})_{y \in X} \in A_{W_{1} } \bigcap B((b^{f}_{y}),\epsilon ^{\prime}/4)$
then there exists some $U \in \mu $ and $U \subseteq \pi ( W_{1})$ such
that $\forall y \in U$ $g_{y} \in W_{1}$ and
$d(g_{y},b^{f}_{y})<\epsilon ^{\prime}/4$. Same thing take
$(h_{y})_{y \in X} \in A_{W_{2}}\bigcap B((b^{f^{\prime}}_{y}),\epsilon ^{\prime}/4)
$ such that $\exists V \in \mu $ such that $V \subseteq \pi (W_{2})$ such
that $\forall y \in V$ $h_{y} \in W_{2} $ and
$d(h_{y},b^{f^{\prime}}_{y})<\epsilon ^{\prime}/4$.

Now this would mean that for any $y \in V\bigcap U \in \mu $,
$d(b^{f^{\prime}}_{y},b^{f}_{y}) <\epsilon+\epsilon ^{\prime}$ and since $\epsilon ^{\prime}$ was arbitrary then when passing to the ultraproduct
$d((b^{f^{\prime}}_{y}),(b^{f}_{y})) \leq \epsilon $ so $\sigma _{\mu}$ is
a contraction (and we get for free that it is also continuous).
\end{proof}

\begin{note*}
\label{note_regarding_left_ultrafunctors_VTEX1}
\normalfont We gave the definition for the maps $\sigma _{\mu}$ for ultrafilters
on $X$. Now this can be easily extended to an arbitrary set $S$ as follows:
if we have an ultrafilter $\mu $ on a set $S$ and a map $M$ of sets from
$S$ to $X$, then if $M \mu $ converges to $x$ and $f$ in
$\pi ^{-1}(x)$, then if $\sigma _{M\mu}(f)=(b^{f}_{y})_{y \in X} $, we define
$\sigma _{\mu}(f)=(b^{f}_{M(s)})_{s \in S}$. Notice that this is the only
valid way to define $\sigma _{\mu}$ for arbitrary $S$ to ensure that the
following diagram commutes:
\begin{equation*}
\begin{tikzcd}
{\mathcal{R}(E)(x)} &&& {\int _{S} \mathcal{R}(E)(M_{s}) d \mu}
\\
\\
&&& {\int _{X} \mathcal{R}(E)(y) d M\mu}
\arrow["{\sigma _{\mu}}", from=1-1, to=1-4]
\arrow["{\sigma _{M\mu}}"', from=1-1, to=3-4]
\arrow["{\Delta _{\mu ,M}}", from=3-4, to=1-4]
\end{tikzcd}
\end{equation*}

But this is exactly the Lemma~\ref{ultresult} that we showed
earlier.
\end{note*}

Now let us prove that this gives a Left ultrastructure on the functor
$x \mapsto \pi ^{-1}(x)$ (which means that we need to check that our definition
satisfies axioms $(0\text{-}1\text{-}2)$ of
\cite[definition 1]{lurie2018ultracategories}. Axiom $0$ is automatic since
the ultraset does not have any non-identity morphism so it remains to check
axioms $(1\text{-}2)$.

\medskip\noindent\textbf{Axiom 1}
Suppose that we have a principal ultrafilter $\delta _{x}$ for some
$x \in X$, let $f$ in $\pi ^{-1}(x)$, let us prove that the limit of the
Cauchy filter $\{A_{W}\}$ is converging to an element that belongs to the
equivalence class of $f$ which we are going to denote by $(f)$, to do so
take any $\epsilon >0$ and take the open ball $B((f),\epsilon )$. Now take
any $\epsilon $-thin neighbourhood $L_{\epsilon}$, we have that
$A_{L_{\epsilon}} \subseteq B((f),\epsilon ) $, thus the Cauchy filter
is converging to $(f)$, and this also provides a proof in
the case when we have an ultrafilter on a set $S$, and a map of sets
$M$ from $S$ to  $X$ since
$M\delta _{s}=\delta _{M(s)}$.

\medskip\noindent\textbf{Axiom 2}
First, let us do the case where we have a collection of ultrafilters on
$X$:
\newline \noindent
Let $(\alpha _{s})_{s \in S}$ be a collection of ultrafilters on $X$ each converging
to $x_{s}$ (that means that they define a map $x : s \mapsto x_{s}$, so
we will be writing $x(s)$ instead of $x_{s}$) and let $\mu $ be an ultrafilter
on $S$. We need to check that the following diagram commutes:
\begin{equation*}
\begin{tikzcd}
F(z) \arrow[rrrrr, "\sigma _{\int _{S}\alpha _{s} d \mu}"]
\arrow[d, "\sigma _{\mu}"] & & & & & \int _{X}F(y)d(\int _{S}\alpha _{s}d
\mu ) \arrow[d, "{\Delta _{\mu ,\alpha \bullet}}"]
\\
\int _{S}F(x_{s}) d\mu
\arrow[rrrrr, "\int _{S} \sigma _{\alpha _{s}}d\mu "] & & & & & \int _{S}(
\int _{X} F(y)d\alpha _{s})d\mu
\end{tikzcd}
\end{equation*}
Here $z$ denotes the limit of the ultrafilter
$\int _{S}\alpha _{s} d\mu $. Now suppose that $f \in F(z)$ and let
$\sigma _{\int _{S}\alpha _{s} d \mu}(f)=(a^{f}_{y})_{y\in X}$. By our
description of the categorical Fubini transform, we have
$\Delta _{\mu ,\alpha \bullet}(\sigma _{\int _{S}\alpha _{s} d \mu}(f))=((a^{f}_{y})_{y\in X})_{s \in S}$.

Let $\sigma _{x\mu}(f)=(b^{f}_{y})_{y \in X}$ (which implies that
$\sigma _{\mu}(f)=(b^{f}_{x(s)})_{s \in S}$ using
\hyperref[note_regarding_left_ultrafunctors_VTEX1]{Note in subsection 4.1}).
Let
$\sigma _{\alpha _{s}}(b^{f}_{x(s)})=(c^{b^{f}_{x(s)}}_{y})_{y \in X}$.
Our goal is to show that
$\Delta _{\mu ,\alpha \bullet}(\sigma _{\int _{S}\alpha _{s} d \mu}(f))=
\int _{S} \sigma _{\alpha _{s}}(\sigma _{\mu}(f))$ which translates to
saying that
$((a_{y}^{f})_{y \in X})_{s \in S}=((c^{b^{f}_{x(s)}}_{y})_{y \in X})_{s \in S}$.

Let $\epsilon >0$, take an $\epsilon /4$-thin open neighbourhood
$W_{1}$ of $f$. Now $A^{x\mu}_{W_{1}} $ must intersect
$B((b^{f}_{y})_{y \in X},\epsilon /4)$ since the Cauchy filter converges
to $(b^{f}_{y})_{y \in X}$ which implies that
$(b^{f}_{y})_{y \in X}$ is in the closure of every element in the filter.
Thus using Lemma~\ref{Important_Lemma_2_VTEX1} we can find an element
$(h_{y})_{y \in X}$ and a set $U_{2} \in x\mu $ such that
$U_{2} \subseteq \pi (W_{1})$ and $d(h_{y},b^{f}_{y})<\epsilon /4$ for
all $y \in U_{2}$ and such that $h_{y} \in W_{1}$
$\forall y \in U_{2}$. And also
$A^{\int _{S} \alpha _{s} d \mu}_{W_{1}}$ must intersect
$ B((a^{f}_{y})_{y \in X},\epsilon /4)$ for the same reason. That means
we can find an element $(g_{y})_{y \in X}$ and a set
$U_{1} \in \int _{S}\alpha _{s} d\mu $ such that
$U_{1} \subseteq \pi (W_{1})$ and $d(g_{y},a^{f}_{y})<\epsilon /4$ for
all $y \in U_{1}$ and such that $g_{y} \in W_{1}$. Now since
$U_{1} \in \int _{S}\alpha _{s} d\mu $ then the set
$H =\{s \in S \ : \ U_{1} \in \alpha _{s}\} \in \mu $, and since
$U_{2} \in x \mu $ then the set $J=x^{-1}U_{2} \in \mu $. Now take
$s \in J \bigcap H $. Since $s \in J$ then $x_{s} \in U_{2}$ then
$d(h_{x(s)},b^{f}_{x(s)})<\epsilon /4$.

Suppose that $\sigma _{\alpha _{s}}(h_{x(s)})=(k_{y})_{y \in X}$. Now,
since each $\sigma _{\alpha _{s}}$ is a contraction then we can deduce
that
$d((c^{b^{f}_{x(s)}}_{y})_{y \in X},(k_{y})_{y \in X})< \epsilon /4$, so
there exists some set $U_{3} \in \alpha _{s}$ such that
$d(c^{b^{f}_{x(s)}}_{y},k_{y})<\epsilon /4$ for every $y \in U_{3} $. We
know that, since $W_{1}$ is a neighbourhood of $h_{x(s)}$, the set
$A^{\alpha _{s}}_{W_{1}}$ must intersect any neighbourhood of
$(k_{y})_{y \in X}$, thus we deduce that there exists some
$U_{4} \in \alpha _{s}$ and an element $(l_{y})_{y \in X}$ such that
$U_{4} \subseteq \pi (W_{1})$ and $d(l_{y},k_{y})<\epsilon /4$ for all
$y \in U_{4}$ and such that $l_{y} \in W_{1}$ $\forall y \in U_{4}$.%

Now since $s \in H$, then $U_{1} \in \alpha _{s}$. Take
$y \in U_{1} \bigcap U_{3} \bigcap U_{4} \in \alpha _{s}$. We have
\begin{equation*}
d(a^{f}_{y},c^{b^{f}_{x(s)}}_{y}) \leq
\underbrace{d(a^{f}_{y},g_{y})}_{y \in U_{1}}+
\underbrace{d(g_{y},l_{y})}_{ W_{1} \ \text{is} \ \epsilon /4 \  \text{-thin}}
+\underbrace{d(l_{y},k_{y})}_{y \in U_{4}} +
\underbrace{d(k_{y},c^{b^{f}_{x(s)}}_{y})}_{y \in U_{3}}<\epsilon ,
\end{equation*}
thus
\begin{equation*}
J \bigcap H \subseteq \{s \in S \mid \ d_{\alpha _{s}}((c^{b^{f}_{x(s)}})_{y \in X},(a^{f}_{y})_{y \in X})<
\epsilon \},
\end{equation*}
and since $J\bigcap H \in \mu $
\begin{equation*}
\{s \in S \mid \ d_{\alpha _{s}}((c^{b^{f}_{x(s)}})_{y \in X}, (a^{f}_{y})_{y \in X})<
\epsilon \} \in \mu .
\end{equation*}
This implies that
$((a_{y}^{f})_{y \in X})_{s \in S}=((c^{b^{f}_{x(s)}}_{y})_{y \in X})_{s \in S}$,
and hence
$\Delta _{\mu ,\alpha \bullet}(\sigma _{\int _{S}\alpha _{s} d \mu}(f))=
(\int _{S} \sigma _{\alpha _{s}} d\mu )(\sigma _{\mu}(f))$ so the diagram
commutes.

Now  consider the more general case when we have a family
of ultrafilters $(\alpha _{s})_{s \in S}$ on some set $T$ and a function
$t \mapsto M_{t}$ from $T$ to $X$. We need to prove that the following
diagram commutes:
\begin{equation*}
\begin{tikzcd}
F(z) \arrow[rrrrr, "\sigma _{\int _{S}\alpha _{s} d \mu}"]
\arrow[d, "\sigma _{\mu}"] & & & & & \int _{T}F(M(t))d(\int _{S}
\alpha _{s}d\mu ) \arrow[d, "{\Delta _{\mu ,\alpha \bullet}}"]
\\
\int _{S}F(x_{s}) d\mu
\arrow[rrrrr, "\int _{S} \sigma _{\alpha _{s}}d\mu "] & & & & & \int _{S}(
\int _{T} F(M_{t})d\alpha _{s})d\mu
\end{tikzcd}
\end{equation*}

To do so let $f \in F(z)$ and  suppose that
$\sigma _{M \int _{S} \alpha _{s} d\mu}(f)=(a^{f}_{y})_{y \in X}$,  then
$\sigma _{\int _{S} \alpha _{s} d\mu}(f)=(a_{M(t)}^{f})_{t \in T}$. On the
other hand, suppose that $\sigma _{x\mu}(f)=(b^{f}_{y})_{y \in X}$, then by
Lemma~\ref{ultresult} applied to $x: S \to X$ ,
$\sigma_{\mu}(f)=(b^{f}_{x(s)})_{s \in S}$. Now
for each $x_{s}=\int _{T}M_t d \alpha _{s}=\int _{X} y d M \alpha _{s}$, suppose
$\sigma _{M\alpha _{s}}(b^{f}_{x(s)})=(c_{y}^{b^{f}_{x(s)}})_{y \in X}$, then
$\sigma _{\alpha _{s}}(b^{f}_{x(s)})=(c_{M_{t}}^{b^{f}_{x(s)}})_{t \in T}$
then
$\int _{S}\sigma _{\alpha _{s}}d\mu((b_{x(s)}^{f})_{s \in S})=((c_{M_{t}}^{b^{f}_{x(s)}})_{t \in T})_{s \in S}$.
We already proved that
$((a^{f}_{y})_{y \in X})_{s \in S}=((c_{y}^{b^{f}_{x(s)}})_{y \in X})_{s \in S}$.
We want to prove that
$((c_{M_{t}}^{b^{f}_{x(s)}})_{t \in T})_{s \in S}=((a_{M_{t}}^{f})_{t \in T})_{s \in S}$.

Let $\epsilon >0$, then the first equality means that
$\{s \in S \mid d((a^{f}_{y})_{y \in X},(c_{y}^{b^{f}_{x(s)}})_{y \in X})<
\epsilon \} \in \mu $. Now take any $s$ in the set above, since
$d((a^{f}_{y})_{y \in X},(c_{y}^{b^{f}_{x(s)}})_{y \in X})<\epsilon $,
then the set
\newline $\{ y \mid d(a^{f}_{y},c_{y}^{b^{f}_{x(s)}})<\epsilon \} \in M \alpha_s $ which allows us to conclude that
$\{ t \in T \mid d(a^{f}_{M_{t}},c_{M_{t}}^{b^{f}_{x(s)}})<\epsilon
\} \in \alpha_s $, thus
$d((a^{f}_{M_{t}}),(c_{M_{t}}^{b^{f}_{x(s)}})) \leq \epsilon $ so we can deduce
that $d(((c_{M_{t}}^{b^{f}_{x(s)}})_{t \in T})_{s \in S},((a_{M_{t}}^{f})_{t \in T})_{s \in S}) \leq \epsilon < 3\epsilon/2$, hence \newline
$\{s \in S \mid d((a^{f}_{M_{t}})_{t \in T},(c_{M_{t}}^{b^{f}_{x(s)}}))<
3\epsilon/2 \} \in \mu $ which shows that
$((c_{M_{t}}^{b^{f}_{x(s)}})_{t \in T})_{s \in S}=((a_{M_{t}}^{f})_{t \in T})_{s \in S}$.

%s4.2 #&#
\subsection{Adjunction}
\label{sec4.2}

We state a basic category theory fact:
%
%l4.2 #&#
\begin{lemma}
\label{lem4.2}
Let $\mathcal{L}$ from $C^{\prime}$ to $C$ be a functor and let
$\mathcal{R}$ be an assignment on objects from $C$ to $C^{\prime}$ such that
$\mathrm{Hom}(\mathcal{L}(X),Y) \simeq \mathrm{Hom}(X,\mathcal{R}(Y))$
for every object $X \in C^{\prime}$ and $Y \in C$ and this bijection
is natural in $X$. Then $\mathcal{R}$ has a functor structure defined as
follows:

The naturality in $X$ allows us to define a natural transformation
$\epsilon $ from  $\mathcal{L}\mathcal{R}$ to $\mathrm{Id}_{C}$ (which would
be the counit of adjunction), then if
$\sigma \in \mathrm{Hom}(X,X^{\prime})$, we define $\mathcal{R}(\sigma )$ to
be the unique map that corresponds to $\epsilon _{X} \circ \sigma $ by
this bijection.
\end{lemma}
We are going to apply this lemma in our case where $\mathcal{L}$ denotes
the functor from the category of left ultrafunctors between an ultraset
and the ultracategory $\mathsf{k}\text{-}\mathrm{CompMet}$ (with natural
transformations of left ultrafunctors as morphisms as defined in
\cite{lurie2018ultracategories}) to bundles over $X$ as we already defined
it, and $\mathcal{R}$ is the reverse assignment defined above.

%t4.3 #&#
\begin{theorem}%
\label{thm4.3}
Let $\mathcal{F}$ be a left ultrafunctor and let $E$ be a bundle, then
$\mathrm{Hom}(\mathcal{L}(\mathcal{F}),E) \simeq \mathrm{Hom}(
\mathcal{F},\mathcal{R}(E)) $, and this bijection is natural in
$\mathcal{F}$.%
\end{theorem}

\begin{proof}
Let $\nu $ be a morphism of bundles from $\mathcal{L}(\mathcal{F})$ to
$E$ then define a natural transformation from $\mathcal{F}$ to
$\mathcal{R}(E)$ by $\nu _{x}=\nu |_{\pi ^{-1}(x)}$. Naturality is immediate
since the category $X$ has no morphisms but identities. Now to check that
it is really a natural transformation of left ultrafunctors: Let
$\mu $ be an ultrafilter on a set $S$, and $M$ a map of sets from
$S$ to $X$ (alternatively a family of points of $X$ indexed by $S$,
$(M_{s})_{s \in S}$) such that $M\mu $ converges to $x$. We
need to check that the following diagram commutes:
\begin{equation*}
\begin{tikzcd}
{\mathcal{F}(x)} && {\int _{S} \mathcal{F}(M_{s}) d\mu}
\\
{}
\\
{\mathcal{R}{E}(x)} && {\int _{S}\mathcal{R}(E)(M_{s}) d\mu}
\arrow["{\int _{S}\nu _{s}d\mu}", from=1-3, to=3-3]
\arrow["{\sigma _{\mu}}", from=1-1, to=1-3]
\arrow["{\sigma ^{\prime}_{\mu}}"', from=3-1, to=3-3]
\arrow["{\nu _{x}}"', from=1-1, to=3-1]
\end{tikzcd}
\end{equation*}
First, we observe that it is enough to check this diagram in the case where
$S=X$, $M=\mathrm{id}$. Indeed, consider the diagram:
\begin{equation*}
\begin{tikzcd}
&&&&& {\int _{S} \mathcal{F}(M_{s}) d\mu}
\\
{\mathcal{F}(x)} &&&& {\int _{X} \mathcal{F}(y) dM \mu}
\\
&&&&& {\int _{S}\mathcal{R}(E)(M_{s}) d\mu}
\\
{\mathcal{R}{E}(x)} &&&& {\int _{X}\mathcal{R}(E)(y) dM\mu}
\arrow["{\nu _{x}}", from=2-1, to=4-1]
\arrow["{\int _{X}\nu _{y}dM\mu}"', from=2-5, to=4-5]
\arrow["{\sigma _{M \mu}}"{description}, from=2-1, to=2-5]
\arrow["{\sigma ^{\prime}_{M\mu}}"{description}, from=4-1, to=4-5]
\arrow["{\sigma _{\mu}}"{description}, from=2-1, to=1-6]
\arrow["{\sigma ^{\prime}_{\mu}}"{description}, from=4-1, to=3-6]
\arrow["{\Delta _{\mu ,M}}"{description}, from=2-5, to=1-6]
\arrow["{\Delta _{\mu ,M}}"{description}, from=4-5, to=3-6]
\arrow["{\int _{S}\nu _{s}d\mu}"{description}, from=1-6, to=3-6]
\end{tikzcd}
\end{equation*}

Our goal is to show that the back square diagram commutes assuming the
front square diagram does (here $x$ is the limit of the ultrafilter
$M\mu )$, notice that the two triangles commute by
\hyperref[note_regarding_left_ultrafunctors_VTEX1]{Note in subsection 4.1}).
The side square commutes by naturality of the ultraproduct diagonal map
(it is easy to check that the naturality condition for these maps follows
from their definition (composition of the (natural) categorical Fubini
transform and the natural isomorphisms $\epsilon $)).

So we will be restricting our attention to ultrafilters on $X$, and we
will be checking the commutativity of the following diagram (again here
$x$ is the limit of the ultrafilter $\mu $):
\begin{equation*}\begin{tikzcd}
{\mathcal{F}(x)} && {\int _{X} \mathcal{F}(y) d\mu}
\\
{}
\\
{\mathcal{R}(E)(x)} && {\int _{X}\mathcal{R}(E)(y) d\mu}
\arrow["{{\int _{X}\nu _{y}d\mu}}", from=1-3, to=3-3]
\arrow["{{\sigma _{\mu}}}", from=1-1, to=1-3]
\arrow["{{\sigma ^{\prime}_{\mu}}}"', from=3-1, to=3-3]
\arrow["{{\nu _{x}}}"', from=1-1, to=3-1]
\end{tikzcd}
\end{equation*}

Now take $f \in \mathcal{F}(x) $, and suppose $\nu _{x}(f)=g$ and
$\sigma _{\mu}(f)=(b_{y})_{y \in X}$. Our goal is to show that
$\sigma ^{\prime}_{\mu}(g)=(\nu _{y}(b_{y}))_{y \in X}$. Suppose that
$\sigma ^{\prime}_{\mu}(g)=(c^{g}_{y})_{y \in X}$, let $W$ be an
$\epsilon /2$-thin neighbourhood of $g$ then by definition of
$(c^{g}_{y})_{y \in X}$, $A_{W}$ must intersect any neighbourhood of
$(c^{g}_{y})_{y \in X}$, in particular
$B((c^{g}_{y})_{y \in X},\epsilon /2)$, thus there exists $L
\in \mu $ and $(f_{y})_{y \in X}$ such that $\forall y \in L$, each
$f_{y} \in W$ and $d(f_{y},c^{g}_{y})<\epsilon /2$.

Now since $\nu ^{-1}(W)$ is a neighbourhood of $f$, then there exists $H
\in \mu $  and $ \epsilon'>0$ such that
$\coprod _{y \in H}B(b_{y},\epsilon ^{\prime}) \subseteq \nu ^{-1}(W)$. Thus,
for any $y \in H$ we get that $\nu _{y}(b_{y}) \in W$.

Now take $y \in H \bigcap L$ (remember that $H \in {\mu}$) then we have
$d(\nu _{y}(b_{y}),c_{y}^{g}) \leq d(\nu _{y}(b_{y}),f_{y}) +d(f_{y},c^{g}_{y})<
\epsilon /2+\epsilon /2=\epsilon $. Thus we get that
$(\nu _{y}(b_{y}))_{y \in X}=(c_{y}^{g})_{y \in X}$, and this terminates
the proof showing the commutativity of the diagram above.

Conversely, suppose that we have a natural transformation $\nu $ of left
ultrafunctors from $\mathcal{F}$ to $\mathcal{R}(E)$. We need to show that
the map $\nu $ defined by $\nu |_{\pi ^{-1}(x)}=\nu _{x}$ is a continuous
map from $\mathcal{L}(\mathcal{F})$ to $E$ (since the other requirements
for being a map of bundles are automatically satisfied). To do so, suppose
that $\mu $ is an ultrafilter on $\mathcal{L}(\mathcal{F})$ that converges
to $f$. Now to prove that $\sigma \mu $ converges to
$\nu (f)=\nu _{\pi (f)}(f)$ in $E$: We know that since $\nu $ is a natural
transformation of left ultrafunctors, if
$\sigma _{\mu}(f)=(b_{y})_{y \in X} $ then
$\sigma ^{\prime}_{\mu}(\nu (f))=(\nu _{y}(b_{y}))_{y \in X}$. Also, we know
that $\coprod _{y \in X}B(b_{y},\epsilon ) \in \mu $ (by definition of
the topology of $\mathcal{L}(\mathcal{F})$). Now, since each
$\nu _{y}$ is a contraction, then
\begin{equation*}
\coprod _{y \in X}B(b_{y},\epsilon ) \subseteq \nu ^{-1}
\coprod _{y \in X}B(\nu _{y}(b_{y}),\epsilon ),
\end{equation*}
thus
\begin{equation*}
\coprod _{y \in X}B(\nu _{y}(b_{y}),\epsilon ) \in \nu
\mu .
\end{equation*}
Thus we have a map of bundles from $\mathcal{L}(\mathcal{F})$ to $E$. Also,
it is clear that these two processes between
$\mathrm{Hom}(\mathcal{L}(\mathcal{F}),E)$ and
$\mathrm{Hom}(\mathcal{F},\mathcal{R}(E)) $ are inverses of each other.
Now, it remains to show naturality in $\mathcal{F}$.
%
%nConvention #&#
\begin{NotationConvention*}
\label{notConvention}
\normalfont If we have a map $\psi $ in
$\mathrm{Hom}(\mathcal{L}(\mathcal{F}),E)$, we will denote
$\widehat{\psi}$ the corresponding map in
$\mathrm{Hom}(\mathcal{F},\mathcal{R}(E))$, conversely, if we have a map
$\kappa $ in $\mathrm{Hom}(\mathcal{F},\mathcal{R}(E))$ then we are going
to denote by $\bar{\kappa}$ the corresponding map in
$\mathrm{Hom}(\mathcal{L}(\mathcal{F}),E)$.
\end{NotationConvention*}

Now to show naturality, let $\nu $ be a natural transformation
of left ultrafunctors from $\mathcal{F^{\prime}}$ to $\mathcal{F}$, we need
to show that the following diagram commutes:
\begin{equation*}
\begin{tikzcd}
{\mathrm{Hom}(\mathcal{L}(\mathcal{F}),E)} &&& {\mathrm{Hom}(
\mathcal{F},\mathcal{R}(E))}
\\
\\
{\mathrm{Hom}(\mathcal{L}(\mathcal{F^{\prime}}),E)} &&& {\mathrm{Hom}(
\mathcal{F^{\prime}},\mathcal{R}(E))} \arrow[from=1-1, to=1-4]
\arrow["{-\circ \mathcal{L}(\nu )}"', from=1-1, to=3-1]
\arrow["{-\circ \nu}", from=1-4, to=3-4] \arrow[from=3-1, to=3-4]
\end{tikzcd}
\end{equation*}
 To do so consider a map $\psi $ of bundles from
$\mathcal{L}(\mathcal{F})$ to $E$. We need to show
$\widehat{\psi \circ \mathcal{L}(\nu)}=\widehat{\psi}\circ \nu $. To do so
let $x \in X$ and let $f \in \mathcal{F^{\prime}}(x) $ then
\begin{equation*}
\widehat{(\psi \circ \mathcal{L}(\nu))}_{x}(f)=(\psi \circ \mathcal{L}(
\nu ))(f)=\psi (\nu _{x}(f)),
\end{equation*}
on the other hand
\begin{equation*}
(\widehat{\psi}\circ \nu )_{x}(f)=(\widehat{\psi}_{x} \circ \nu _{x})(f)=
\widehat{\psi}_{x}((\nu _{x})(f))=\psi (\nu _{x}(f)),
\end{equation*}
so for each $x$
\begin{equation*}
\widehat{(\psi \circ \mathcal{L}(\nu)})_{x}=(\widehat{\psi}\circ \nu )_{x},
\end{equation*}
so
\begin{equation*}
\widehat{\psi \circ \mathcal{L}(\nu)}=\widehat{\psi}\circ \nu ,
\end{equation*}
so the diagram commutes.
\end{proof}

A last thing that we should enlighten is that the functor structure of
$\mathcal{R}$ comes from the adjunction. Let us explain further;
suppose that we have a map of bundles
$ \nu : \  E \rightarrow E^{\prime}$, we defined
$R(\nu )=\widehat{(\epsilon _{E} \circ \nu )}$, where
$\epsilon _{E}$ is the counit of adjunction. We are going to give a better
description of this map once we prove that the counit is an isomorphism.

Now we turn to showing our main theorem. In what follows
$\mathrm{Bun}(\mathsf{k\text{-}CompMet},X)$ denotes the category of bundles
with base space $X$. We remind the reader that
$\mathcal{L}$ is the functor that assigns to every left ultrafunctor a
bundle and that $\mathcal{R}$ is the functor that assigns to every bundle
a left ultrafunctor.

%t4.4 #&#
\begin{theorem}
%%LEAP%%%\label{thm4.4}
\label{first_theorem_VTEX1}
Let $X$ be a compact Hausdorff space, then the pair of functors
$\mathcal{R}$ and $\mathcal{L}$ constitute an equivalence of categories between
$\mathrm{Lult}(X,\mathsf{k\text{-}CompMet})$ and
$\mathrm{Bun}(\mathsf{k\text{-}CompMet},X)$
\end{theorem}

\subsubsection*{The counit {of adjunction} is an isomorphism}

Let
\begin{equation*}
\epsilon : \  \mathcal{L}\mathcal{R} \rightarrow \mathrm{Id}_{\mathrm{Bun}(\mathsf{k\text{-}CompMet},X)}
\end{equation*}
be the counit of adjunction.

%t4.5 #&#
\begin{theorem}
\label{thm4.5}
For every bundle $E$, $\epsilon _{E}$ is a homeomorphism.
\end{theorem}

\begin{proof}
It is clear that $\epsilon _{E}$ is a bijection of sets, so it remains
to show that $E$ and $\mathcal{L}\mathcal{R}(E)$ have the same topology.

\medskip\noindent
\textbf{The topology of $E$ is coarser than $\mathcal{L}\mathcal{R}(E)$}
First, we already get that $\epsilon _{E} $ is continuous from
$\mathcal{L}\mathcal{R}(E)$ to $E$, by the fact that $\epsilon _{E}$ is
a counit which implies it's a map of bundles.

\medskip\noindent
\textbf{The topology of $E$ is finer than $\mathcal{L}\mathcal{R}(E)$}
For the other direction, suppose that $\mu $ is an ultrafilter on
$E$ that converges to $f$, we need to prove that $\mu $ also
converges to $f$ in the topology of $\mathcal{L}\mathcal{R}(E)$. Suppose
that $\sigma _{\pi \mu}(f)=(b_{y}^{f})_{y \in X}$, we need to show
that for any $\epsilon >0$ the set
$\coprod _{y \in X}B(b_{y}^{f},\epsilon ) \in \mu $, to do this take an
$\epsilon /2$-thin neighbourhood $W$ (in the topology of $E$ of course)
of $f$. Now we know that $A_{W} $ must intersect any neighbourhood of
$(b_{y})_{y \in X}$, in particular
$B((b_{y})_{y \in X},\epsilon  /2)$, thus there exists
$L \in \pi \mu $ and $(c_{y})_{y \in X}$, such that for each
$y \in L$ $c_{y} \in W$ and $d(c_{y},b_{y}^{f}) <\epsilon /2$, so
$\pi ^{-1}(L) \in \mu $. On the other hand, $W \in \mu $ since $\mu $ converges
to $f $ in the first topology (topology of $E$).

Now let us prove that
$W \bigcap \pi ^{-1}(L) \subseteq \coprod _{y \in X}B(b^{f}_{y},
\epsilon )$, to do this we take
$g \in W \bigcap \pi ^{-1}(L) $ then
$d(g,b_{\pi (g)})<d(g,c_{\pi (g)})+d(c_{\pi (g)},b_{\pi (g)}) <
\epsilon $. Thus $\coprod _{y \in X}B(b^{f}_{y},\epsilon ) \in \mu $. So,
by the definition of the topology of $\mathcal{L}\mathcal{R}(E)$,
$\mu $ converges to $f$.

All of this allows us to deduce that the two topologies
coincide and $E$ is isomorphic to $\mathcal{L}\mathcal{R}(E)$ as bundles.
\end{proof}

This allows us to describe better how $\mathcal{R}$ acts on morphisms,
suppose that we have a map of bundles
$ \nu : \  E \rightarrow E^{\prime}$. Then
$\mathcal{R}(\nu )=\widehat{(\epsilon _{E} \circ \nu )}$. More precisely, from the
fact that $\epsilon _{E}$ is an isomorphism we get that
$\mathcal{R}(\nu )_{x}(f)=\nu (f)$ for $f \in \mathcal{R}(E)(x)$ (which is exactly
what we expected it to be).

\subsubsection*{The unit {of adjunction} is an isomorphism}

To prove that the unit is an isomorphism consider:
\begin{equation*}
\eta _{F} : \ \mathcal{F} \mapsto \mathcal{R}\mathcal{L}(\mathcal{F}).
\end{equation*}
The two left ultrafunctors from $X$ to
$\mathsf{k}\text{-}\mathrm{CompMet}$ are the same thing as functors. It
remains to show that they have the same left ultrastructure. But this immediately
follows from $\eta _{F}$ being a natural transformation of left ultrafunctors
which is an isomorphism for every $x \in X$.

\subsubsection*{A nice property of bundles}

%t4.6 #&#
\begin{theorem}
\label{thm4.6}
Let $E$ be a bundle of complete bounded metric spaces, then the subspace
topology and the complete metric space topology agree on every fibre.
\end{theorem}
\begin{proof}
Let $W_{x}$ be an open set in the subspace topology of the fibre
$E_{x}$ for some $x \in X$, and let $f \in E_{x}$, there exists an open
set $W$ of $E$ such that $W_{x}= W \bigcap E_{x}$. Now we know that there
exists a set $V$ such that $f \in V \subseteq _{\epsilon} W$. In other
words, $f \in V \subseteq V_{\epsilon} \subseteq W$. Now, by definition
of $V_{\epsilon}$
$B(f,\epsilon ) \subseteq V_{\epsilon} \bigcap E_{x}$, hence $W_{x}$ is
open in metric topology.

On the other hand, let $(b_{i})$ be a net of elements in $E_{x}$ that converges
to $b$ in the topology of $E$, we need to show that $(b_{i})$ converges
to $b$ in metric topology. To do so, consider the net
$(b_{i},b) \in E \times _{X} E$, this net converges to $(b,b)$ which satisfy
$d(b,b)=0$. By upper semi-continuity of the distance, for every
$\epsilon > 0$, there exists a neighbourhood $W$ in
$E \times _{X} E$ and some $i_{0}$ such that every two points in the same
fibre in $W$ have distance $\leq \epsilon $, and such that for any
$i \succ i_{0}$, $(b_{i},b) \in W$, thus $(b_{i})$ converges to $b$ in
the metric topology. So both topologies on $E_{x}$ agree. This proof is
inspired by a similar one in
\cite[proposition 1.3]{fell1969extension} or
\cite[proposition 13.11]{fell1988representations}.
\end{proof}

\subsubsection*{Another construction of the left-ultrastructure of $\mathcal{R}(E)$}

We give another construction of the left ultrastructure of
$\mathcal{R}(E)$ for a bundle $E$, that works only when the bundle
$E$ has enough cross-sections.

\medskip
\noindent
\textbf{Note.} By a bundle having enough (local) cross sections, we mean that
for every $f \in E$, there exists an open neighbourhood $U$ of
$\pi (f)$, and continuous function $a:U \rightarrow E$, such that
$\pi \circ a=\mathrm{id}_{U}$.
\begin{theorem*}
Let $\mathcal{M}=\coprod _{x \in X}\mathcal{M}_{x}$ be a bundle of complete
metric spaces bounded by a certain $k$, and let $\mathcal{F}$ be the left
ultrafunctor $x \mapsto \mathcal{M}_{x}$, then for any $x \in X$, if
$\mu $ is an ultrafilter on $X$ converging to $x$, and if
$a: U \rightarrow \mathcal{M}$ is a local continuous section to the projection
map $\pi $ (here $U$ open in $X$), such that $a(x)=f$, then we claim that
$\sigma _{\mu}(f)=(a(y))_{y \in U}$.

\begin{note*}
\normalfont It is enough to define a member of the ultraproduct on some
$U \in \mu $.
\end{note*}
\end{theorem*}
\begin{proof}
Let $W$ be an open neighbourhood of $f$, by continuity of $a$ the ultrafilter
$a \mu $ converges to $f$, that means that for any $\epsilon > 0$
$\coprod _{y \in U}B(a(y),\epsilon ) \in \mu $, hence
$(a(y))_{y \in U} \in A_{W}$, and the Cauchy filter associated to the construction
of $\mathcal{R}(E)$ converges to $(a(y))_{y \in U}$.
\end{proof}
%

%s5 #&#
\section{Generalising to any structure}
%%LEAP%%%\label{sec5}
\label{5}

In this section, we give the construction of bundles of structures of continuous
model theory.
\newline
{}By a structure of continuous model theory, we
mean an interpretation of sorts, relation, and function symbols, which
is not required to satisfy any axiom. This is a necessary intermediate
step before defining bundles of models of continuous model theory. But
first, we give a necessary introduction to continuous model theory.

%s5.1 #&#
\subsection{The ultracategory of models}
\label{sec5.1}

We first recall a few concepts from continuous model theory; this exposition
follows mostly \cite{farah2021model} and \cite{hart2023introduction}.

\subsubsection*{Signature}

The signature of continuous model theory consists of the following triplet
$\langle \mathfrak{S},\mathfrak{F},\mathfrak{R}\rangle $, where
\begin{enumerate}
\item $\mathfrak{S}$ is the set of \emph{sorts symbols}, such that each
symbol comes equipped with a symbol $d_{S}$ (should be interpreted as the
distance function), and a constant $k_{S}$ (actual constant not just a
symbol) (which should be interpreted as an upper bound for the distance
function).
\item $\mathfrak{F}$ is the set of \emph{function symbols}, and for each
symbol $f$ we specify a formal domain
$\mathrm{dom}(f)=(S_{1},\dots ,S_{n})$, a formal range
$\mathrm{rng}(f)=S^{\prime}$, and a function $\delta _{f}$, which should be
interpreted as the uniform continuity modulus of $f$.
\item $\mathfrak{R}$ is the set of \emph{relation symbols}, each equipped
with a compact interval of $\mathbb{R}$ (which should be interpreted as
the range of these relations), as well as a uniform continuity modulus
$\delta _{\phi}$ for every $\phi \in \mathfrak{R}$.
\end{enumerate}

\begin{note*}
\normalfont We can, and we are going to treat the distance symbol as a
relation symbol.
\end{note*}

Now we are in a position to define terms and formulae in continuous model
theory:%
\newline As usual, the definition is inductive:
\begin{itemize}
\item We start first by considering infinitely many variable symbols for
each sort $x_{i}^{S}$ as terms.
\item If $t_{1},\dots ,t_{n}$ are terms of sorts
$S_{1},\dots ,S_{n}$ and $f$ is a function symbol with range $S^{\prime}$ then
$f(t_{1},\dots ,t_{n})$ is a term of sort $S^{\prime}$.
\end{itemize}
All terms get uniform continuity moduli inductively. An example of a term
is $x^{*}x$ in the language of $\mathrm{C}^{*}$-algebras (to be more precise,
we need to specify the sort in that language, but we will make this more
clear in the examples section).%
\newline Now for formulae:
\begin{itemize}
\item First, we consider atomic formulae: these are defined using relation
symbols, i.e. if $t_{1},\dots ,t_{n}$ are terms of sorts
$S_{1},\dots ,S_{n}$ and $\phi $ is a relation symbol then
$\phi (t_{1},\dots ,t_{n})$ is a formula.
\item Connectives are just continuous functions from
$\mathbb{R}^{n}$ to $\mathbb{R}$, so if $f$ is such a function and
$t_{1},\dots ,t_{n}$ are terms, then $f(t_{1},\dots ,t_{n})$ is a formula.
\item Finally, we consider quantifiers: if $\phi $ is a formula and
$x^{S}_{i} \in \mathrm{FV}(\phi )$, then
$\mathrm{Sup}_{x^{S}_{i} \in S} \phi $ and
$\mathrm{Inf}_{x^{S}_{i} \in S} \phi $ are both formulae.
\end{itemize}

A formula with no free variable is called a sentence. Again, formulae inherit
uniform continuity moduli by their inductive construction. An example of
such formulae in the language of $\mathrm{C}^{*}$-algebras would be:
$x^{*}x +2$, $x^{*}yx$, $\mathrm{Sup}_{x} x^{*}x +y^{*}y$ \dots

\begin{note*}
\normalfont Free variables of a formula are defined the same way as in
the case of regular model theory.
\end{note*}

\subsubsection*{Structures and models}

An $\mathfrak{L}$-structure is a triplet
$M=\langle \mathcal{S},\mathcal{F},\mathcal{R} \rangle $, such that
\begin{itemize}
\item For each sort symbol in $S \in \mathfrak{S}$, there
is a complete metric space $M^S$ in the set $\mathcal{S}$ bounded by $k_{S}$.
\item For each element $f \in \mathfrak{F}$,there is a function
$f^{M}$ in $\mathcal{F}$, such that if the formal domain of $f$ is
$(S_{1},\dots ,S_{n})$, and the formal range is $S^{\prime}$, then its interpretation
$f^{M}$ has domain $M^{S_{1}}\times \dots \times M^{S_{n}}$ and range
$M^{S^{\prime}}$, also $f^{M}$ is uniformly continuous with uniform continuity
modulus $\delta _{f}$.
\item  For each element $\phi \in \mathfrak{R}$, we have
a relation $\phi ^{M} \in \mathcal{R}$ such, if the formal domain of
$\phi $ is $(S_{1},\dots , S_{n})$ and the formal range is $B$ a compact
interval of $\mathbb{R}$, then the interpretation $\phi ^{M}$ is a function
with domain $M^{S_{1}}\times \dots \times M^{S_{n}}$ and with range
$B$, which is uniformly continuous with uniform continuity modulus
$\delta _{\phi}$. In the same manner, we interpret terms and formulae.
\end{itemize}

Now let $M$ be an $\mathfrak{L}$-structure, and let $\mathbb{T}$ be a set
of sentences in the language $\mathfrak{L}$, then we say that $M$ is a
model of $\mathbb{T}$ if for every $\psi \in \mathbb{T}$,
$\psi ^{M}=0$, and in this case we write $M \models \mathbb{T}$. We say
that $\mathbb{T}$ is consistent if it has a model. Notice that if we take
$\mathbb{T}=\varnothing $, then its models in this case are exactly
$\mathfrak{L}$-structures.

Let
$\mathfrak{L}=\langle \mathfrak{S},\mathfrak{F},\mathfrak{R} \rangle $
be a Language (or signature, or similarity type), and let
$\mathbb{T}$ be a family of sentences in the language $\mathfrak{L}$, we
are going to denote by $\mathrm{CompMet}_{\mathfrak{L}}$ the category of
structures of $\mathfrak{L}$ and by
$\mathrm{CompMet}_{\mathfrak{L},\mathbb{T}}$ the full subcategory of models
of $\mathbb{T}$. To be more precise, we should specify what is
a morphism in this category: let $M$ and $N$ be two models, %%tvarkyta
then a morphism of models $g$ is a family of morphisms $g^{S}$ for each
sort (we will omit the superscript if the context is clear) such
that for every function symbol $f$ with domain
$(S_{1},\dots ,S_{n})$ and with range $S^{\prime}$, we have for every
$(a_{1},\dots ,a_{n})\in M^{S_{1}}\times \dots \times M^{S_{n}}$,
$f^{N}(g^{S_{1}}(a_{1}), \dots ,g^{S_{n}}(a_{n})) =g^{S^{\prime}}(f^{M}(a_{1},
\dots ,a_{n}))$. And for every relation symbol, with domain
$(S_{1},\dots ,S_{n})$, we have
\begin{equation*}
\phi ^{N}(g^{S_{1}}(a_{1}),\dots ,g^{S_{n}}(a_{n})) \leq \phi ^{M}(a_{1},
\dots ,a_{n}).
\end{equation*}

One important particular case of this is when we have only one sort
$S$ and only one relation (the distance relation on this sort); in this
case, we get a category equivalent to the category of complete metric spaces
bounded by a certain $k$ with contractions as morphisms, which we denoted
by $\mathsf{k}\text{-}\mathrm{CompMet}$.

\subsubsection*{Ultraproducts and models}

In all the previous cases, the ultraproduct construction given explicitly
in \cite{farah2021model} and \cite{hart2023introduction}, makes these categories
ultracategories. We think it's important to highlight this construction,
which is similar to the ultraproduct construction in usual model theory
(after all these are just directed colimits of products). Of course, we
assume that the reader is at this point familiar with the ultraproduct
of metric spaces bounded by a certain constant.

Suppose we have a similarity type
$\mathfrak{L}=\langle \mathfrak{S},\mathfrak{F},\mathfrak{R} \rangle $,
and a family of structures $(V_{i})_{i \in I}$ of that similarity type,
we define their ultraproduct as follows:
\begin{itemize}
\item For each sort $S \in \mathfrak{S}$ we define
$(\int _{I} V_{i} d\mu )^{S} $ by
$(\int _{I} V_{i} d\mu )^{S}= \int _{I} V_{i}^{S} d
\mu $.
\item For a relation symbol $\phi$ with domain $ S_1 \times \dots \times S_n$ , we define
$ \phi^{{\int _{I} V_{i} d\mu}}((a_{i}^{m})_{i \in I})_{1\leq m \leq n} =
\lim _{\mu}(\phi^{V_{i}}(a_{i}^{m}))$. Here $\lim _{\mu}$ is the
ultralimit in $[-k,k]$ (remember that the family
$(\phi^{V_{i}}(a_{i}^{m}))$ is bounded so we can replace
$(-\infty,\infty )$ by $[-k,k]$ and define this as the limit of the push forward
of the ultrafilter $\mu $ by the map
$(a_{i}^{m})_{i \in I} \mapsto \phi^{V_{i}}(a_{i}^{m}) _{1 \leq m \leq n}$).
\item For a function symbol, things are the same as in usual model theory.
That means that for $f \in \mathfrak{F}$ with domain $ S_1 \times \dots \times S_n$, we define
$f^{{\int _{I}(V_{i})d\mu}}(((a^{m} _{i})_{i \in I}) _{1 \leq m \leq n})=(f^{V_i}(a^{m} _{i})_{1 \leq m \leq n})_{i \in I}$.
\end{itemize}

The fact that models are closed under taking this construction above follows
from {\L}os theorem; {\L}os theorem is an
important result in classical model theory which has a version in continuous
model theory:

%t5.1 #&#
\begin{theorem}[{\L}os theorem in continuous model theory]
\label{thm5.1}
For any family of structures $\{M_{x}\}_{x \in X}$ if $\mu $ is an ultrafilter on $X$, if we call
$M=\int _{X}M_{x} d\mu $ then we have the following: for any formula
$\phi $ and any $\bar{m} =(m_{x})_{x \in X}$ we have
$\phi ^{M}(\bar{m})=\int _{X}\phi ^{M_{x}}(m_{x})d\mu $.
\end{theorem}

We see clearly that a similar version of the classical {\L}os theorem is
a consequence of the theorem above if $\phi $ is a sentence and if for
every $x \in X$ $M_{x} \models \phi $ (which is the same thing as saying
that $\phi ^{M_{x}}=0$), then $\int _{X} M_{x} d\mu \models \phi $.

%s5.2 #&#
\subsection{The ultracategory of Banach spaces is not a category of models of geometric logic}
%%LEAP%%%\label{sec5.2}
\label{not_coherent_VTEX1}

Most of the results in this subsection and their proofs are due to Simon
Henry.

Let us introduce the following notations:

Let $\mathcal{M}$ be an ultracategory of models of continuous logic, and let
$A \in \mathcal{M}$, then by an element $a \in A$, we mean a tuple
$(a_{S})_{S \in \mathfrak{S}}$. We remind that a morphism $f$ from
$A$ to $B$, is given by a family of contractions
$(f^{S})_{S \in \mathfrak{S}}$; this allows us to equip
$\mathrm{Hom}(A,B)$ with the topology of pointwise convergence, where
$(f_{i})$ converges to $f$ iff for every $a \in A$ and every
$S \in \mathfrak{S}$, $(f^{S}_{i}(a_{S})) \to f^{S}(a_{S})$.

Suppose that we have an ultracategory $\mathcal{M}$, and let
$A \in \mathcal{M}$, and let $\mu $ be an ultrafilter on some set
$I$, then we can define the diagonal map $\Delta _{A,\mu}$ from $A$ to
$\int _{I}Ad\mu $ using the data of the ultracategory as follows (denoting
by $*$ simultaneously the one point set and the unique ultrafilter on it):
\begin{equation*}
\begin{tikzcd}
A &&& {\int _{*}Ad*=\int _{*}A d\int _{I}*d\mu} &&& {\int _{I}\int _{*}Ad*d
\mu} &&& {\int _{I} A d\mu}
\arrow["{\epsilon _{*,*}}", from=1-1, to=1-4]
\arrow["{\Delta _{\mu ,*_{\bullet}}}", from=1-4, to=1-7]
\arrow["{\int _{I} \epsilon ^{-1}_{*,*}d\mu}", from=1-7, to=1-10]
\end{tikzcd}
\end{equation*}

Now we state an important theorem regarding categories of models of continuous
logic:

%t5.2 #&#
\begin{theorem}
%%LEAP%%%\label{thm5.2}
\label{first_lemma_VTEX1}
Let $\mathcal{M}$ be a category of models of a continuous first order theory,
and let $F: \mathcal{M} \to \mathsf{Set}$ be a left ultrafunctor. Let $A$ and $B$ be two objects
of $\mathcal{M}$ and $(f_{i})_{i \in I}$ a collection of morphisms from
$A$ to $B$. Let $\mu $ be an ultrafilter on $I$, such that
$f_{i}(x) \xrightarrow \mu f(x)$ for all $x \in A$ (in a more simple language,
this means that for every $S$ the ultrafilter
$ f^{S}_{\bullet}(x) \mu $ converges to $f(x)$), then
$\forall x \in F(A)$, $\exists I' \in \mu $ such that
$\forall i \in I'$, $F(f_{i})(x)=F(f)(x)$.
\end{theorem}

\begin{proof}
In $\mathcal{M}$, the composites:
\begin{equation*}
\begin{tikzcd}
A &&& {\int _{I} A d\mu} &&& {\int _{I} B d\mu}
\arrow["{\Delta _{A,\mu}}"', from=1-1, to=1-4]
\arrow["{\int _{I}f_{i} d\mu}", shift left, curve={height=-18pt}, from=1-4, to=1-7]
\arrow["{\int _{I}f d\mu}", shift right, curve={height=18pt}, from=1-4, to=1-7]
\end{tikzcd}
\end{equation*}
coincide by the assumption that $f_{i}(x) \xrightarrow{\mu} f(x)$.

So we get the following commutative diagrams in $\mathsf{Set}$:
\begin{equation*}
\begin{tikzcd}
{F(A)} & {F\left (\int A d\mu \right )} & {F\left (\int B d \mu
\right )}
\\
{\int F(A) d\mu} & {\int F(A) d \mu} & {\int F(B)d \mu}
\arrow["{F(\Delta _{A,\mu})}", from=1-1, to=1-2]
\arrow["{\Delta _{F(A),\mu}}", from=1-1, to=2-1]
\arrow["{F(f_{i})}", shift left, from=1-2, to=1-3]
\arrow["{F(f)}"', shift right, from=1-2, to=1-3]
\arrow["{\sigma _{\mu}}", from=1-2, to=2-2]
\arrow["{\sigma _{\mu}}", from=1-3, to=2-3]
\arrow[equals, from=2-1, to=2-2]
\arrow["{\int F(f_{i})}", shift left, from=2-2, to=2-3]
\arrow["{\int F(f)}"', shift right, from=2-2, to=2-3]
\end{tikzcd}
\end{equation*}

The commutativity of this diagram in $\mathsf{Set}$ means exactly that
$\forall x \in F(A)$, the $F(f_{i})$ are $\mu $-almost everywhere equal
to $F(f)$.

Of course we should justify why the following diagram is commutative:
\begin{equation*}
\begin{tikzcd}
{F(\int _{I} A d\mu )} && {F(A)}
\\
\\
{\int _{I} F(A) d\mu} \arrow["{\sigma _{\mu}}", from=1-1, to=3-1]
\arrow["{F(\Delta )}"', from=1-3, to=1-1]
\arrow["\Delta "{description}, from=1-3, to=3-1]
\end{tikzcd}
\end{equation*}

This looks like a natural property of left ultrafunctors, but it's not
very evident from the axioms; in order to show it we use a combination
of the axioms of the definition of left ultrafunctors, and of course
we should rewrite the diagonal map in terms of the data of ultracategories;
let us look at the following diagram:
\begin{equation*}
\begin{tikzcd}
{\int _{I}F(A)d\mu } & {\int _{I}\int _{*}F(A)d*d\mu} & {\int _{*}F(A)d*}
\\
&& {F(A)}
\\
& {\int _{I} F(\int _{*}Ad*)d\mu}
\\
{F(\int _{I}Ad\mu )} & {F(\int _{I}\int _{*}Ad*d\mu )} & {F(\int _{*}Ad*)}
\arrow["1"{pos=0.2}, draw=none, from=1-1, to=4-2]
\arrow["{{\int _{I} \epsilon d\mu}}"', from=1-2, to=1-1]
\arrow["2"{description, pos=0.2}, draw=none, from=1-2, to=4-3]
\arrow["{{\Delta _{I,*}}}"', from=1-3, to=1-2]
\arrow["{{\epsilon ^{-1}}}"', from=2-3, to=1-3]
\arrow[""{name=0, anchor=center, inner sep=0}, "{{F(\epsilon ^{-1})}}", from=2-3, to=4-3]
\arrow["{{\int_I  \sigma _{*} d\mu}}"{description}, from=3-2, to=1-2]
\arrow["{{\sigma _{\mu}}}", from=4-1, to=1-1]
\arrow["{{\sigma_{\mu}}}"{description}, from=4-2, to=3-2]
\arrow["{{F(\int_I \epsilon d \mu )}}", from=4-2, to=4-1]
\arrow[""{name=1, anchor=center, inner sep=0}, "{{\sigma _{*}}}"{description}, curve={height=-24pt}, from=4-3, to=1-3]
\arrow["{{F(\Delta _{I,*})}}", from=4-3, to=4-2]
\arrow["3"{description}, draw=none, from=1, to=0]
\end{tikzcd}
\end{equation*}

Squares $1$ and $2$ commute by axiom $2$ of definition of left ultrafunctors \cite[Definition 1.4.1]{lurie2018ultracategories} (note that for square $1$ one would also need to use \cite[Corollary $1.3.6$]{lurie2018ultracategories}),
while triangle $3$ commutes by axiom $1$ of the same definition.
\end{proof}

A consequence of this theorem is the following lemma:

%l5.1 #&#
\begin{lemma}
\label{lem5.1}
Let $\mathcal{M}$ be a category of models of continuous logic and let
$A,B \in \mathcal{M}$. If $x \in F(A)$ and $y \in F(B)$, the set of
$g: A \rightarrow B$ such that $F(g)(x)=y$ is open in
$\mathrm{Hom}(A, B)$ for the topology of pointwise convergence.
\end{lemma}

\begin{proof}
Let $U_{x,y}$ be this set. Theorem~\ref{first_lemma_VTEX1} shows that if
$f_{i} \in \operatorname{Hom}(A,B)$ converge along $\mu $ to
$f \in U_{x,y}$, then $f_{i} \in U_{x,y}$, $\mu $-almost everywhere. This
means $U_{x,y}$ is open by ultrafilter characterisation of open sets.
\end{proof}

We can use this lemma to show the following theorem:

%t5.3 #&#
\begin{theorem}
%%LEAP%%%\label{thm5.3}
\label{not_coherent_theorem_VTEX1}
Let $\mathcal{M}$ be an ultracategory of models of a continuous first order
theory, that satisfies the following two conditions:
\begin{enumerate}
\item For any $A ,B \in \mathcal{M}$, $\mathrm{Hom}(A,B)$ is connected
with the topology of pointwise convergence.
\item The category $\mathcal{M}$ has a zero object.
\end{enumerate}
then the only left ultrafunctors $\mathcal{M} \to \mathsf{Set}$ are constants.
\end{theorem}

\begin{proof}
In what follows we are going to denote by $0$ the unique morphism in and
out of the zero object of $\mathcal{M}$. The proof requires the following
lemma:

%l5.2 #&#
\begin{lemma}
\label{lem5.2}
$\forall f, g: A \rightarrow B$ in $\mathcal{M}$, $F(f)=F(g)$.
\end{lemma}
\begin{proof}
Fix $x \in F(A)$. The sets
$U_{x,y}=\{g: A \rightarrow B \mid F(g)(x)=y\}$ for $y \in F(B)$ form a
partition of $\operatorname{Hom}(A,B)$ into open sets. As
$\mathrm{Hom}(A, B)$ is connected by assumption, it means only one of them
is non-empty, so $\exists ! y$ such that
$\forall g: A \rightarrow B$, $F(g)(x)=y$. This proves that all
$f \in \operatorname{Hom}(A,B)$ have the same image by $F$.
\end{proof}

Let $B \in \mathcal{M}$. The maps $0: 0 \rightarrow B$ and
$0: B \rightarrow 0$ induce functions $F(0) \longrightarrow F(B)$ and
$F(B) \longrightarrow F(0)$. The composite
$F(0) \rightarrow F(B) \rightarrow F(0)$ is the identity of $F(0)$; if
we call $0_{B}$ the composition $B \rightarrow 0 \rightarrow B$, then
$F(B) \rightarrow F(0) \rightarrow F(B)$ is $F(0_{B})$. But by the previous
lemma, $F(0_{B})=F(\mathrm{id}_{B})=\mathrm{id}_{F(B)}$. Now let us look at the following diagram:
\begin{equation*}
\begin{tikzcd}
{F(0)} && {F(0)} \\
& {F(B)} && {F(B)}
  \arrow["{{\mathrm{id}}_{F(0)}}", from=1-1, to=1-3]
  \arrow["{F(0)}", from=1-1, to=2-2]
  \arrow["{F(0)}", from=1-3, to=2-4]
  \arrow["{F(0)}"{description}, from=2-2, to=1-3]
  \arrow["{{\mathrm{id}}_{F(B)}}"{description}, from=2-2, to=2-4]
\end{tikzcd}
\end{equation*}
%\begin{equation*} %virsuje formule taisyta su claude
%%
%\begin{tikzcd}
%{F(0)} && {F(0)} &
%\\
%& {F(B)} && {F(B)} \arrow["{{\mathrm{id}}_{F(0)}}", from=1-1, to=1-3]
%\arrow["{F(0)}", from=1-1, to=2-2] \arrow["{F(0)}", from=1-3, to=2-4]
%\arrow["{F(0)}"{description}, from=2-2, to=1-3]
%\arrow["{{\mathrm_{id}}_{F(B)}}"{description}, from=2-2, to=2-4]
%\end{tikzcd}
%%
%\end{equation*}

Using this diagram $F(0):F(B) \rightarrow F(0)$ is  an isomorphism, so we get an isomorphism $F(B) \simeq F(0)$.

Now on morphisms, for any map $f: B \rightarrow A$ in $\mathcal{M}$, we have:
\begin{equation*}
\begin{tikzcd}
F(B) \arrow[r, "F(f)"] \arrow[d, "F(0)"'] & F(A) \arrow[d, "F(0)"]
\\
F(0) \arrow[r, "\mathrm{id}_{F(0)}"'] & F(0)
\end{tikzcd}
\end{equation*}

This makes $F(B) \longrightarrow F(A)$ an isomorphism, and this makes
$F$ isomorphic to a constant functor.

Finally, we need to show that the ultrastructure of $F$ is that of a constant
functor, which is equivalent to saying that up to the respective isomorphisms,
the maps $\sigma _{\mu}$ are diagonal maps. Our goal is to show that the
following diagram is commutative:
\begin{equation*}
\begin{tikzcd}
{F(\int _{I}A_{i}d\mu )} &&& {F(0)}
\\
\\
\\
{\int _{I} F(A_{i})d\mu} &&& {\int _{I}F(0)d\mu}
\arrow["{F(0)}"', from=1-1, to=1-4]
\arrow["{\sigma _{\mu}}", from=1-1, to=4-1]
\arrow["\Delta _{F(0),\mu}"', from=1-4, to=4-4]
\arrow["{\int _{I} F(0) d\mu}", from=4-1, to=4-4]
\end{tikzcd}
\end{equation*}

But we already have the commutativity of the following diagram:
\begin{equation*}
\begin{tikzcd}
{F(\int _{I}A_{i}d\mu )} && {F(\int _{I} 0 d\mu )}
\\
\\
{\int _{I}F(A_{i})d\mu} && {\int _{I} F(0) d\mu}
\arrow["{F(\int _{I} 0 d\mu )}", from=1-1, to=1-3]
\arrow["{\sigma _{\mu}}"', from=1-1, to=3-1]
\arrow["{\sigma _{\mu}}", from=1-3, to=3-3]
\arrow["{\int _{I} F(0) d\mu}"', from=3-1, to=3-3]
\end{tikzcd}
\end{equation*}

So we can restrict our attention to
$F(\int 0d\mu ) \xrightarrow {\sigma _{\mu}} \int _{I} F(0) d\mu $, and
to attempt to show that the following diagram commutes:
\begin{equation*}
\begin{tikzcd}
{F(0)} &&& {F(\int _{I}0 d\mu )}
\\
\\
\\
{\int _{I}F(0) d\mu} \arrow["{F(0)=F(\Delta )}"', from=1-1, to=1-4]
\arrow["\Delta ", from=1-1, to=4-1]
\arrow["{\sigma _{\mu}}"', from=1-4, to=4-1]
\end{tikzcd}
\end{equation*}

But we have already shown this in the proof of \ref{first_lemma_VTEX1}.
\end{proof}

This in particular shows that Banach spaces, Hilbert spaces with their
usual ultraproducts are not axiomatisable in coherent logic, using conceptual
completeness (in fact they are not even axiomatisable in geometric logic,
but this requires generalised ultracategories
\cite{hamad2025generalisedultracategoriesconceptualcompleteness}).

%s5.3 #&#
\subsection{Bundles of structures}
\label{sec5.3}

We define what it means to be a bundle of structures:

%d5.1 #&#
\begin{definition}
%%LEAP%%%\label{defn5.1}
\label{sorted_bundle_definition_VTEX1}
Let
$\mathfrak{L}=\langle \mathfrak{S},\mathfrak{F},\mathfrak{R}\rangle $ be
a language, we define a bundle of structures $E$ of that language with
base space $X$, to be a family of bundles of complete bounded metric spaces
$((E_{S},\pi _{S}))_{S \in \mathfrak{S}}$, such that for any $x \in X$,
$(\pi _{S}^{-1}(x))_{S \in \mathfrak{S}}$ is a structure of the language
$\mathfrak{L}$ (so in particular it comes with the function and relation
symbols data), such that the following axioms are satisfied:

For any function symbol $f$ with formal domain $\mathrm{dom}(f)$ and formal
range $\mathrm{rng}(f)$ and for any relation symbol $\phi $ with formal
domain $\mathrm{dom}(\phi )$ we are going to denote by $f^{E}$ and
$\phi ^{E}$ the global function and relations respectively (so for any
$x \in X$ $f^{E}$ restricts to the interpretation of the function symbol
$f$ of $(\pi _{S}^{-1}(x))_{S \in \mathfrak{S}}$, same thing for relation
symbols).
\begin{itemize}
\item Axiom(1): For every function symbol $f$, the map $f^{E}$ is continuous.
\item Axiom(2): For every relation symbol $\phi $, the map
$\phi ^{E}$ is upper semi-continuous.
\end{itemize}
\end{definition}
\begin{note*}
\normalfont
We are going to denote by $E^{S}$ the bundle of structures that corresponds
to a sort $S$ and by $E_{x}$ the fibre over $x$ which is a structure, so
following this convention $E^{S}_{x}$ is the $x$-th fibre of the bundle
of structure corresponding to the sort $S$.
\end{note*}
%

%s5.4 #&#
\subsection{Maps of bundles}
\label{sec5.4}

Let $E$ and $E^{\prime}$ be two bundles, a morphism $\psi $ in the category
of bundles consists of the following:

For each sort $S$, a map of bundles of bounded metric spaces
$\psi ^{S}$ between the bundles $E^{S}$ and ${E^{\prime}}^{S}$ such that the
following diagram commutes (in $\mathsf{Top}$):
\begin{equation*}
\begin{tikzcd}
E^{S} \arrow[rr, "\psi ^{S}"] \arrow[rdd, "\pi _{S}"'] & & {E^{\prime}}^{S}
\arrow[ldd, "\pi _{S}"]
\\
& &
\\
& X &
\end{tikzcd}
\end{equation*}
and such that for any $x$, $\psi _{x}$ is a map of structures of the language
$\mathfrak{L}$ from $E_{x}$ to $E^{\prime}_{x}$.

Now, we want to extend the equivalence obtained in section~\ref{3}(Theorem~\ref{first_theorem_VTEX1}) to structures of continuous model theory:%
\newline In other words, we want to show the following:
%
%t5.4 #&#
\begin{theorem}
%%LEAP%%%\label{thm5.4}
\label{second_theorem_VTEX1}
Let $X$ be a compact Hausdorff space, then there is an equivalence of categories
between
\newline $\mathrm{Lult}(X,\mathrm{CompMet}_{\mathfrak{L}})$ and the category
$\mathrm{Bun}(\mathrm{CompMet}_{\mathfrak{L}},X)$.
\end{theorem}

The rest of section~\ref{5} is devoted to showing the theorem above \ref{second_theorem_VTEX1}.

In order to define this equivalence of categories, we are going to expand
the definitions of the functors $\mathcal{L}$ and $\mathcal{R}$ already
defined to the categories above.

%s5.5 #&#
\subsection{The functor $\mathcal{L}$ on $\mathrm{Lult}(X,\mathrm{CompMet}_{\mathfrak{L}})$}
\label{sec5.5}

Let $X$ be a compact Hausdorff space, and suppose we have
$\mathcal{F}$, a left ultrafunctor from $X$ to the ultracategory of structures
of some language $\mathfrak{L}$. We know that the functor
$\mathcal{F}$ will give rise to a family of functors
$\mathcal{F}^{S}$ for each sort $S$. If we define each mono-sorted bundle
$E_{S}$ to be $\coprod _{x \in X}\mathcal{F}^{S}(x)$ with its bundle topology
given in \ref{topology}, then we have already seen that the first three
axioms are satisfied for this multi-sorted bundle.
\newline
It remains to check axioms $1$ and $2$:

\subsubsection*{Axiom {$1$} (functions)}

To prove that axiom $1$ is satisfied by our definition of multi-sorted
bundle suppose that $f$ is a function symbol, and suppose that
$\mathrm{dom}(f)=S_{1} \times \dots \times S_{n}$ and
$\mathrm{rng}(f)=S^{\prime}_{1}$. Suppose that $\mu $ is an ultrafilter on
$E^{S_{1}} \times _{X}\dots \times_{X} E^{S_{n}}$ (in the case where we have a constant
symbol this space is $X$ the $0$-th product in $\mathsf{{Top}}/\mathsf{X}$) that converges
to
$(a^{1},\dots ,a^{n}) \in M_{y}^{S_{1}} \times \dots \times M_{y}^{S_{n}}
\subseteq E^{S_{1}} \times _{X}\dots \times _{X}E^{S_{n}}$, and suppose
that $f^{M_{y}}(a^{1},\dots ,a^{n})=a^{\prime}$. Now suppose that for each
$S_{i}$
$\sigma _{\pi _{S_{i}}\mu}^{S_{i}}(a^{i}) = (b_{x}^{i})_{x \in X}$. Since
$\sigma _{\pi _{S_{i}} \mu}$ is a map of $\mathfrak{L}$ structures we get
that
$\sigma _{\pi _{S_{i}}\mu}^{S^{\prime}_{1}}(a^{\prime})=f^{M}((b_{x}^{1}),
\dots ,(b_{x}^{n}))$ so we may use
$(f^{M_{x}}(b_{x}^{1}, \dots , b_{x}^{n}))_{x \in X}$ as representative
of the class of $\sigma _{\pi _{S'_{1}} \mu}^{S_{i}}(a^{\prime})$ (using the
definition of the structure of the ultraproduct) (in the case of constant
symbol $c$ of sort $S^{\prime}_{1}$ we use $(c_{x})_{x \in X}$ as representative
of its class).

Now let $\epsilon >0$. We know that for any $x$, $f^{M_{x}}$ is uniformly
continuous with uniform continuity modulus independent of $x$, thus we
can deduce that there exists some $\delta $, such that if
$d(m_{x}^{i},b_{x}^{i}) < \delta $, we get that
$d(f^{M_x}(m_{x}^{1},\dots ,m_{x}^{n}),f^{M_x}(b_{x}^{1},\dots ,b_{x}^{n}))<
\epsilon $. We want to show that $f^{M_{y}} \mu $ converges to
$a^{\prime}$: We have that
$\coprod _{x \in X}B(b_{x}^{i},\delta ) \in \pi _{S_{i}}\mu $. Now take
the following set
$\bigcap _{i=1}^{n}\pi _{S_{i}}^{-1}(\coprod _{x \in X}B(b_{x}^{i},
\delta )) \in \mu $. If we take
$(l_{1},\dots ,l_{n}) \in \bigcap _{i=1}^{n}\pi _{S_{i}}^{-1}(
\coprod _{x \in X}B(b_{x}^{i},\delta ))$, and suppose that
$\pi _{S_{i}}(l_{i})=z$ we have the following:
\begin{equation*}
d(f^{E}(l_{1},\dots ,l_{n}),f^{E}(b_{z}^{1},\dots ,b_{z}^{n}))=d(f^{M_{z}}(l_{1},
\dots ,l_{n}), f^{M_{z}}(b_{z}^{1},\dots ,b_{z}^{n}))<\epsilon ,
\end{equation*}
then this set satisfies
\begin{equation*}
\bigcap _{i=1}^{n}\pi _{S_{i}}^{-1}(\coprod _{x \in X}B(b_{x}^{i},
\delta )) \subseteq (f^{E})^{-1}(\coprod _{x \in X}B(f^{M_x}(b_{x}^{1},
\dots ,b_{x}^{n}),\epsilon )),
\end{equation*}
thus we get that
\begin{equation*}
\coprod _{x \in X}B(f^{M_x}(b_{x}^{1},\dots ,b_{x}^{n}),\epsilon ) \in f^{E}
\mu ,
\end{equation*}
thus $f^{E}$ is continuous (In the case we have a constant symbol we have
that $\coprod _{x \in X}B(c_{x},\epsilon ) \in c^{E} \mu $ trivially since
$(c^{E})^{-1}\coprod _{x \in X}B(c_{x},\epsilon ) =X \in \mu $).

\subsubsection*{Axiom {$2$} (relations)}

We are going to denote $\mathcal{F}(x)$ by $M_{x}$. We want to prove that
the family of $S$-bundles for $S \in \mathfrak{S}$ satisfies the upper
semi-continuity for each global relation. To do so suppose that
$\phi $ is a relation symbol, and suppose that $\mu $ is an ultrafilter
on $X$ that converges to $y$ and that
$\mathrm{dom}(\phi )=S_{1} \times \dots \times S_{n}$. From this point
forward let us denote by $M$ the ultraproduct
$\int _{X} M_{x} d\mu $.

Let us prove that $\phi ^{E}$ is upper semi-continuous: Let $\mu $ be an
ultrafilter on $E^{S_{1}}\times _{X}\dots \times _{X}E^{S_{n}}$ such that
$\mu $ converges to
$(a^{1},\dots ,a^{n}) \in M_{y}^{S_{1}}\times \dots \times M_{y}^{S_{n}}
\subseteq E^{S_{1}}\times _{X} \dots \times _{X}E^{S_{n}}$ (for some
$y \in X$) and take $r>0$ such that
$\phi ^{M_{y}}(a^{1},\dots ,a^{n})<r$. Let us call the quantity
$r- \phi ^{M_{y}}(a^{1},\dots ,a^{n})=\epsilon $. Notice that for any
$i,j$ we get that $\pi _{S_{i}} \mu =\pi _{S_{j}} \mu $ is the same ultrafilter
on $X$, so we'll call this ultrafilter $\pi _{S_{i}}\mu $ regardless of
which $i$ this ultrafilter comes from. Since
$\sigma _{ \pi _{S_{i}} \mu}$ is a morphism of $\mathfrak{L}$-structures,
then calling
$\sigma _{\pi _{S_{i}} \mu}^{S_{i}}(a^{i}) = (b_{x}^{i})_{x \in X}$, we
get that
$\phi ^{M}((b_{x}^{1}),\dots ,(b_{x}^{n})) \leq \phi ^{M_{y}}(a^{1},
\dots ,a^{n})$, thus for any $\epsilon ^{\prime}>0$ there exists
$L \in \mu $ such that for every $x \in L$ we have
$\phi ^{M_{x}}(b_{x}^{1},\dots ,b_{x}^{n}) \leq \phi ^{M_{y}}(a^{1},
\dots ,a^{n}) +\epsilon ^{\prime}$. So let us pick the $L$ corresponding to
$\epsilon ^{\prime}=\epsilon /2$.

We know that for each $x$, the functions $\phi ^{M_{x}}$ are uniformly
continuous with the same uniform continuity modulus (independent of
$x$), which implies that there exists some $\delta $ such that for any
$m^{i}_{x} \in M_{x}^{S_{i}}$, if $d(m^{i}_{x},b_{x}^{i})<\delta $, we
have
$|\phi ^{M_{x}}(b_{x}^{1},\dots ,b_{x}^{n})-\phi ^{M_{x}}(m^{1}_{x},
\dots ,m^{n}_{x})|< \epsilon /2$. Let us take the set
$\bigcap _{i=1}^{n}\pi _{S_{i}}^{-1}(\coprod _{x \in L}B(b_{x}^{i},
\delta ))$. First, we know that each
$\coprod _{x \in L}B(b_{x}^{i},\delta ) \in \pi _{S_{i}}\mu $, which allows
us to deduce that
$\bigcap _{i=1}^{n}(\pi _{S_{i}}^{-1}\coprod _{x \in L}B(b_{x}^{i},
\delta )) \in \mu $. Suppose that
$(l^{1},\dots ,l^{n}) \in \bigcap _{i=1}^{n}\pi _{S_{i}}^{-1}(
\coprod _{x \in L}B(b_{x}^{i},\delta ))$, let us call
$z=\pi _{S_{i}}(l^{i})$ then we have that:
\begin{equation*}
\phi ^{M_{z}}(l^{1},\dots ,l^{n}) <\phi ^{M_{z}}(b_{z}^{1},\dots ,b_{z}^{n})
+ \epsilon /2 \leq \phi ^{M_{y}}(a^{1},\dots ,a^{n}) + \epsilon /2+
\epsilon /2 =r,
\end{equation*}
this implies that
\begin{equation*}
\bigcap _{i=1}^{n}\pi _{S_{i}}^{-1}(\coprod _{x \in L}B(b_{x}^{i},
\delta )) \subseteq (\phi ^{E})^{-1}([0,r)).
\end{equation*}
Thus $(\phi ^{E})^{-1}([0,r)) \in \mu $. Thus, we may deduce that
$\phi ^{E} \mu $ converges to $\phi^{M_y} (a^{1},\dots ,a^{n})$ (if we equip
$[0,\infty ]$ with the left order topology), thus $\phi ^{E}$ is upper
semi-continuous.

\subsubsection*{Functoriality of $\mathcal{L}$}

Since each $\mathcal{L^{S}}$ is a functor by the previous construction,
we may deduce that $\mathcal{L}$ defined this way is a functor.

%s5.6 #&#
\subsection{The inverse functor { $\mathcal{R}$: extending the definition}}
\label{sec5.6}

Suppose we have a bundle $E$ of structures, we define the inverse functor
by sending a bundle $E$ to the left ultrafunctor $\mathcal{F}(E)$ defined
as follows: for every $x \in X$ we define $\mathcal{F}(x)= E_{x}$ (the
fibre at $x$). Now the left ultrastructure of the functor
$\mathcal{R}(E)$ is constructed from the left-ultrastructure of the restriction
of the functor to each sort as described in \ref{sigma}.

Now it remains to check compatibility for both function and relation symbols,
which means that we are going to show that the $\sigma _{\mu}$ constructed
sort-wise is really a morphism in the category of structures.

\subsubsection*{Compatibility of function symbols}

The proof in section~\ref{sigma} shows that for each sort $S$, the maps
$\sigma ^{S}_{\mu}$ are contractions and thus continuous. Suppose that
$\{S_{i}\}_{i=1}^{n}$ is a finite family of sorts. We are going to denote
by $\sigma _{\mu}^{S_{1} \times \dots \times S_{n}}$ the map such that
$\pi _{S_{i}} \circ \sigma _{\mu}^{S_{1} \times \dots \times S_{n}}=
\sigma _{\mu}^{S_{i}}$. Let $\mu $ be an ultrafilter on $X$ that converges
to $y$. As stated before, our goal is to show the compatibility of the
morphism $\sigma _{\mu}$. To do so, suppose that $f$ is a function symbol,
and suppose that $\mathrm{dom}(f)=S_{1} \times \dots \times S_{n}$ and
$\mathrm{rng}(f)=S^{\prime}_{1}$.

Suppose that
$(a_{1},\dots ,a_{n}) \in M_{y}^{S_{1}} \times \dots \times M_{y}^{S_{n}}
$ (in case we have a constant symbol, this space is $X$) and suppose that
for each $i$, the already constructed Cauchy filter converges in
$ \int _{X}M_{x}^{S_{i}} d\mu $ to $(b^{i}_{x})_{x \in X}$ (this means
that $\sigma _{\mu}^{S_{i}}(a_{i})=(b_{x}^{i})_{x \in X}$).

For simplicity, we are going to call the space
$\int _{X}M_{x}^{S_{i}} d\mu =M$. We know that
$f^{M}((b^{1}_{x}),\dots ,(b^{n}_{x}))=(f^{M_{x}}(b^{1}_{x},\dots ,b^{n}_{x}))_{x \in X}$
(by definition), and let us call
$f^{M_{y}}(a_{1},\dots ,a_{n})=a^{\prime}$ and furthermore, we call the limit
of the Cauchy filter corresponding to $a^{\prime}$,
$(a_{x})_{x \in X}$ (this means that
$\sigma _{\mu}^{S^{\prime}_{1}}(a^{\prime})=(a_{x})_{x \in X}$).

Our goal is to show that
$(f^{M_{x}}(b^{1}_{x},\dots ,b^{n}_{x}))_{x \in X}=(a_{x})_{x \in X}$ (as
equivalence classes of the ultraproduct); to do so, let
$\epsilon >0$. Since $f^{M}$ is uniformly continuous with uniform continuity
modulus independent of $x$, then there exists some $\delta $ such that
if for all $i$ $d((b^{i}_{x}),(h^{i}_{x}))<\delta $, we get that
$d(f^{M_x}(b^{1}_{x},\dots,b^{n}_x),f^{M_x}(h^{1}_{x},\dots,h^{n}_{x}))<\epsilon /3$. Now take a neighbourhood
$W$ of $a^{\prime}$ that is $\epsilon /3$-thin. Using Lemma~\ref{Important_Lemma_2_VTEX1}, we know that we can find
$(g_{x})_{x \in X}$, such that there exists some $U \in \mu $ such that
$U \subseteq \pi (W)$, and such that $\forall x\in U$ $g_{x} \in W$ and
$d(g_{x},a_{x})<\epsilon /3$. Now, since $E$ is a bundle, we may deduce
that there exist neighbourhoods $W_{i}$ of each $a_{i}$, such that
$f^{E}(W_{1} \times _{X}\dots \times _{X}W_{n})\subseteq W$ (using the continuity of
$f^{E}$) (in the case of a constant symbol $c$, we deduce the existence
of $W^{\prime}$ neighbourhood of $y$, such that for any $x \in W^{\prime}$,
$c_{x} \in W$). Again using the Lemma~\ref{Important_Lemma_2_VTEX1}, we know
there exist $(V_{i})_{i=1}^{n}$ such that each $V_{i} \in \mu $, and such
that $V_{i} \subseteq \pi _{i}(W_{i})$, and $(e^{i}_{x})$ such that for
any $x \in V_{i}$, we have $d(e^{i}_{x},b^{i}_{x})<\delta $ and
$e^{i}_{x} \in W_{i} $. Now take the set
$\bigcap _{i=1}^{n}V_{i} \bigcap U \in \mu $; for any $x $ in this set
we have
\begin{equation*}
d(f^{M_{x}}(e^{1}_{x},\dots ,e^{n}_{x}),f^{M_x}(b^{1}_{x},\dots ,b^{n}_{x}))<
\epsilon /3.
\end{equation*}
On the other hand, we have
$d(g_{x},f^{M_{x}}(e^{1}_{x},\dots ,e^{n}_{x}))<\epsilon /3$; this follows
from the fact that
$f^{E}(W_{1} \times _{X}\dots \times _{X}W_{n}) \subseteq W$, and that
$W$ is $\epsilon /3$-thin. This implies that for any
$x \in \bigcap _{i=1}^{n}V_{i} \bigcap U $, we have
\begin{equation*}
d(f^{M_{x}}(b^{1}_{x},\dots ,b^{n}_{x}),a_{x}) < \epsilon ,
\end{equation*}
thus
$(f^{M_{x}}(b^{1}_{x},\dots ,b^{n}_{x}))_{x \in X}=(a_{x})_{x \in X}$ as
equivalence classes, thus we get compatibility for each function symbol.
In the case of a constant symbol, it suffices to take
$W^{\prime} \bigcap U$ in the previous argument.

\subsubsection*{Compatibility of relation symbols}

Let $\mu $ be an ultrafilter on $X$ that converges to $y$, and let
$\phi $ be a relation symbol such that
$\mathrm{dom}(\phi )=S_{1}\times \dots \times S_{n}$. Suppose that
$(a_{1},\dots ,a_{n}) \in M_{y}^{S_{1}} \times \dots \times M_{y}^{S_{n}}
$, and suppose that for each $i$ we have the already constructed Cauchy
filter that converges in $\int _{X}M_{x}^{S_{i}} d \mu $ to some
$(b^{i}_{x})_{x \in X}$ (that means that
$\sigma ^{S_{i}}_{\mu}(a_{i})=(b_{x}^{i})_{x \in X}$).

Now our objective is to show that
$\phi ^{M}((b_{x}^{1}),\dots ,(b_{x}^{n})) \leq \phi ^{M_{y}}(a_{1},
\dots ,a_{n})$. To do so we do an argument by contradiction:

Suppose it's not the case, then we have
$\phi ^{M_{y}}(a_{1},\dots , a_{n})< \phi ^{M}((b_{x}^{1}),\dots ,(b_{x}^{n}))$.
Let us call their difference $\epsilon $, and let $\epsilon ^{\prime}= \epsilon/3$. By upper semi-continuity of
$\phi ^{E}$, there exist neighbourhoods $W_{i}$ of $a_{i}$ such that
$\phi ^{E}(W_{1} \times _{X}\dots \times _{X} W_{n}) \subseteq [0,
\phi ^{M_{y}}(a_{1},\dots ,a_{n})+\epsilon')$. Now using the fact that each
$\phi ^{M_{x}}$ is continuous with the same uniform continuity
modulus, we get that there exists $\delta $ such that if for every
$i$, if $d_{M_x^{S_i}}(e^{i}_{x},b^{i}_{x}) <\delta $, we have
$|\phi ^{M_{x}}(e^{1}_{x},\dots,e^{n}_{x})-\phi ^{M_{x}}(b^{1}_{x},\dots,b^{n}_{x})|\ < \epsilon' $. Now using Lemma~\ref{Important_Lemma_2_VTEX1}, we know there exist a family of sets
$\{V_{i}\}_{i=1}^{n}$ such that each $V_{i} \in \mu $, and such that
$V_{i} \subseteq \pi _{i}(W_{i})$, and a family
$(e^{i}_{x})$ such that for any $x \in V_{i}$, we have
$d(e^{i}_{x},b^{i}_{x})<\delta $ and $e^{i}_{x} \in W_{i}$.

Now, we know that there exists $U \in \mu $ such
that for any $x \in U$,\newline
$\phi ^{M}((b^{1}_{x}),\allowbreak \dots ,\allowbreak (b^{n}_{x})) <\allowbreak  \phi ^{M_{x}}(b^{1}_{x},\allowbreak
\dots ,\allowbreak b^{n}_{x}) +\allowbreak  \epsilon ^{\prime} $ (this follows from the fact that \newline
$\phi ^{M}((b^{1}_{x}),\dots ,(b^{n}_{x}))=
\allowbreak
\int _{X} \phi ^{M_{x}}(b_{x}^{1},\allowbreak \dots ,\allowbreak b_{x}^{n})d\mu $ and then we apply
the fact that
$\int _{X} \phi ^{M_{x}}(b_{x}^{1},\dots ,b_{x}^{n})d\mu =$
\newline
$\mathrm{Sup}_{U \in \mu} \mathrm{Inf}_{x \in U}\phi ^{M_{x}}(b_{x}^{1},\allowbreak
\dots ,\allowbreak b_{x}^{n})$). Now take the set
$(\bigcap _{i=1}^{n}V_{i} ) \bigcap U$, for any $x$ in this set, we have\newline
 $|\phi ^{M_x} (b^{1}_{x},\allowbreak \dots ,\allowbreak b^{n}_{x})-\phi ^{M_x} (e^{1}_{x},\allowbreak \dots ,\allowbreak e^{n}_{x})|\allowbreak <\allowbreak
\epsilon' $, but this implies, substituting $\epsilon' $ by its value, that \newline
$\phi^{M_x} (e^{1}_{x},\allowbreak \dots ,\allowbreak e^{n}_{x}) > \phi^{M_y} (a_{1},\allowbreak \dots ,\allowbreak a_{n})+\epsilon'$,  contradiction with $\phi ^{E}(W_{1} \times _{X}\dots \times _{X} W_{n}) \subseteq [0,
\phi ^{M_{y}}(a_{1},\allowbreak \dots ,\allowbreak a_{n})+\epsilon')$.

%s5.7 #&#
\subsection{Adjunction}
\label{sec5.7}

We have already established that for each sort
$\mathrm{Hom}(\mathcal{L}^{S}(\mathcal{F}^{S}),E^{S})\simeq
\mathrm{Hom}(\mathcal{F}^{S},\mathcal{R}^{S}(E^{S}))$; so the only thing
left is to prove that the $\mathrm{Hom}$ functor is compatible with the
structure, in the following sense:

Let us make clear what we exactly mean by the compatibility of the
$\mathrm{Hom}$ functor: We have already established the fact that if we
have a left ultrafunctor $\mathcal{F}$, and a map of bundles $\psi$ from
$\mathcal{L}(\mathcal{F})$ to $E$, then we get a natural transformation
of left ultrafunctors $\widehat{\psi ^{S}}$ for each sort. Also, we have
established that if we have a natural transformation of left ultrafunctors,
we have already seen that for every sort we get a map of bundles
$\bar{\kappa ^{S}}$ from $\mathcal{L}^{S}(\mathcal{F}^{S})$ to $E^{S}$. And we know
that these two processes are inverses of each other at the level of each
sort. So the question is if we can extend this equivalence to the level
of the whole structure.

Suppose we have a left ultrafunctor $\mathcal{F}$ and a map of bundles
$\psi $ from $\mathcal{L}(\mathcal{F})$ to some bundle $E$, then we get
a natural transformation of left ultrafunctors $\widehat{\psi ^{S}}$ for
each sort. So, we define for each $x$ the map $\widehat{\psi}_{x}$ by
$\widehat{\psi}_{x} =\psi _{x}$ from
$\mathcal{L}(\mathcal{F})_{x}=\mathcal{F}(x)$ to
$E_{x}=\mathcal{R}(E)(x)$. Now, the fact that we have a natural transformation
follows from the fact that for every $x$, the map
$\psi _{x}=\widehat{\psi}_{x}$ (by definition) is a map of sorts, and the
fact that it's a natural transformation of left ultrafunctors follows from
the commutativity of this for each ultrafilter on the base space $X$ converging
to arbitrary $y$:
\begin{equation*}
\begin{tikzcd}
\mathcal{F}(y) \arrow[dd, "\psi _{y}"] \arrow[rr, "\sigma _{\mu}"] & &
\int _{X}\mathcal{F}(x)d \mu \arrow[dd, "\int _{X}\psi _{x}d\mu "]
\\
& &
\\
\mathcal{R}(E)(y)=E_{y} \arrow[rr, "\sigma _{\mu}"] & & \int _{X}E_{x}
d\mu
\end{tikzcd}
\end{equation*}
which means exactly that for every sort the following diagram commutes:
\begin{equation*}
\begin{tikzcd}
\mathcal{F}(y)^{S} \arrow[dd, "\psi _{y}^{S}"]
\arrow[rr, "\sigma _{\mu}^{S}"] & & \int _{X}\mathcal{F}(x)^{S} d
\mu \arrow[dd, "\int _{X}\psi _{x}^{S} d\mu "]
\\
& &
\\
\mathcal{R}(E)(y)^{S}=E_{y}^{S} \arrow[rr, "\sigma _{\mu}^{S}"] & &
\int _{X}E_{x}^{S} d\mu
\end{tikzcd}
\end{equation*}
\noindent
which we already showed. So we get that $\widehat{\psi}$ is well-defined.

For the other direction, suppose that we have a natural transformation
of left ultrafunctors $\kappa $ from $\mathcal{F}$ to
$\mathcal{R}(E)$. Define $\bar{\kappa}$ a morphism of bundles
by $(\bar{\kappa}^{S})(a) =(\kappa _{\pi (a)})^{S}(a)$ (reminder that defining
$\bar{\kappa}$ amounts to defining for every sort $S$ a map
$\bar{\kappa}^{S}$ of sorted bundles, such that for every $x$,
$\bar{\kappa}_{x}$ (whose data consists of restricting the various maps
$(\bar{\kappa}^{S})$ to the fibre of $x$) is a map of structures).

From the fact that $\kappa $ is a natural transformation of left ultrafunctors,
we get that for each $x$, $\bar{\kappa}_{x} =\kappa _{x}$ is a map of structures.
The only thing remaining to check is that for each sort,
$\bar{\kappa}^{S}$ is a map of bundles of the corresponding sort, but this
follows immediately from our work for bounded complete metric spaces. Finally,
we know that these two processes are inverses of each other on the level
of each sort, thus they are inverses of each other, and the two functors
$\mathcal{L}$ and $\mathcal{R}$ are adjoints.

Now, the unit and the counit of adjunctions are isomorphisms at the level
of each sort, and hence we get an equivalence of categories between
$\mathrm{Lult}(X,\mathrm{CompMet}_{\mathfrak{L}})$ and the category
$\mathrm{Bun}(\mathrm{CompMet}_{\mathfrak{L}},X)$.

%s6 #&#
\section{Generalising to models}
%%LEAP%%%\label{sec6}
\label{6}

Let $\mathfrak{L}$ be a language and let $\mathbb{T}$ be a theory (set
of sentences) in this language; we define the category of models of
$\mathbb{T}$ to be the category whose objects are $\mathfrak{L}$- structures
that are models of $\mathbb{T}$ (that means for any object $M$ in this
category we have that for any sentence $\phi \in \mathbb{T}$
$\phi ^{M}=0$), and having as morphisms just morphisms of structures.

We can see that the category of models defined this particular way is a
full subcategory of the category of structures, and thus it inherits the
ultrastructure, since it's closed under the ultraproduct functor by {\L}os
theorem. Let us denote by $\mathrm{CompMet}_{\mathfrak{L}}$ the category of structures
of the language $\mathfrak{L}$, and by
$\mathrm{CompMet}_{\mathfrak{L},\mathbb{T}}$ the full subcategory of models of
$\mathbb{T}$. Now we turn to the next important but easy-to-show lemma:
%
%l6.1 #&#
\begin{lemma}
\label{lem6.1}
Let $X$ be a compact Hausdorff space (an ultraset), take the category of
left ultrafunctors from $X$ to $\mathrm{CompMet}_{\mathfrak{L}}$, then left ultrafunctors
from $X$ to $\mathrm{CompMet}_{\mathfrak{L},\mathbb{T}}$ form a full subcategory
of the previous category.
\end{lemma}
\begin{proof}
This follows immediately from the fact that the condition of being a natural
transformation of left ultrafunctors, doesn't depend on whether a functor
$F$ is taking values in $\mathrm{CompMet}_{\mathfrak{L},\mathbb{T}}$ or not.
\end{proof}
The important thing regarding this discussion is that we already know that
the concept of full subcategory is carried over by equivalence of categories,
thus we get an equivalence between left ultrafunctors from $X$ to
the category of models of $\mathbb{T}$, and bundles of structures
whose every fibre is a model of the theory $\mathbb{T}$. This inspires
our next definition:
%
%d6.1 #&#
\begin{definition}
%%LEAP%%%\label{defn6.1}
\label{bundle_of_models_VTEX1}
\noindent
We define a bundle of models of a theory $\mathbb{T}$ in a Language
$\mathfrak{L}$ to be a bundle of structures such that every fibre of the
bundle is a model. As for morphisms of bundles of models, we define them
to be just morphisms of the bundle of structures.
\end{definition}

In other words, we can see that the category
$\mathrm{Bun}(\mathrm{CompMet}_{\mathfrak{L},\mathbb{T}},X)$ is a full subcategory
of the category of bundles of structures
$\mathrm{Bun}(\mathrm{CompMet}_{\mathfrak{L}}, X)$. This definition allows
us to deduce the following theorem:
%
%t6.1 #&#
\begin{theorem}
%%LEAP%%%\label{thm6.1}
\label{third_theorem_VTEX1}
Let $X$ be a compact Hausdorff space, then the functor $\mathcal{L}$ restricts
to an equivalence of categories between
$\mathrm{Lult}(X,\mathrm{CompMet}_{\mathfrak{L},\mathbb{T}})$ and the category
$\mathrm{Bun}(\mathrm{CompMet}_{\mathfrak{L},\mathbb{T}}, X)$.
\end{theorem}

%s7 #&#
\section{Functoriality in $\mathsf{CompHaus}$}
%%LEAP%%%\label{sec7}
\label{7}

Let $\mathsf{CompHaus}$ denote the category of compact Hausdorff spaces.
Let $\mathcal{M}$ be a category of models of continuous model theory. The
category $\mathsf{CompHaus}^{\mathsf{o}}_{\mathcal{M}}$ is defined to have
as objects: left ultrafunctors from some compact Hausdorff space $X$ to
$\mathcal{M}$ and a morphism between
$ \mathcal{F} :X \rightarrow \mathcal{M}$ and
$ \mathcal{G}: Y \rightarrow \mathcal{M}$ consists of a pair
$(f,\alpha )$, where $f $ is a continuous map from $X$ to $Y$ and
$\alpha $ is a natural transformation of left ultrafunctors from
$\mathcal{F}$ to $\mathbf{\mathcal{G} \circ f}$.

This construction resembles the construction
$\mathrm{Comp}_{{\mathcal{M}}}$ in \cite{lurie2018ultracategories}.
$\mathrm{Comp}_{{\mathcal{M}}}$ is defined to have as objects left ultrafunctors
from a compact Hausdorff space $X$ to some ultracategory, and a morphism
from $(X,\mathcal{F})$ to $(Y,\mathcal{G})$ consists of a continuous function
$f$ from $X$ to $Y$, together with a natural transformation of left ultrafunctors
$\alpha $ from $\mathbf{\mathcal{G}} \circ f$ to
$\mathbf{\mathcal{F}}$. In his paper
\cite[Proposition 4.1.5]{lurie2018ultracategories}, Lurie showed this construction
to be a stack over $\mathsf{CompHaus}$ with the latter equipped with the
coherent topology. This indicates that
$\mathsf{CompHaus}^{\mathsf{o}}_{{\mathcal{M}}}$ is also a topological
stack in categories (this is simply the opposite stack).

Now we claim the following result:

%t7.1 #&#
\begin{theorem}
%%LEAP%%%\label{thm7.1}
\label{functoriality_in_the_compact_space_VTEX1}%
The construction $X \mapsto \mathrm{Bun}({\mathcal{M}},X)$ depends contravariantly
on $X$ (which means it defines a Grothendieck fibration over
$\mathsf{CompHaus}$).
\end{theorem}

\begin{proof}
First, we start with the case where
${\mathcal{M}}=\mathsf{k}\text{-}\mathrm{CompMet}$. We did not define what
$X \mapsto \mathrm{Bun}({\mathcal{M}},X)$ should do on morphisms, so we
do that: we define a functor from $\mathrm{Bun}/X$ (which is another way
of calling $\mathrm{Bun}({\mathcal{M}},X)$, here ${\mathcal{M}}$ is fixed
to be $\mathsf{k}\text{-}\mathrm{CompMet}$), to $\mathrm{Bun}/Y$ as follows:
suppose that we have a continuous map $Y \rightarrow X$ and some bundle
$E$ over $X$, then we define a bundle $E^{\prime}$ over $Y$ as the pullback
in $\mathsf{Top}$:
\begin{equation*}
\begin{tikzcd}
E^{\prime} \arrow[d, "\pi _{2}"] \arrow[r, "\underline{f}"] & E
\arrow[d, "\pi _{1}"]
\\
Y \arrow[r, "f"] & X
\end{tikzcd}
\end{equation*}

We need, of course, to verify this is a bundle. Notice that $E^{\prime}$ as
a set is equal to $\coprod _{y \in Y}E_{f(y)}$. The fact that the distance
function is upper semi-continuous on $E^{\prime} \times _{Y} E^{\prime}$ follows
from the following diagram:
\begin{equation*}
\begin{tikzcd}
&& {}
\\
{E^{\prime} \times _{Y} E^{\prime}} &&&&& {E \times _{X} E}
\\
{} & {E^{\prime}} &&& {E}
\\
&& {}
\\
& Y &&& X & {[0,k]_{\text{left
order topology}}} \arrow["{\pi _{1}}", from=3-5, to=5-5]
\arrow["{\pi _{3}}"', from=2-6, to=3-5]
\arrow["f"{description}, from=5-2, to=5-5]
\arrow["{\pi _{2}}", from=3-2, to=5-2]
\arrow["{\pi _{6}}", from=2-1, to=3-2]
\arrow["{\underline{f}}"{description}, from=3-2, to=3-5]
\arrow["d"', from=2-6, to=5-6]
\arrow["{\langle \underline{f} \circ \pi _{5},\underline{f} \circ \pi _{6} \rangle }"', from=2-1, to=2-6]
\arrow["{\pi _{4}}"{pos=0.4}, shift left=2, from=2-6, to=3-5]
\arrow["{\pi _{5}}"', shift right=2,
from=2-1, to=3-2]
\end{tikzcd}
\end{equation*}

The distance function on $E^{\prime}$ is equal to the composition
$d \circ \langle \underline{f} \circ \pi _{5},\underline{f} \circ
\pi _{6} \rangle $, thus it's upper semi-continuous (here $[0,k]$ was equipped
with the left order topology,  $\pi _{1}$ is the projection
of the bundle $E$, $\pi _{2}$ is the projection of the bundle
$E^{\prime}$ and $\pi _{3},\dots \pi _{6}$ denote the pullback maps).

Next, we need to show axiom $(2)$ of the definition of bundles is satisfied,
which means that we need to show that $\pi _{2}$ is continuous and open
but this is straightforward: $\pi _{2}$ is continuous by definition, and
open since the pullback along an open map is an open map.

Finally, we need to show axiom $(3)$ of the definition of bundle, suppose
that we have an element $g \in E^{\prime}$ contained in some open set $W$ and
suppose without loss of generality that $W$ is basic which means that
$W = \pi _{2}^{-1}(K) \bigcap \underline{f}^{-1}(\omega )$ where $K$ is
an open set in $Y$ and $\omega $ is an open set in $E$. Now, since
$E$ is a bundle there exists $\epsilon > 0$, and a neighbourhood
$V$ of $\underline{f}(g)$, such that $V \subseteq _{\epsilon }\omega $. Now using
the fact that by definition $\underline{f}$ is isometric on each fibre,
we have
$g \in \pi _{2}^{-1}(K) \bigcap \underline{f}^{-1}(V) \subseteq _{\epsilon}
W$. So the pullback of a bundle in $\mathsf{Top}$ is again a bundle.

Before continuing, let us describe the morphisms in the fibred category
in the case ${\mathcal{M}}=\mathsf{k}\text{-}\mathrm{CompMet}$: suppose
that $E \xrightarrow[]{\pi _{1}}X$ and
$E^{\prime} \xrightarrow[]{\pi _{2}}Y$ are two bundles a morphism from
$E$ to $E^{\prime}$ is a pair $(f,h)$ where $f$ is a continuous map from $X$ to
$Y$ and $h$ is a map in $\mathrm{Bun}(\mathcal{M}, X)$; equivalently such
morphism can be defined to be a pair $(f,h^{\prime})$, where $h^{\prime}$ is a continuous
map from $E$ to $E^{\prime}$, such that the following diagram commutes:
\begin{equation*}
\begin{tikzcd}
E \arrow[dd, "\pi _{1}"] \arrow[rr, "h^{\prime}"] & & E^{\prime}
\arrow[dd, "\pi _{2}"]
\\
& &
\\
X \arrow[rr] \arrow[rr," f"] & & Y
\end{tikzcd}
\end{equation*}
And such that for every $x$, $h^{\prime}|_{\pi _{1}^{-1}(x)}$ is a contraction.

Now we want to extend the same construction when
$\mathcal{M} = \mathrm{CompMet}_{\mathfrak{L}}$, and also the case
$\mathcal{M} = \mathrm{CompMet}_{\mathfrak{L},\mathbb{T}}$. Suppose we are
given a continuous first-order signature
$\mathfrak{L}$, and a theory (family of axioms)
$\mathbb{T}$. For each sort $S$, we already know that the pullback of
$E_{S}$ is going to be a bundle of complete metric spaces bounded by
$k_{S}$. It remains to show that for every function symbol
$g^{E^{\prime}}$, the global function defined from
$E^{\prime}_{S_{1}} \times _{Y} \dots \times _{Y} E^{\prime}_{S_{n}}$ to
$E^{\prime}_{S^{\prime}}$ is continuous, and for every relation symbol $\phi $, the
global relation defined on
$E^{\prime}_{S_{1}} \times _{Y} \dots \times _{Y} E^{\prime}_{S_{n}}$ is upper semi-continuous.
The proof of both those facts follow exactly the proof that the global
distance function is upper semi-continuous. First starting with a function
symbol $g$:
\begin{equation*}
\begin{tikzcd}
&& {} & {}
\\
{E^{\prime}_{S_{1}} \times _{Y} \dots \times _{Y} E^{\prime}_{S_{n}}} &&&& {} & {E_{S_{1}}
\times _{X} \dots E_{S_{n}}}
\\
\\
\\
&&&&& {}
\\
{E^{\prime}_{S^{\prime}}} &&&&& {E_{S^{\prime}}}
\arrow["{\langle \underline{f}_{S_{1}} \circ \pi _{2,S_{1}},\dots ,\underline{f}_{S_{n}} \circ \pi _{2,S_{n}}\rangle}"', from=2-1, to=2-6]
\arrow["{g^{E^{\prime}}}"', from=2-1, to=6-1]
\arrow["{g^{E}}"', from=2-6, to=6-6]
\arrow["{\underline{f}_{S^{\prime}}}"{description}, from=6-1, to=6-6]
\end{tikzcd}
\end{equation*}

The map $g^{E^{\prime}}$ is the unique map that exists because of the universal
property of $E^{\prime}_{S^{\prime}}$ being a pullback in $\mathsf{Top}$, hence it's
continuous (here $\pi _{2,S_{i}}$ is the projection map of
$E^{\prime}_{S_{1}} \times _{Y} \dots \times _{Y} E^{\prime}_{S_{n}}$ onto
$E^{\prime}_{S_{i}}$). Now for relation symbols, suppose that we have a relation
symbol $\phi $ then the global relation function $\phi ^{E^{\prime}}$ for the
bundle $E^{\prime}$ is the composition
$\phi ^{E} \circ \langle \underline{f}_{S_{1}} \circ \pi _{2,S_{1}},
\dots ,\underline{f}_{S_{n}} \circ \pi _{2,S_{n}} \rangle $:
\begin{equation*}
\begin{tikzcd}
{E^{\prime}_{S_{1}} \times _{Y} \dots \times _{Y} E^{\prime}_{S_{n}}} &&&&& {E_{S_{1}}
\times _{X} \dots E_{S_{n}}} &&& V
\arrow["{\langle \underline{f}_{S_{1}} \circ \pi _{2,S_{1}},\dots ,\underline{f}_{S_{n}} \circ \pi _{2,S_{n}}\rangle}"', from=1-1, to=1-6]
\arrow["{\phi ^{E}}"', from=1-6, to=1-9]
\end{tikzcd}
\end{equation*}
Here $V$ is a compact interval of $\mathbb{R}$ (i.e. of the form
$[a,b]$ where $a,b$ are reals) equipped with the left order topology. Thus
we get that $\phi ^{E^{\prime}}$ is upper semi-continuous.

Now the fact that the construction $E \mapsto E^{\prime}$ where $E^{\prime}$ is the
pullback along $f: Y \rightarrow X$, is a contravariant pseudo-functor
comes from the fact that the pullback along $f \circ f^{\prime}$ is the pullback
along $f^{\prime}$ of the pullback along $f$, up to natural isomorphism.
\end{proof}

We define $\mathrm{Bun}_{\mathfrak{L},\mathbb{T}}$ to be the fibred category
for this pseudofunctor from $\mathsf{CompHaus}$ to $\mathsf{Cat}$.
Explicitly, this is the category that has as objects a bundle
over a topological space $(E,X,(\pi _{S})_{S \in \mathfrak{S}})$ (here
$\mathfrak{S}$ is the set of sorts of the continuous theory
$(\mathfrak{L},\mathbb{T})$) and as morphisms between
$(E,Y,(\pi _{S})_{S \in \mathfrak{S}})$ and
$(E',X,(\pi '_{S})_{S \in \mathfrak{S}})$ a continuous function $f$ from
$Y$ to $X$ together with a map of bundles $g$ between the bundle $E$ and
the bundle $f^{*}(E')$, where the bundle $f*(E')$ is the bundle constructed
by pulling back the bundle $E'$ along $f$ sortwise.

We are going to replace the category
$\mathrm{Bun}_{\mathfrak{L},\mathbb{T}}$ with the equivalent category
$\mathrm{Bun^{\prime}}_{\mathfrak{L},\mathbb{T}}$, in which for every sort
$S$ the bundle $E^{S}$ as a set is equal (not just isomorphic) to
$\coprod _{x \in X}E^{S}_{x}$, and in which the projection is defined by
sending $(x,g) \in E$ to $x$. The reason we did this is because this is
going to force the pullback along the identity to be just the same bundle,
also it forces the pullback along $f \circ f^{\prime}$ to be the pullback along
$f^{\prime}$ of the pullback along $f$ (not just up to isomorphism), hence this
forces the assignment $X \mapsto \mathrm{Bun}^{\prime}/X$ to be a functor and
not just a pseudofunctor. And we are going to rename
$\mathrm{Bun'}_{\mathfrak{L},\mathbb{T}}$ to
$\mathrm{Bun}_{\mathfrak{L},\mathbb{T}}$ (since they are essentially the
same). Now we claim the following result:
%
%t7.2 #&#
\begin{theorem}
\label{thm7.2}
The functors defined by $\mathcal{L}$ on each fibre, extend to an equivalence
of fibred categories between the fibred category (which
we denoted by $\mathrm{Bun}_{\mathfrak{L},\mathbb{T}}$) and the category
$\mathsf{CompHaus}^{\mathsf{o}}_{\mathcal{M}}$, where
$\mathcal{M}$ is the category of models of $\mathbb{T}$.
\end{theorem}
\begin{proof}
For the category of bundles over $X$
$\mathrm{Bun}_{\mathfrak{L},\mathbb{T}}/X$, let us denote by
$\mathcal{L}_{X}$ the equivalence of categories between
$\mathrm{Bun}_{\mathfrak{L},\mathbb{T}}/X$ and
$\mathrm{Lult}(X,\mathcal{M})$ and suppose that we have a continuous function
$f$ from $Y$ to $X$ we want to show that the following diagram commutes:
\begin{equation*}
\begin{tikzcd}
{\mathrm{Bun}_{\mathfrak{L},\mathbb{T}}/Y} &&& {\mathrm{Lult}(Y,
\mathcal{M})}
\\
\\
{\mathrm{Bun}_{\mathfrak{L},\mathbb{T}}/X} &&& {\mathrm{Lult}(X,
\mathcal{M})} \arrow["{\mathcal{L}_{Y}}", from=1-4, to=1-1]
\arrow["{f^{*}}"', from=3-1, to=1-1]
\arrow["{- \ \circ f}"', from=3-4, to=1-4]
\arrow["{\mathcal{L}_{X}}", from=3-4, to=3-1]
\end{tikzcd}
\end{equation*}
Here $f^{*}$ sends a bundle over $X$ to the bundle over
$Y$ obtained by pulling back in $\mathsf{Top}$ sortwise.
Suppose that we have a left ultrafunctor $\mathcal{F}$ from $X$ to
$\mathcal{M}$. First, suppose that
$\mathcal{M}=\mathsf{k}\text{-}\mathrm{CompMet}$:
\begin{equation*}
\begin{tikzcd}
E^{\prime} \arrow[d, "\pi _{2}"] \arrow[r, "\underline{f}"] & E
\arrow[d, "\pi _{1}"]
\\
Y \arrow[r, "f"] & X
\end{tikzcd}
\end{equation*}
The set $E^{\prime}=\coprod _{y \in Y}E_{f(y)}$ admits two bundle topologies
with the same projection map $\pi _{2}$, the first being the pullback topology,
and the second being the topology resulting from the left ultrafunctors
$ \mathcal{F} \circ f$. We now show they coincide:

Before that let us introduce a notation convention: Let
$E=\coprod _{x \in X}M_{x}$, then for any element $ g \in M_{x}$, we are
going to denote the element $(x,g) \in E$ by $g^{(x)}$.

Now suppose that $\mu $ is an ultrafilter on
$E^{\prime}=\coprod _{y \in Y} \mathcal{F}(f(y))$ with the pullback topology,
that converges to some point $g^{(y)}$. First, we have that
$\pi _{2} \mu $ converges to $y$ by definition of the pullback
topology. Now suppose that
$\sigma _{f \pi _{2} \mu }(g^{(f(y))})=(b_{x})_{x \in X}$, then we get
that $\sigma _{ \pi _{2} \mu }(g^{(y)})= (b_{f(y)})_{y \in Y}$ (this follows from \ref{ultresult}).

Now notice the following
$ \coprod _{y \in Y }B(b_{f(y)},\epsilon ) \in \mu \iff \coprod _{x \in X}B(b_{x},
\epsilon ) \in \underline{f}\mu $. Hence $\mu $ converges to $g^{(y)}$ in the topology
resulting from the left ultrafunctor $ \mathcal{F} \circ f$. On the other
hand, suppose that $\mu $ converges to $g^{(y)}$ in the topology resulting
from the left ultrafunctor $\mathcal{F} \circ f$, first we get that
$\pi _{2} \mu $ converges to $\pi _{2}(g)$, also
$\pi _{1} \underline{f} \mu $ converges to $\pi _{1} \underline{f}g$ and
using
$ \coprod _{y \in Y }B(b_{f(y)},\epsilon ) \in \mu \iff \coprod _{x \in X}B(b_{x},
\epsilon ) \in \underline{f}\mu$ we get that $ \underline{f} \mu $ converges to $g^{(f(y))}$. This implies
that the ultrafilter $\mu $ converges to $g$ in the pullback topology.

Now we turn to the case where
$\mathcal{M}=\mathrm{CompMet}_{\mathfrak{L}}$ or
$\mathcal{M}=\mathrm{CompMet}_{\mathfrak{L},\mathbb{T}}$, working in the
same setting (a bundle $E$ over $X$ and continuous function $f$ from
$Y$ to $X)$).

As in the previous case, we are going to get two bundles, one from the
left ultrafunctor composition and the other from taking the pullback of
the bundle $E$ along $f$ sortwise. We know that for each sort the two topologies
on the sorted bundles agree, as we have shown. Also, for each
$y \in Y$, the map between the structures (which are the fibres of
$y$ in both bundles) is the identity, and hence a morphism in the category
of structures (or models).

All these results combined mean that the family of functors
$\{ \mathcal{L}_{X} \}_{X \in \mathsf{CompHaus}} $ defines a natural equivalence
of pseudofunctors between, on one hand the functor that
sends a compact Hausdorff space to the category $\mathrm{Bun}/X$, and
on the other hand the functor that sends a compact Hausdorff space to
the category of left ultrafunctors from $X$ to $\mathrm{CompMet}_{\mathfrak{L},\mathbb{T}}$, or
in other words defines an equivalence of fibred categories
between $\mathrm{Bun}_{\mathfrak{L},\mathbb{T}}$ and
$\mathsf{CompHaus}^{\mathsf{o}}_{\mathcal{M}}$.
\end{proof}

%s8 #&#
\section{Examples}
%%LEAP%%%\label{sec8}
\label{8}

At this point, it is important for us to give examples of our constructions
of bundles and show that they correspond to the already existing notions
of continuous families of metric structures. But first, we need to explain
how to axiomatise some structures in continuous model theory.

%s8.1 #&#
\subsection{Banach bundles}
\label{sec8.1}

\subsubsection*{Axiomatisation of Banach spaces}

The signature of Banach spaces includes a sort for each ball of radius
$n$, inclusion symbols between sorts, and additional symbols for the
$\mathbb{K}$-vector space structure. This means:
\begin{enumerate}
\item A constant symbol $0$ (a function symbol) with formal range
$D_{1}$. This symbol should be interpreted as the $0$ of the vector space
(we can get rid of that symbol since we can get $0$ by multiplication by
$0$, strictly speaking, we must also check whether including this symbol
or not is going to affect the continuous model-theoretic ultraproduct,
since we want to be able to recover the usual definition of ultraproduct
of Banach spaces or any related structure).
\item For each natural number $n \geq 1$, we define a sort $D_{n}$; this
sort should be interpreted as the closed ball of radius $n$.
\item For every pair of sorts we define a function symbol $+_{n,m}$ which
has formal domain $D_{n} \times D_{m}$ and formal range $D_{n+m}$ and should
be interpreted as the addition, the modulus of continuity of this symbol
is $2\mathrm{Id}$.
\item For every pair of sorts $D_{n}$ and $D_{m}$ such that
$n < m $, we define a function symbol $\iota _{n,m}$, which should be interpreted
as the inclusion of the ball $D_{n}$ inside the ball $D_{m}$. This symbol
is of uniform continuity modulus the identity function.
\item For every sort $D_{n}$ and every $k \in \mathbb{C}$, we define a
function symbol $m_{n,k}$; this function should be interpreted as the multiplication.
The formal domain of this symbol is $D_{n}$, and the formal range is
$D_{m}$ where $m =\lceil |k|.n \rceil $, the modulus of this symbol is
$|k|.\mathrm{Id}$.
\item[$6^{\prime}$] If we want to have isometries between Banach spaces, we
are going to add the following relation symbols $k_{n}$ with domain
$D_{n}$, which take values in the interval $[0, n]$, with uniform continuity
modulus the identity function. $k_{n}(x)$ should be interpreted as
$n-\|x\|$. Note that these symbols have not been introduced in the literature
before.
\end{enumerate}

Now we are going to list the necessary axioms informally, here
$\|x\|$ means $d(x,0)$ (notice that it is possible to make the norm an
additional function symbol (sortwise) and add axioms ensuring that the
distance and the norm define the same metric):
\begin{enumerate}
\item $\mathbb{K}$ vector space axioms ($\mathbb{K}$ is by default
$\mathbb{C}$ unless it's indicated to be $\mathbb{R}$).
\item Norm axioms: which are axioms ensuring that the norm (which is defined
for each sort $\|x\| =d(x,0)$) is a norm.
\item Axioms that ensure that the inclusion function is compatible with
distance, addition, and additive inverse.
\item Axioms ensuring that each $D_{n}$ is interpreted as the ball of radius
$n$. These are $\mathrm{Sup}_{x \in D_{1}}(\|x\| \dotminus 1)$, and
$\mathrm{Sup}_{x \in D_{n}}\mathrm{Inf}_{y \in D_{1}}(d(x,\iota _{1,n}(y))
\dotminus (\|x\| \dotminus 1))$, here $\dotminus $ denotes truncated subtraction i.e. $x \dotminus y= \max(x-y,0)$ (what the last axiom is telling us informally is that an
element $x$ in $D_{n}$ has norm less than or equal to $1$, iff there exists
an element $y$ in $D_{1}$ such that $\iota _{1,n}(y)=x$, see
\cite{farah2021model}).
\item[$5^{\prime}$] Axiom ensuring that the new symbol $k_{n}$ is interpreted
as $n-\|x\|$, formally speaking this axiom should be
$\mathrm{Sup}_{x \in D_{n}}|(k_{n}(x) - (n - \|x\|))|$.
\end{enumerate}
As stated before, axiomatising Banach spaces without this newly introduced
symbol leads to the category of Banach space with contractions, while axiomatising
Banach spaces with the newly introduced symbol $k_{n}$, will force maps
to be isometries, and hence this is going to lead to two different notions
of bundles of Banach spaces, which turned out to be already existing in
the literature.

\subsubsection*{Definition of Banach bundles}
\label{definition_of_Baanch_bundles_VTEX1}

This definition is the one present in \cite{hofmann1977bundles}, and we
are going to be calling it a semi-continuous bundle of Banach spaces.

We say that a triple $(E,X,\pi )$ defines a bundle of Banach spaces, where
$E$ and $X$ are topological spaces ($X$ is usually required to be Hausdorff,
in our work we studied the case where the space $X$ is compact Hausdorff) and
$\pi : E \rightarrow X$ is a function required to satisfy the following
conditions:
\begin{enumerate}
\item For every $x$, $\pi ^{-1}(x) $ is a Banach space.
\item $\pi $ is continuous and open.
\item Scalar multiplication from $\mathbb{K} \times E $ to $E$, and addition
from $E\times _{X} E$ to $E$ are continuous.
\item Norm $\| \  \|$ from $E$ to $[0,\infty )$ is upper semi-continuous
(it is not hard to see that in the presence of the other axioms, this is
equivalent to saying that the distance from $E \times _{X} E$ to
$[0,\infty )$ is upper semi-continuous).
\item For any $x \in X$, if we call $\mathcal{N}_{x}$, the set of all open
neighbourhoods of $x$, then
$\{\coprod _{y \in U} B(0_{y},r)\}_{r>0,U \in \mathcal{N}_{x}}$ is a neighbourhood
basis at $0_{x}$.

First, notice that axiom $3$ can be replaced with the following, apparently
weaker axiom $3^{*}$:
\item[$3^{*}$] For each $k \in \mathbb{K}$, the function from $E$ to
$E$ defined by multiplication by $k$ is continuous, also addition from
$E \times _{X} E$ to $E$ is continuous.
\end{enumerate}
An unnecessary condition is imposed in the definition
\cite{hofmann1977bundles} which is requiring the map
$x \mapsto 0_{x}$ to be continuous (we can deduce this easily from condition
$5$).

Another unnecessary condition required in \cite{hofmann1977bundles} is
the requirement that the subspace topology agrees with the Banach space
topology on each fibre. The argument for dropping it can
be found in \cite[proposition 1.3]{fell1969extension} (notice that the
argument uses the fact that the norm is continuous, but this can be easily
replaced by the requirement that the norm is upper semi-continuous since
the neighbourhood filter of $0 \in [0,\infty )$ is the same in the left
order topology and the usual topology), also the argument uses a different
equivalent version of axiom $5$  (namely the axiom
$5^{*}$ that we will be introducing in few paragraphs).

In what follows, a section (or local section) from $U \subseteq X$ to
$E$, here $U$ is open, means a continuous map such that
$\pi \circ f =\mathrm{Id}_{U}$, such section is called global
if $U=X$.

In \cite{hofmann1977bundles} the definition above is called a pre-bundle,
an additional condition is imposed in \cite{hofmann1977bundles} in order
to obtain the definition of bundle: for every $f \in E$, and for every
$\epsilon >0$, there exists a local section $\gamma $ such that
$\| \gamma (\pi (f)) -f \| < \epsilon $.

A bundle for which every element has a global section that hits it is called
a full bundle in \cite{hofmann1977bundles}, another name for this property
is a bundle with enough cross-sections. A good result is that every pre-bundle
over a locally paracompact space is a full bundle (so when it comes to
our work in which we studied bundles over a compact Hausdorff space, every pre-bundle
is a full bundle). This is due to a result by Douady and Dal Soglio-H\'{e}rault
which can be found in the appendix of \cite{Fell1977}.
\noindent
In what follows, we are going to call a semi-continuous Banach bundle a
triple $(E,X,\pi )$ satisfying these five conditions.

Next, we state the following theorem regarding this definition of bundles:
%
%t8.1 #&#
\begin{theorem}
%%LEAP%%%\label{thm8.1}
\label{neighbourhood_theorem_VTEX1}
Let $(E,X,\pi )$ be a Banach bundle, and suppose that
$\mu $ is an ultrafilter on $E$ such that $\pi \mu $ converges to
$y$; furthermore suppose that $\gamma $ is a section such
that  $\gamma (y)=f$, then the set
$\{ \coprod _{x \in U} B(\gamma (x),r) \}_{U \in N_{y},r>0}$ is a basis
for the neighbourhood system at $f$, here $N_{y}$ is the set of all open
neighbourhoods at $y$.
\end{theorem}
\begin{proof}
Take the homeomorphism from $E$ to itself defined by
$g \mapsto g +\gamma (\pi (g))$, and use axiom $5$ in the first definition.
\end{proof}

Before continuing, we should note that there is an alternative way to state
axiom $5$ above:
\begin{enumerate}
\item[$5^{*}$] Suppose that $(b_{i})$ is a net such that
$\|b_{i}\|\rightarrow 0$, and such that $\pi (b_{i}) \rightarrow x$ then
$(b_{i})$ converges to $0_{x}$ (we can write this axiom in ultrafilter
terms as follows if $\mu $ is an ultrafilter on $E$ such that
$ \| \  \| \mu $ converges to $0 \in [0,\infty )$ and $\pi \mu $ converges
to $x$ then $\mu $ converges to $0_{x}$).
\end{enumerate}
Here we should note that when we say that $ \| \  \| \mu $ (or
$\|b_{i}\|$) converges to $0 \in [0,\infty )$, we are either equipping
$[0,\infty )$ with the left order topology, or with the usual topology
because we remind the reader that the neighbourhood filter of $0$ is the
same in these two topologies.

Now, we show that the axioms $5$ and $5^{*} $ are equivalent
(in the presence of the other four axioms):

Let $E$ be a bundle satisfying axioms
$1 \mathrm{-}2\mathrm{-}3\mathrm{-}4\mathrm{-}5^{*}$. We need to check that
the axiom $5$ is satisfied; this axiom states that the set
$\{\coprod _{y \in U}B(0_{y},r)\}_{U \in \mathcal{N}_{x},r>0}$, is a neighbourhood
basis for $0_{x}$.

To show that we can use the Lemma~\ref{important_lemma_VTEX1}, towards this let
$V$ be an open neighbourhood of $0_{x}$ and let $\mu $ be an ultrafilter
on $E$. If
$\{ \coprod _{y \in U} B(0_y,r) \}_{U \in \mathcal{N}_{x},r>0 }
\subseteq \mu $ this would imply that $\| \ \| \mu $ converges to
$0$, and that $ \pi \mu $ converges to $x$. Hence $\mu $ converges
to $0_{x}$, but this implies that $V \in \mu $, and hence by \ref{important_lemma_VTEX1} there exists $r>0$ and $U$ open neighbourhood of
$x$, such that $\coprod _{y \in U}B(0_{y},r ) \subseteq V$.

Now, suppose that we have a bundle satisfying axioms
$1\mathrm{-}2\mathrm{-}3\mathrm{-}4\mathrm{-}5$, we need to check that
axiom $5^{*}$ holds; Towards that suppose that $\mu $ is
an ultrafilter on $E$, such that $ \| \  \| \mu $ converges to
$0 \in [0,\infty )$ and $\pi \mu $ converges to $x$; both these conditions
imply that for any $r >0$ and $U$ open neighbourhood of $x$,
$\coprod _{y \in U}B(0_{y},r) \in \mu $ then $\mu $ converges to
$0_{x}$, since
$ \{ \coprod _{y \in U}B(0_{y},r)\}_{U \in \mathcal{N}_{x},r>0}$ is a basis
of the neighbourhood filter of $0_x$.

There is another definition of Banach bundles given in
\cite{Fell1977}. In that definition, the norm function is required to be
continuous instead of being just semi-continuous. And we are going to call
such a bundle a continuous Banach bundle. Note that in that definition
the \textbf{bundle space $E$} is required to be Hausdorff, but
this requirement can be dropped provided the \textbf{base space} is Hausdorff
(see \cite[16.4]{gierz2006bundles}), our work will provide
an alternative proof of this fact when the base space is
compact Hausdorff.

\subsubsection*{Relating the definition of Banach bundles to our work}
\noindent\textbf{Semi-continuous bundles}
Now we should explain how to relate the concept of semi-continuous Banach
bundles as defined in \cite{hofmann1977bundles}, to the bundles of models
for the continuous model theory of Banach spaces (the classic definition
not including the symbol $k_{n}$). The idea is clear: given
a family of bundles of balls $(E_{n})_{n \in \mathbb{N}}$ over $X$ (a bundle
for which every fibre is the ball of radius $n$ of the Banach space), which
is the notion of bundles corresponding to the continuous model theory of
Banach spaces, we can construct a bundle of Banach space as introduced
by Hofmann by taking $E = \bigcup _{n} E_{n}$ equipped with the final topology
along the inclusion maps. On the other hand, given a bundle in the definition
of Hofmann we can easily recover the bundle in our definition by defining
$E_{n} =\{ f \in E \mid \|f\| \leq n \}$.

%t8.2 #&#
\begin{theorem}%
\label{thm8.2}
Let $X$ be a compact Hausdorff space, then there exists an equivalence
of categories of Banach bundles over $X$, and that of bundles of models
of the continuous model of Banach spaces over $X$.
\end{theorem}

\begin{proof}
Let us further explain to the reader what we are trying
to do, we already have an equivalence of ultracategories between dissections
of Banach spaces and the category of Banach spaces, we want to extend this
equivalence to the level of bundles, and we claim that our notion of bundles
as developed in section~\ref{3} through \ref{6} (the bundles of the continuous
theory of Banach spaces, each of which is a family of sorted bundles
$(E_{n})_{n \in \mathbb{N}}$, where each fibre is exactly the ball of radius
$n$ of the Banach space, satisfying certain axioms), and bundles as a single
topological space, as defined above, are equivalent.

Suppose that we have a family of sorted bundles
$(E_{n})_{n \in \mathbb{N}}$. Take the topological space
$E=\bigcup _{n =1}^{\infty} E_{n}$ (equipped with the final topology along
the inclusion maps i.e. the colimit of
$E_{1} \xhookrightarrow{} E_{2} \xhookrightarrow{} \dots \xhookrightarrow{}  E_{n}
\dots $). This space satisfies the fact that the projection $\pi $ and
scalar multiplication by any $k \in \mathbb{K}$ are continuous, by the
universal property of the final topology of
$\bigcup _{n =1}^{\infty} E_{n}$; The fact that the addition (distance) function is continuous (upper semi-continuous) follows from the following argument: if we equip the space $E \times_{X} E$ with the directed colimit topology (along $(i,j)\leq (i',j')$ iff $i \leq i'$ and $j \leq j'$), then this map is automatically continuous (upper semi-continuous) by the universal property of the directed colimit topology on $E \times_X E$, the problem is that the directed colimit topology is in general finer than the pullback topology on $E \times_X E$. We show they coincide, let $\mu$ be an ultrafilter on $E \times_X E$ converging to $(f,g)$ in the pullback topology, it's not hard to see that this ultrafilter contains the filter $\pi_1 \mu \times \pi_2 \mu$. So it's enough to show that $\pi_1\mu$ restricts to some ultrafilter on some $E_i$ and $\pi_2 \mu$ restricts to an ultrafilter on $E_j$, since if this is the case, then $ \mu$ converges to $(f,g)$ in the colimit topology, but we can get that by upper semi-continuity of the global norm on $E$ in particular $\pi_1 \mu$ restricts to an ultrafilter on $E_{\lceil\|f\|\rceil+1}$, similarly $\pi_2 \mu$ restricts to an ultrafilter on $E_{\lceil\|g\|\rceil+1}$.

The fact that $\pi $ is open follows from the fact that an open set
$V$ in $E$ can be written as
$V =\bigcup _{n=1}^{\infty}V \bigcap E_{n} $, hence
$\pi (V) =\bigcup _{n=1}^{\infty}\pi (V \bigcap E_{n} )$, but
since the restriction of $\pi $ to every sorted bundle is open, then
$\pi $ is open. So we have shown that the bundle
$E=\bigcup _{n \in \mathbb{N}}E_{n}$ satisfies axioms
$1\text{-}2\text{-}3^{*}\text{-}4$.

Finally, let us show that the bundle $E =\bigcup _{n \in \mathbb{N}}E_{n}$
in our definition satisfies axiom $5^{*}$. Let $\mu $ be an ultrafilter
on $E$, such that $ \pi \mu $ converges to $x$, and also suppose that
$ \| \ \| \mu $ converges to $0$.

We know that every bundle of the continuous theory corresponds to a left
ultrafunctor $\mathcal{F}$ from $X$ to the ultracategory
$\mathsf{Ban}_{1}$ i.e. Banach spaces with contractions, this can be done
by regarding the following composition:
\[
\begin{tikzcd}
X &&& {\mathsf{Ban}_{\mathrm{diss}}} &&& {\mathsf{Ban}_{\mathrm{1}}}
\arrow["(\mathcal{F}_{n})_{n \in \mathbb{N}}", from=1-1, to=1-4]
\arrow["\simeq ", from=1-4, to=1-7]
\end{tikzcd}
\]
Here $\simeq $ is the equivalence of ultracategories between
$\mathsf{Ban}_{\mathrm{diss}}$, the category of dissections of Banach spaces
and $\mathsf{Ban}_{1}$ the category of Banach spaces, and
$(\mathcal{F}_{n})_{n \in \mathbb{N}}$ is the family of left ultrafunctors
to $\mathsf{n}\text{-}\mathrm{CompMet}$, which defines a left ultrafunctor
from $X$ to $\mathsf{Ban}_{\mathrm{diss}}$, thus by construction
$\mathcal{F}$ (defined on objects by $\mathcal{F}(x) =E_{x}$) satisfies
the commutativity of the following diagram:
\begin{equation*}
\begin{tikzcd}
{{E_{x}}_{n}} &&& {\int _{X}{E_{y}}_{n}d\mu \simeq (\int _{X}E_{y} d
\mu )_{n}}
\\
\\
{E_{x}} &&& {\int _{X}E_{y} d\mu}
\arrow["{\sigma ^{(n)}_{\pi \mu}}", from=1-1, to=1-4]
\arrow[hook, from=1-1, to=3-1] \arrow[hook, from=1-4, to=3-4]
\arrow["{\sigma _{\pi \mu}}"', from=3-1, to=3-4]
\end{tikzcd}
\end{equation*}

First $\mu $ restricts to an ultrafilter on $E_{1}$, i.e.
$E_{1} \in \mu $, this follows from the fact that $\| \ \| \mu $ converges
to $0$.

Thus since $\sigma _{\pi \mu}(0_{x})=(0_{y})_{y \in X}$ (because it's a
Banach spaces map), we can deduce, by the diagram above that
$\sigma _{\pi \mu}^{(1)}(0_{x})=(0_{y})_{y \in X}$.

Now using semi-continuity of the norm and the fact that
$\| \  \| \mu $ converges to $0$, we get that
$\coprod _{x \in U}B(0_{y},\epsilon ) \in \mu $ for any $U$ open neighbourhood
$U$ of $x$ and $\epsilon >0$, thus $\mu $ converges to
$0_{x}$ (this follows from the definition of topology associated to a left
ultrafunctor \ref{topology}, and we know that every bundle of metric spaces
bounded by $n$ over $X$, comes from a left ultrafunctor from $X$ to
$n\text{-}\mathrm{CompMet}$), thus we showed that our definition satisfies
axiom $5^{*}$, which we showed to be equivalent to axiom $5$.

Now suppose that we have a bundle of Banach spaces $(E,X,\pi)$ in the
definition above. We claim that $(E_{n})_{n \in \mathbb{N}}$, where each
$E_{n}=\{f \in E \mid \|f\| \leq n \}$ equipped with the subspace topology,
is a bundle of the continuous model theory of Banach spaces. First, for
each $E_{n}$ the global distance function is upper semi-continuous
and the restriction of $\pi $ to each $E_{n}$ is continuous.
Now let us show that for each $E_{n}$  axiom
$(1)$ of \ref{sorted_bundle_definition_VTEX1} is satisfied:

Since the base space is compact Hausdorff then the bundle $E$ has enough
cross-sections. Let $W$ be an open set and let $f \in W$, we know that
by \ref{neighbourhood_theorem_VTEX1} there exists a set of the form
$\coprod _{y \in U}B(\sigma (y),\epsilon )$, such that
$\coprod _{y \in U}B( \sigma (y),\epsilon ) \subseteq W$. Now we get
$\coprod _{y \in U}B( \sigma (y),\epsilon /2) \subseteq _{\epsilon /2}
W$, here $U$ is some open neighbourhood of $\pi (f)$. The final thing is
to justify why the sets of the form
$\coprod _{y \in U}B( \sigma (y),\epsilon ) $ are open. To answer this,
notice that they are the image of the sets of form
$\coprod _{y \in U}B( 0,\epsilon ) $ by the homeomorphism defined in the
proof of \ref{neighbourhood_theorem_VTEX1}, and these are open by semi-continuity
of the norm.

Now, to show that $\pi |_{E_{n}}$ is open; let $O$ be an open set in
$E_{n}$, that means that there exists an open set $O'$ in
the topology of $E$, such that $O = O^{\prime} \bigcap E_{n}$. Define
$E_{n}^{\mathrm{o}}$ to be $\{ f \in E \mid \|f\|<n \}$, this set is open
by upper semi-continuity of the norm. Let $x \in \pi (O)$. Take
$f \in O \bigcap \pi ^{-1}(x) $. Since the subspace topology of
$\pi ^{-1}(x) $ agrees with the metric topology of $\pi ^{-1}(x)$ (this
result follows from axiom $(5)$), then $O \bigcap \pi ^{-1}(x)$ is an open
set in the metric topology of $\pi ^{-1}(x) \bigcap E_{n}$, which is the
closed ball $B(0_{x},n)$ in the Banach space $\pi ^{-1}(x)$. This means
that there exists a sequence of elements $(y_{i})$ of $O$ that converges
into $f$, such that $\|y_{i}\| < n $ for every $i$; this means that for
any open set $O$, we have
$\pi (O) =\pi ( O \bigcap E_{n}^{\mathrm{o}}) =\pi (O^{\prime} \bigcap E_{n}^{\mathrm{o}})$
which is open since $\pi $ is open.

Finally, we have that the function from $X$ to $E_{1}$ defined by
$x \mapsto 0_{x}$ is continuous (as we stated before this can be deduced
from axiom $5$ of the definition of Banach bundles), and for any
$n,m$ the inclusion of $E_{n}$ inside  $E_{m}$
is continuous. So the collection $(E_{n})_{n \in \mathbb{N}}$ is a bundle
of structures of the language of Banach spaces as we defined it in \ref{sorted_bundle_definition_VTEX1}, where each fibre is a model of the theory
of Banach space, so this is a bundle of the continuous theory of Banach
spaces as we defined it in \ref{bundle_of_models_VTEX1}.

So far, we have shown that the nested union of every family of sorted bundles
as defined above is a Banach bundle, and vice versa, the dissection of
a Banach bundle is a bundle of the continuous theory of Banach spaces.
We need to check that these two processes (which are obviously functorial)
are inverses: Given a bundle of the continuous model theory
$(E_{n})_{n \in \mathbb{N}}$, it is clear that the topology of each
$E_{n}$ is the subspace topology inside
$\bigcup _{n \in \mathbb{N}}E_{n}$. On the other hand, suppose that we
are given a Banach bundle $E$, we want to show that its topology is the
final topology of the colimit of
$E_{1} \xhookrightarrow{} E_{2} \xhookrightarrow{} \dots
\xhookrightarrow{} E_{n} \dots $, by the universal property of the colimit,
the topology of the colimit is finer than that of $E$, on the other hand,
let $\mu $ be a converging ultrafilter on $E$, since $E$ has a basis of
some $\epsilon $-thin neighbourhood by \ref{thin_lemma_VTEX1}, there exists
$n$, such that $E_{n} \in \mu $, which shows the colimit topology is coarser
than that of $E$ (notice that this is just a generalisation of the argument
that shows the topology of any normed space $M$ is the colimit of
$M_{1} \xhookrightarrow{} M_{2}\xhookrightarrow{} \dots
\xhookrightarrow{} M_{n} \dots $).
\end{proof}

Before continuing let us state a useful lemma that also follows from the
last argument, which extends the result of subsection \ref{charecterisation}.
%
%l8.1 #&#
\begin{lemma}
%%LEAP%%%\label{lem8.1}
\label{third_lemma_VTEX1}
Let $X$ be a compact Hausdorff space, and let $\mathcal{F}$ be a left ultrafunctor
from $X$ to $\mathsf{Ban}_{1}$, and let $E$ the corresponding semi-continuous
Banach bundle, then a set $V \subseteq E$ is open, if for every ultrafilter
$\mu $ on $X$ converging to $x \in \pi (V)$, and every
$f \in V \bigcap \pi ^{-1}(x)$, if
$\sigma _{\mu}(f) =(b_{x})_{x \in X}$, then there exists
$U \in \mu $ and $\epsilon > 0$, such that
$\coprod _{x \in U}B(b_{x}, \epsilon ) \subseteq V$.
\end{lemma}
\noindent\textbf{Continuous bundles}
%
%t8.3 #&#
\begin{theorem}
\label{thm8.3}
Continuous Bundles over $X$ are the bundles of the theory of Banach spaces
as defined above with the new function symbol and the new corresponding
axiom.
\end{theorem}

\begin{proof}
Since we added new relation symbols $k_{n}$, upper semi-continuity in the
global function corresponding to these symbols, implies lower semi-continuity
in norm on each $E_{n}$, which in turn, implies lower semi-continuity of
the norm on the semi-continuous bundle
$E = \bigcup _{n =1}^{\infty}E_{n}$. And thus, the norm function is continuous,
hence we obtain continuous Banach bundles.
\end{proof}

Now we show that the bundle space of continuous Banach bundles is always
Hausdorff (this result is correct even if the base space
is not compact Hausdorff, but in the setting of this paper we restrict
to the compact Hausdorff base space case). For that we first show the
following theorem:

%t8.4 #&#
\begin{theorem}
\label{thm8.4}
Let $E$ be a bundle of models of continuous model theory, which is a single
topological space and satisfies a version of Lemma~\ref{third_lemma_VTEX1} (this
could be just the simple case where we have metric spaces bounded by
$n$, or Banach spaces where we can form the bundle by taking the union
bundles of balls and equipping it with the final topology along the inclusions),
then the bundle is Hausdorff iff the maps $\sigma _{\mu}$ of the left ultrafunctor
associated with the bundle are injective.
\end{theorem}

\begin{proof}
In this proof, we use the left ultrafunctor bundle equivalence, so in particular,
we show the theorem above for bundles constructed using the functor
$\mathcal{L}$ from left ultrafunctors to bundles.

We start the proof as follows: let $\mu $ be an ultrafilter on the bundle space converging
to $f$ and $g$, suppose that
$\sigma _{ \pi \mu}(f)=(a_{x})_{x \in X}$ and also suppose that
$\sigma _{\pi \mu}(g)=(b_{x})_{x \in X}$. We already know that for any
$\epsilon >0$ we have $\coprod _{x \in X}B(b_{x},\epsilon ) \in \mu $, also
we have that $\coprod _{x \in X}B(a_{x},\epsilon ) \in \mu $, but this
simply implies that there exists a set $U \in \pi \mu $, such that for every
$x \in U $ $d(b_{x},a_{x}) < 2 \epsilon $, thus
$d((b_{x}),(a_{x}))\leq 2 \epsilon $, and since $\epsilon $ is arbitrary,
this implies that $(a_{x})= (b_{x})$ and since $\sigma _{\pi \mu}$ is injective,
we deduce that $f=g$ and hence the bundle topology is Hausdorff. On the
other hand, suppose that the bundle is Hausdorff, let $\mu $ be an ultrafilter
on $X$ converging to $x$, and suppose that
$\sigma _{\mu}(f) = \sigma _{\mu}(g)$, we need to show that $f =g $, suppose
$\sigma _{\mu}(f)=(a_{x})_{x \in X}$. Take the family of
sets
$(\coprod _{x \in A }B(a_{x},\epsilon ))_{\epsilon >0, A \in \mu}$, this
is clearly a filter basis, and thus extends to an ultrafilter, this ultrafilter
converges to $f$ and to $g$ and thus since $E$ is Hausdorff, we get that
$f=g$.
\end{proof}

Now we know that in the cases of Banach spaces with isometries, the maps
$\sigma _{\mu}$ are isometries, hence injective, hence the bundle space
is Hausdorff.

%s8.2 #&#
\subsection{Bundles of $\mathrm{C}^*$-algebras}
\label{sec8.2}

The signature of $\mathrm{C}^{*}$-algebras is built on that of Banach spaces. So
we require on top of Banach spaces signature, these additional symbols:
\begin{itemize}
\item For every sort $D_{n}$, we define a function symbol $*_{n}$ from
$D_{n}$ to $D_{n}$, the modulus of this symbol is the identity function.
\item For every pair of sorts $D_{n}, D_{m}$, we define a function symbol
$\mathrm{dot}_{n,m}$ with formal domain $D_{n} \times D_{m}$ and formal
range $D_{n.m}$, the modulus of this symbol is $(n+m) \mathrm{Id}$ (we
are of course going to be writing $xy$ instead of
$\mathrm{dot}_{n,m}(x,y)$).
\end{itemize}

Of course, we require $*$ operation and multiplication axioms in the case
of $\mathrm{C}^{*}$-algebras, for example $\|x\| =\|x^{*}\|$ (which formally
stated is an infinite family of axioms for every sort of the form:
$\mathrm{Sup} _{x \in D_{n}}| \ \|x\| -\|x^{*}\| |$); the $\mathrm{C}^{*}$ identity,
which can be stated as an infinite family of axioms of the form
$\mathrm{Sup}_{x \in D_{n}} ( | \  \|x^{*}x\| - \|x\|^{2}|)$, and of course
the fact that $(x^{*})^{*} =x$ (again axiomatised with an infinite family
of axioms), and axioms ensuring that the algebra is a Banach algebra (for
example we need: $\|xy\| \leq \|x\|. \|y\|$, this can be written formally
as
$\mathrm{Sup}_{x\in D_{n}} \mathrm{Sup}_{y\in D_{m}} \|xy\|
\dotminus \|x\|.\|y\|$ ). For a detailed account of the axioms see
\cite{farah2021model}.

%d8.1 #&#
\begin{definition}
\label{defn8.1}
A semi-continuous bundle of $\mathrm{C}^*$-algebra is a semi-continuous bundle of Banach
spaces, such that every fibre is a $\mathrm{C}^{*}$-algebra, and such that
the global multiplication and $*$ maps are continuous
\cite{forger2013locally,Williams2007CrossedPO,Niels}.
\end{definition}

This concept is equivalent to $C_{0}(X)$ algebras as defined in
\cite{dadarlat2009continuous,blanchard2004global,Niels} (sometimes called
$C(X)$ algebra):

A $C_{0}(X)$ algebra $A$ is defined to be an inclusion $\iota $ of
$C_{0}(X)$ inside $\mathcal{Z}(\mathcal{M}(A))$, such that
$C_{0}(X)A$ is dense in $A$, a detailed account of this equivalence can
be found in \cite[Appendix  C]{Williams2007CrossedPO} or
\cite{Niels}, an important detail to note is that $A$ is the
$\mathrm{C}^{*}$-algebra of continuous sections to the ``topological''
bundle, so, in particular, the space $A$ can be used to define the left
ultrastructure on the left ultrafunctor corresponding to the bundle. Also,
it should be noted that semi-continuous bundles over $X$ are equivalent
to continuous functions from $\mathrm{Prim}(A)$ to $X$ where $A$ is a
$\mathrm{C}^{*}$-algebra \cite{Williams2007CrossedPO}.

Similarly, we may require the global norm function to be continuous, so
we can get continuous bundles of $C^{*}$ algebras as defined in
\cite{dupre1974hilbert,Niels}, this turns out to be equivalent to continuous
fields of $\mathrm{C}^{*}$-algebras as defined in
\cite[CH.10]{dixmier1982c} (for this equivalence see
\cite{dupre1974hilbert}), and to $C_{0}(X)$ algebras satisfying that for
each $a \in A$ the map $N(a)$ on $X$, defined by
$x \mapsto \|a_{x}\|$ is continuous, here $a_{x}$ is the image of
$a$ in the quotient $A/\mathcal{I}_{x}$, where $\mathcal{I}_{x}$ is the
ideal of $A$ generated by elements $\{\iota (f) \mid f(x)=0 \}$
\cite{blanchard2004global}.

Equivalently, these are $C_{0}(X)$ algebras satisfying that
$\mathrm{Res}_{\iota}:\ \mathrm{Spec}(A) \rightarrow \mathrm{Spec}(C_{0}(X))
\simeq X: \mathrm{ker}(\sigma ) \mapsto \mathrm{ker}(\bar{\sigma})
\circ \iota $ is open \cite{Niels}, here $\bar{\sigma}$ is the extension
of the representation $\sigma $ to $\mathcal{M}(A)$, the
multiplier algebra of $A$ (this is the $\mathrm{C}^*$-algebra equivalent of the Stone-Cech
compactification).

These definitions fit our framework, since we want every map corresponding
to a function symbol to be continuous. Of course, when dealing with continuous
$\mathrm{C}^{*}$ bundles we should add the additional relation symbols
$(k_{n})_{n \in \mathbb{N}}$ as we did with Banach spaces. In other words,
semi-continuous bundles correspond to the usual axiomatisation of
$\mathrm{C}^{*}$-algebras which has $*$ homomorphisms as morphisms, while
continuous bundles arise from the axiomatisation of $C^{*}$ algebras with
the additional symbols $(k_{n})_{n \in \mathbb{N}}$ which gives injective
$*$ homomorphisms as morphisms.

%s8.3 #&#
\subsection{Bundles of Hilbert spaces}
\label{sec8.3}

The axiomatisation of Hilbert spaces is also built upon that of Banach
spaces, there are two different ways, one should give us Hilbert spaces
with isometries and the other with contractions. If we want isometries,
we add a family of symbols for the real and imaginary part of the inner
product with specific axioms ensuring it's an inner product.

On the other hand, if we want the maps of models to be just
contractions, this can be done by adding the parallelogram law as an axiom
to the axioms of Banach spaces with contractions, the parallelogram law
can be stated as the following axiom:
$\mathrm{Sup}_{y \in D}\mathrm{Sup}_{x \in D}(| \ \|x-y\|^{2}+\|x+y\|^{2}
-2\|x\|^{2}-2\|y\|^{2} |)$.

%d8.2 #&#
\begin{definition}
\label{defn8.2}
A bundle of Hilbert spaces is a continuous bundle of Banach spaces, where
each fibre is a Hilbert space, more precisely we say that a triple
$(E,X,\pi )$ defines a bundle of Hilbert spaces, where $E$ and $X$ are topological
spaces ($X$ is usually required to be Hausdorff, in our work we studied
the case where the space $X$ is compact Hausdorff) and
$\pi : E \rightarrow X$ is a function required to satisfy the following
conditions:
\begin{enumerate}
\item For every $x$, $\pi ^{-1}(x) $ is a Hilbert space.
\item $\pi $ is continuous and open.
\item Scalar multiplication from $\mathbb{K} \times E $ to $E$, and addition
from $E\times _{X} E$ to $E$ are continuous.
\item Norm $\| \  \|$ from $E$ to $[0,\infty )$ is continuous.
\item For any $x \in X$, if we call $\mathcal{N}_{x}$ the set of all open
neighbourhoods of $x$, then
$\{\coprod _{y \in U} B(0_{y},r)\}_{r>0,U \in \mathcal{N}_{x}}$ is a neighbourhood
basis at $0_{x}$.
\end{enumerate}
\end{definition}

This corresponds exactly to the bundle of the model theory of Hilbert space
with isometries, which is the theory of Hilbert spaces with real and imaginary
parts of the inner product symbols (for each sort).

To see why we necessarily get the continuity of the norm,
notice that in a Hilbert space $\|x\|=\sqrt{\langle x ,x \rangle}$. Now,
by our requirement for a bundle the function $E \times _{X} E$ to
$[0,k]$ defined by $\Re (\langle x ,y \rangle )$ is upper semi-continuous,
but this implies that it is also lower semi-continuous since the map defined
by $\Re (\langle -x , y \rangle )=-\Re (\langle x , y \rangle )$ is upper
semi-continuous (using the fact that multiplying by $-1$ is continuous),
which implies that $\Re (\langle x ,y \rangle )$ is continuous, hence the
norm is continuous.

If we want Bundles of Hilbert spaces with semi-continuous norm, we need
to use the second axiomatisation of Hilbert spaces (with the parallelogram
identity as an axiom), which corresponds to the category of Hilbert spaces
with contractions (these bundles are to our best knowledge not explored
in literature).

%s8.4 #&#
\subsection{Tracially continuous W* bundles}
\label{sec8.4}

Now, we move to a more subtle case, which is tracially continuous
$\mathrm{W}^{*}$-bundles, or bundles of tracial von Neumann algebras, we
use the axiomatisation of tracial von Neumann algebras as present in
\cite{goldbring2022survey}, in which the ultraproduct corresponds to the
tracial ultraproduct \cite{goldbring2022survey}, which is nothing but the Ocneanu ultraproduct (for more
information see \cite{ANDO20146842}) when all states are tracial. The
interesting dichotomy in this axiomatisation is the fact that sorts are
interpreted as bounded operator unit balls with the $\| \ \|_{2}$ topology.

We want to show the following theorem:

%t8.5 #&#
\begin{theorem}
\label{thm8.5}
There is an equivalence of categories, between tracially continuous
$\mathrm{W}^{*}$-bundles over $X$, and the bundle of models of the continuous
model theory of tracial von Neumann algebras over $X$.
\end{theorem}

Tracially continuous $\mathrm{W}^{*}$-bundles are defined (\cite{ozawa2013dixmier},
\cite{bosa2019covering}, \cite{evington2016locally}) as a unital inclusion
of $C(X) \xhookrightarrow[]{} \mathcal{Z}(A)$ where $A$ is a
$\mathrm{C}^{*}$-algebra, together with a $\mathrm{C}^{*}$ conditional
expectation $E$ from $A$ to $C(X)$ satisfying the following properties:
\begin{enumerate}
\item $E(a_{1}a_{2})=E(a_{2}a_{1})$.
\item $E(a^{*}a)=0$ iff $a=0$.
\item the unit ball of $A$ is complete with respect to the $2$-norm defined
by $\|a\|_{2} = \|E(a^{*}a)^{1/2}\|_{C(X)}$.
\end{enumerate}

However, in \cite{evington2016locally}, the authors showed a theorem allowing
us to express a bundle as a topological space $B$ over $X$ (which they
called the topological bundle) satisfying the following axioms:
\begin{itemize}[leftmargin=2.5em]
\item[(i)] (Global) Addition $B \times _{X} B \rightarrow B$ is continuous.
\item[(ii)] (Global) Scalar multiplication, viewed as a map
$\mathbb{C} \times B \rightarrow B$, is continuous.
\item[(iii)] The global $*$ operation viewed as a map
$B \rightarrow B$ is continuous.
\item[(iv)] The map $X \rightarrow B$ which sends $x$ to the additive identity
$0_{x}$ of $B_{x}$ is continuous, and so is the analogous map
$X \rightarrow B$ which sends $x$ to the multiplicative identity
$1_{x}$ of $B_{x}$.
\item[(v)] The map $\tau : B \rightarrow \mathbb{C}$ that
restricts to the corresponding trace on each fibre is continuous, and so
is the map $\|\cdot \|_{2}:B \rightarrow \mathbb{C}$ arising from combining
the 2-norms from each fibre.
\item[(vi)] A net $(b_{\lambda}) \subseteq B$ converges to $0_{x}$ whenever
$\pi (b_{\lambda}) \rightarrow x$ and
$\|b_{\lambda}\|_{2} \rightarrow 0$.
\item[(vii)] Multiplication, viewed as a map
$B \times _{X} B \rightarrow B$, is continuous on $\|\cdot \|$-bounded
subsets.
\item[(viii)] The restriction $\pi |_{B_{1}}: B_{1} \rightarrow X$ is open.
\end{itemize}

\textbf{Note.} If we want to be precise, the equivalence shown in
that paper assumes that the topological bundles have enough cross sections,
but in his PhD thesis Evington showed that any such topological bundle
always have enough cross sections by imitating the argument done by Douady
and Dal Soglio-Herault \cite{EvingtonPhD}.

Before continuing, we can immediately notice that in the presence of the
other axioms, and as in every other definition of metric bundles, axiom
$\mathrm{(ii)}$ can be replaced with:
\begin{itemize}[leftmargin=2.1em]
\item[(ii')] for every scalar $k$, the map $B \xrightarrow[]{k} B$, defined
by scalar multiplication by $k$ is continuous.
\end{itemize}

Also, we can notice that the axiom $\mathrm{(v)}$ can be replaced with
either one of the following equivalent (in the presence of other axioms)
axioms:
\begin{itemize}[leftmargin=2.3em]
\item[(v')] The maps $\tau : \ B \rightarrow \mathbb{C}$ which restricts
to the corresponding trace on each fibre is continuous.
\item[(v'')] The maps
$\|\cdot \|_{2} \ : \ B \rightarrow \mathbb{C}$ arising from combining
the 2-norms from each fibre is continuous.
\end{itemize}
To see why notice that
$\tau (a)=(1/4)\sum _{k=0}^{3} i^{k} \|a+i^{k}.1\|_{2}^{2}$, here
$i$ is the root of $x^{2}+1$.
\newline
We also prefer to write axiom $(\mathrm{vi})$ in a filter language for
our convenience:
\begin{itemize}[leftmargin=2.3em]
\item[(vi')] an ultrafilter $\mu $ on $B$ converges to $0_{x}$ iff
$\pi \mu $ converges to $x$ and $\| \  \|_{2}\mu $ converges to $0$.
\end{itemize}
Showing that $(\mathrm{vi})$ implies $(\mathrm{vi}^{\prime})$ is straightforward,
for the other direction we use Lemma~\ref{important_lemma_VTEX1}, we omit the
details.

\medskip\noindent\textbf{Proof description}
\begin{enumerate}
\item Showing that the $\mathbf{GNS}$ construction from the category of tracial
von Neumann algebras to the category of Hilbert spaces is a left ultrafunctor.
\item Starting from a topological bundle as defined by
\cite{evington2016locally}, we construct a bundle of models of the continuous
model theory of tracial von Neumann algebras. This bundle looks like
$(E_{n})_{n \geq 1}$, where each bundle is a bundle of balls of operator
norm radius $n$ each equipped with the $2$-norm on the tracial von Neumann
algebra, and hence using the left ultrafunctor-bundle of models equivalence,
we get a left ultrafunctor from $X$ to the category of tracial von Neumann
algebras. Let us call this process $\mathcal{F}$.
\item Starting from a left ultrafunctor from $X$ to the category of tracial
von Neumann algebras, we construct a topological bundle as follows, we
take the composition with the $\mathbf{GNS}$ construction to construct
a Hilbert bundle, and then we construct the topological
$\mathrm{W}^{*}$-bundle by equipping the disjoint union of fibres with
the initial topology along the inclusion map defined by the
$\mathbf{GNS}$ construction on each fibre. Let us call this process
$\mathcal{G}$.

At this point, we want to show that the two constructions
$\mathcal{F}$ and $\mathcal{G}$ are inverse of each other.
\item We show that given a left ultrafunctor, then every bundle of operator
norm balls $E_{n}$ equipped with the $2$-norm, coming from the corresponding
bundle of models $(E_{n})_{n \geq 1}$ is homeomorphic onto its image inside
the Hilbert bundle.
\item We show that given a topological bundle $E$, then after constructing
the associated left ultrafunctor and composing with the
$\mathbf{GNS}$ construction, $E$ is going to be homeomorphic onto its image
inside the Hilbert bundle.

This result shows that the two processes are really inverses; to see why:

Starting from a topological bundle $E$, if $\mathcal{F}(E) = F$, then the
corresponding bundle of models to $F$ is $(E_{n})_{n \geq 1}$, and the corresponding
topological bundle to $\mathcal{G}\mathcal{F}(E)$ is homeomorphic to
$E$ using the result $5$ above.

On the other hand, suppose that we have a left ultrafunctor $F$, with corresponding
bundle of models $(E_{n})$, then take $\mathcal{G}(F)=E$;
we have shown that each $E_{n}$ as well as $\mathcal{G}(F)=E$
are homeomorphic onto their images inside the Hilbert bundle, and hence
the bundle of models corresponding to $E$ is going to be just the collection
of operator norm balls of $E$ inside the Hilbert bundle, hence by result
$4$, the bundle of models of tracial von Neumann algebras
$(\mathcal{G}(F))_{n \geq 1}$ and $(E_{n})_{n \geq 1}$ are isomorphic and
hence the left ultrafunctors $\mathcal{F}\mathcal{G}(F)$ and $F$ are isomorphic.
\end{enumerate}

\subsubsection*{The proof}

Before starting our proof, we assume the reader is familiar
with \cite{evington2016locally}, but we are going to briefly explain the
constructions there:

Given a $\mathrm{W}^{*}$-bundle $\mathcal{A}$ over $X$, i.e.
an inclusion of $C(X)$ inside $Z(\mathcal{A})$ together with a conditional
expectation $E$ satisfying certain conditions that we stated before, one
defines the fibres of the topological bundle by taking
$\mathcal{A}_{x}=\mathcal{A}/I_{x}$, where
$\mathcal{I}_{x}=\{a \mid E(a^{*}a)(x)=0\}$,  alternatively
these fibres can be viewed as the images of $\mathcal{A}$ under the
$\mathbf{GNS}$ construction corresponding to the normal state on
$\mathcal{A}$ defined by $a\mapsto E(a)(x)$. Now the bundle topology
on the space $\coprod _{x \in X}\mathcal{A}_{x}$ is defined to be the topology
generated by the basic open sets of the form
$\coprod _{x \in X}\mathsf{B}(a(x),\epsilon )$, for $a$ continuous bounded
section of the projection map.

On the other hand, given a topological bundle, one may define an inclusion
$C(X) \xhookrightarrow{} \mathcal{A}$, where $\mathcal{A}$ is the
$\mathrm{C}^{*}$-algebra of all bounded, continuous sections to the projection
map, and the conditional expectation is defined by
$E(a)(x) =\tau _{x} (a(x))$.

Now we get to see why the case of relating the bundles of the continuous
model theory of tracial von Neumann algebras to that defined above is more
subtle, let us look at the trivial case where $X=\{*\}$, in that case the
bundle reduces to a von Neumann algebra with the $\| \ \|_{2}$ norm topology,
and such topology is $\mathbf{not}$ the inductive limit topology of the
operator norm balls with the $\| \ \|_{2}$ norm. So the question becomes,
for an arbitrary compact Hausdorff space $X$, to find a suitable topology
on the nested union of bundles of bounded operator norm balls (with the
$ \| \ \|_{2}$ topology fibre-wise). In order to do so, we will attempt
to construct a bundle of Hilbert spaces for which every fibre is the
$\mathbf{GNS}$ Hilbert space corresponding to the tracial von Neumann algebra.
Before that, we start by stating a necessary theorem:
%
%t8.6 #&#
\begin{theorem}
\label{thm8.6}
The $\mathbf{GNS}$ construction is a left ultrafunctor from the category of
tracial von Neumann algebras to that of Hilbert spaces (with isometries).
\end{theorem}

\begin{proof}
Let $(B_{i},\phi _{i})_{i \in I}$ be a family of tracial von Neumann algebras
and let $\mu $ be an ultrafilter on $I$, and define the map
$\sigma ^{\prime}_{\mu} \ : \mathcal{H}_{\phi _{\mu}} \rightarrow \int _{I}
\mathcal{H}_{\phi _{i}}d\mu $ by sending $\widehat{(a_{i})}$ to
$(\widehat{a_{i}})$. Of course, we need to make sure that such construction
is well-defined, to do that notice that
$\|\widehat{(a_{i})}\|^{2}_{2,\mu}=\tau _{\mu}(\widehat{(a_{i})}^{*}
\widehat{(a_{i})})=\lim _{\mu}\tau _{i}(a_{i}^{*}a_{i})=\lim _{\mu}\| a_{i}
\|^{2}_{2,i}$ which is by definition the square of the norm in
$\int _{I} \mathcal{H}_{\phi (i)} d\mu $.

\begin{note*}
\normalfont The reader may notice that we defined
$\sigma ^{\prime}_{\mu}$ only on elements of the form $\widehat{(a_{i})}$, but
these elements are by definition, dense in
$\mathcal{H}_{\phi _{\mu}}$.
\end{note*}

This shows that such construction is an isometry and hence well-defined.
Now, showing that the $\mathbf{GNS}$ construction is a left ultrafunctor
is mostly trivial, and requires, for axiom $(3)$ of left ultrafunctor axioms,
using the description of the categorical Fubini transform that we explained
in \ref{description}.

\begin{note*}
\normalfont We defined $\mathbf{GNS}$ only on objects, but its definition
on morphisms is clear. Notice that maps between tracial von Neumann algebras
(seen as models of their continuous model as defined in
\cite{goldbring2022survey}) are $2$-norm isometries (this follows from
a similar argument to the argument that showed that maps
of the continuous model theory of Hilbert spaces with an inner product
symbol are isometries), and hence the induced maps between the
$\mathbf{GNS}$ Hilbert spaces are isometries, as we want them to be (since
we want to work with continuous Hilbert bundles).\qedhere
\end{note*}
\end{proof}
\noindent\textbf{Every topological bundle defines a bundle of models}
Let $B$ be a topological $\mathrm{W}^{*}$-bundle; first we claim that
$(B_{n})_{n \geq 1}$ is a sorted bundle (that means that each
$B_{n}$ is a bundle corresponding to a sort in the continuous first-order
axiomatisation of tracial von Neumann algebras appearing in
\cite{goldbring2022survey}), which in turn implies that it's a left ultrafunctor
from $X$ to the ultracategory of tracial von Neumann algebras.

Now we show that we have a bundle of models:

The first thing we need is that the restriction of the projection to each
sorted ball is open but this is just axiom $(
\mathrm{viii})$; the next thing we are going to show is that the global
functions corresponding to relation and function symbols appearing in the
axiomatisation \cite{goldbring2022survey} are continuous; the continuity
for the $1$ and $0$ symbols for each sorted bundle follows from axioms
$(\mathrm{iv})$; continuity of scalar multiplication with appropriate source
and target sorted bundle follows from axiom $(\mathrm{ii})$; continuity
of addition from appropriate fibre product of sorted bundles follows from
axiom $(\mathrm{i})$, while continuity of subtractions follows from continuity
of addition and multiplication by $-1$, continuity of $*$ operation on
each sorted bundle follows from axiom $(\mathrm{iii})$, continuity of multiplication
on sorted norm balls is just axiom $(\mathrm{vii})$ (reminder that the
sorted bundles are by definition operator norm bounded on each fibre).
Now, the continuity of $2$-norm and trace operations on each sorted bundle
follows from axiom $\mathrm{(v)}$ of the definition
of the topological bundle. So by our equivalence of sorted bundles and
left ultrafunctors, we can deduce that every bundle in the definition of
\cite{evington2016locally}, defines a left ultrafunctor from
the compact Hausdorff base space to the ultracategory of tracial von Neumann
algebras.

\medskip\noindent\textbf{Every left ultrafunctor defines a topological bundle}
Let $X$ be a compact Hausdorff space, and suppose that we
have a left ultrafunctor $\mathcal{F}$ from $X$ to the ultracategory of
tracial von Neumann algebras (reminder that, by composing with
$\mathbf{GNS}$ we get a left ultrafunctor from $X$ to
$\mathsf{Hilb}$, which as we saw earlier defines a Hilbert bundle as defined
in \cite{Fell1977}); now we claim that the space
$\coprod _{x \in X}\mathcal{F}(x)$ equipped with the initial topology of
its inclusion in $\coprod _{x \in X}(\mathbf{GNS}\circ \mathcal{F})(x)$ is
a topological $\mathrm{W}^{*}$-bundle, where
$\coprod _{x \in X}(\mathbf{GNS}\circ \mathcal{F})(x)$ is regarded as bundle
of Hilbert spaces by the adequate topology resulting from the left ultrafunctor
bundle equivalence; to show that we need to show the topological bundle
axioms:

In this proof, we denote $2$-norm ball by $\mathsf{B}(a,r)$ (these can
be subsets of tracial von Neumann algebras or their Hilbert $2$-norm completion),
while if $B$ is a von Neumann algebra bundle, we denote by $B_{n}$ the
subset of $B$ of all elements with operator norm less than or equal to
$n$; we are also going to denote by $\sigma _{\mu}$ the left ultrastructure
of $\mathcal{F}$, $\sigma ^{\prime}_{\mu}$ the left ultrastructure of the
$\mathbf{GNS}$, and by $\sigma ''_{\mu}$ the left ultrastructure of
$\mathbf{GNS} \circ \mathcal{F}$.

Axioms $(\mathrm{i})$ and $(\mathrm{ii})$ are Hilbert bundle properties,
so they extend to subspaces (reminder that what we claim that
$\coprod _{x \in X}\mathcal{F}(x)$ is a W* bundle when equipped with the
initial topology of its inclusion by the Hilbert bundle whose fibres are
the GNS construction of each $\mathcal{F}(x)$).

Now, we turn to axiom $(\mathrm{iii})$, let us call 
$B=\coprod _{x \in X}\mathcal{F}(x)$ with projection map $\pi $, and \newline
$B^{\prime}=\coprod _{x \in X}(\mathbf{GNS}\circ \mathcal{F})(x)$ with projection
map $\pi ^{\prime}$, and let us call $\widehat{~}$ the inclusion map (so
the inclusion of $a$ is $\widehat{a}$ for example).

Let $\mu $ be an ultrafilter on $B$ and let $\widehat{\mu}$ be the ultrafilter
 $\widehat{~} \mu $ (which is the pushforward of $\mu $ by the map
 $\widehat{~}$), and suppose that $\mu $ converges to $a$; we want to
show that $*\mu $ converges to $a^{*}$ or equivalently
$\widehat{ * \mu}$ converges to $\widehat{a^{*}}$ (because we defined the
topology on $B$ to be the initial topology by the map that sends
$a \in B_{i}$ to $\widehat{a} \in \mathcal{H}_{\phi _{i}}$). By
definition, we have that
$\pi \widehat{ * \mu} =\pi ^{\prime}\widehat{ \mu}$ converges to
$\pi (a)$; now suppose that $\sigma _{\pi \mu}(a)=(b_{x})_{x \in X}$, first
let us remind how the map $\sigma ''_{\pi \mu}$ which corresponds to
the left ultrastructure of the composition of $\mathcal{F}$ with the $\mathbf{GNS}$
left ultrafunctor, is constructed. We define
$\sigma ''_{\pi \mu}= \sigma' _{\pi \mu} \circ \mathbf{GNS}(\sigma _{\pi \mu})$.
Then $\sigma ''_{\pi \mu}(\widehat{a})=(
\widehat{b_{x}})_{x \in X}$, now since
$\sigma _{\pi \mu}(a)=(b_{x})_{x \in X}$, and since
$\sigma _{\pi \mu}$ is a $*$ homomorphism, then
$\sigma _{\pi \mu}(a^{*})= (b^{*}_{x})_{x \in X}$, then we get that
$\sigma ''_{\pi \mu}(\widehat{a^{*}})=(\widehat{b^{*}_{x}})_{x \in X}$.
We want to show that for any $\epsilon >0$  $\coprod _{x \in X}\mathsf{B}(b^{*}_{x},\epsilon ) \in \widehat{*\mu}$,
but we already have that
$\coprod _{x \in X}\mathsf{B}(\widehat{b_{x}},\epsilon ) \in
\widehat{\mu}$ (by definition of the topology of Hilbert bundle coming
from left ultrafunctor) which implies that
$\coprod _{x \in X}\mathsf{B}(b_{x},\epsilon ) \in \mu $, and
this implies that
$\coprod _{x \in X}\mathsf{B}(b^{*}_{x},\epsilon ) \in *\mu $, which in
turn implies that
$\coprod _{x \in X}\mathsf{B}(\widehat{b^{*}_{x}},\epsilon ) \in
\widehat{*\mu}$.

Now we get into axiom $\mathrm{(iv)}$, we already have the continuity of
the $0$ selection by a property of Hilbert bundles; for the $1$ selection
we do the same proof as in the case of adjoint by noticing that
$\sigma ''_{\mu}(\widehat{1_{x}})=(\widehat{1_{y}})_{y \in X}$. Axiom
$\mathrm{(v'')}$ is true for any Hilbert bundle. Now we showed that
the axiom $\mathrm{(vi')}$ holds for Banach bundles  (axiom
$5^{*}$ of the definition of Banach bundles present in
\hyperref[definition_of_Baanch_bundles_VTEX1]{8.1}), and so it holds in our case.

For axiom $\mathrm{(vii)}$, we should first remind that the
$\| \ \|_{2}$ norm satisfies the following inequality
$\|ab\|_{2} \leq \|a\| \ \|b\|_{2}$ (because left multiplication on the
Hilbert space is a representation of the von Neumann algebra), also we
have $\|ab\|_{2} \leq \|b\| \ \|a\|_{2}$ (because
$\|ab\|_{2}=\|b^{*}a^{*}\|_{2}$, since the state is tracial). Now, let
$\mu $ be an ultrafilter on $B \times _{X} B$, such that there exists some
$n$ such that $B_{n} \times _{X} B_{n} \in \mu $, that converges to
$(a,b)$, which is equivalent to saying that $\widehat{\mu}$ converges to
$(\widehat{a},\widehat{b})$, where $\widehat{\mu}$ is the pushforward of
$\mu $ by the map $\langle \ \widehat{},\  \widehat{} \ \rangle $. Let
$\epsilon >0$ and suppose that
$\sigma _{\pi \pi _{1} \mu}(a)=(a_{x})_{x \in X}$, and that
$\sigma _{\pi \pi _{1} \mu}(b)=(b_{x})_{x \in X}$. Now using the fact that
$\sigma _{\pi \pi _{1} \mu}$ is a $*$ homomorphism and hence a contraction,
we get the existence of  a set
$X^{\prime} \in \pi \pi _{1} \mu $, such that for any $x \in X^{\prime}$, $\|a_{x}\|
\leq \|a\| +\epsilon '' \leq n +\epsilon ''$, where
$\epsilon ''$ is to be chosen later. Now using the definition of topology
associated to a left ultrafunctor, we have
$\coprod _{x \in X^{\prime}}\mathsf{B}(\widehat{b_{x}},\epsilon ^{\prime}) \in
\widehat{ \pi _{2} \mu}$, and
$\coprod _{x \in X^{\prime}}\mathsf{B}(\widehat{a_{x}},\epsilon ^{\prime}) \in
\widehat{ \pi _{1} \mu}$ for some $\epsilon ^{\prime}$ that we are going to
choose later. Now take any $f$ such that
$\widehat{f} \in \coprod _{x \in X^{\prime}}\mathsf{B}(\widehat{a_{x}},
\epsilon ^{\prime}) \bigcap \widehat{B_{n}}$, and any $g$ such that
$\widehat{g} \in \coprod _{x \in X^{\prime}}\mathsf{B}(\widehat{b_{x}},
\epsilon ^{\prime}) \bigcap \widehat{B_{n}}$, and such that $f,g$ are in the same
fibre; we have that
$\| a_{x}b_{x} -fg\|_{2} \leq \|a_{x}\| \  \|b_{x}-g\|_{2} + \|g\| \ \|a_{x}-f\|_{2}
\leq 2(n+ \epsilon '')\epsilon ^{\prime} $, so we choose
$\epsilon ^{\prime}$ and $\epsilon ''$, such that
$2(n +\epsilon '' )\epsilon ^{\prime} \leq \epsilon $. Now, we know that
$\coprod _{x \in X^{\prime}}\mathsf{B}(\widehat{a_{x}},\epsilon ^{\prime})
\bigcap \widehat{B_{n}} \times _{X} \coprod _{x \in X^{\prime}}\mathsf{B}(
\widehat{b_{x}} ,\epsilon ^{\prime}) \bigcap \widehat{B_{n}} \in
\widehat{\mu} $, and hence
$ \coprod _{x \in X} \mathsf{B}(\widehat{ab},\epsilon ) \in
\widehat{. \mu}$, and hence $\widehat{.\mu}$ converges to
$\widehat{ab}$, and thus $. \mu $ converges to $ab$.

Now, we turn to the last axiom $\mathrm{(viii)}$; let
$V=\widehat{B_{1}} \bigcap W$ be an open set in the subspace topology of
$\widehat{B_{1}}$ (remember that we are equipping $B_{1}$ with the subspace
topology of its image $\widehat{B_{1}}$), here $W$ is an open set in the
topology of the Hilbert bundle. Let $\mu $ be an ultrafilter on $X$ that
converges to some $x \in \pi (B_{1} \bigcap W)$, then there exists
$a \in (B_{1} \bigcap W)$ in the fibre over $x$; we may assume without
loss of generality that $\|a\| < 1$, why is that? Since $W$ is open, then
its intersection with any fibre is open in that fibre (in a Hilbert bundle
the subspace topology agrees with the metric topology on each fibre), then
there exists some $\delta $, such that if
$ \| a -a^{\prime}\|_{2} < \delta $, we have then $a^{\prime} \in W_{x}$, let us
take the element $a^{\prime} =(1-\delta ) a$ then
$\|a^{\prime}\| < \|a\| \leq 1 $ and also $a^{\prime} \in W_{x}$, so we can always
pick $\|a\|< 1$. Suppose that $\sigma _{\mu}(a)=(c_{x})_{x \in X}$, which
implies that
$\sigma ''_{\mu}(\widehat{a})=(\widehat{c_{x}})_{x \in X}$. By the fact
that $\sigma ''_{\mu} $ is contractive ($*$ homomorphism), we may deduce
that $\|(c_{x})\| \leq \|a\| < 1$.

Now, since $W$ is open there exists $U_{1} \in \mu $ and
$\epsilon >0$, such that
$\coprod _{x \in U}\mathsf{B}(\widehat{c_{x}},\epsilon ) \subseteq W $
(using \ref{third_lemma_VTEX1}). Now the fact that
$\|(c_{x})\| \leq \|a\|$ means that for any $\epsilon ^{\prime}$, there exists
a set $U_{2} \in \mu $, such that for any $y \in U_{2}$, we have
$\|c_{y}\|\leq \|a\| + \epsilon ^{\prime}$; we choose $\epsilon ^{\prime}$ such that
$\|a\|+ \epsilon ^{\prime} < 1$. We get that
$U_{1} \bigcap U_{2} \subseteq \pi (W \bigcap B_{1})$, this implies that
$\pi (W \bigcap B_{1}) \in \mu $ thus $\pi (W \bigcap B_{1})$ is open (by
the ultrafilter characterisation of open sets in topological spaces).

We conclude by stating the following theorems:

%t8.7 #&#
\begin{theorem}
\label{thm8.7}
Let $\mathcal{F}$ be a left ultrafunctor from $X$ to the category of tracial
von Neumann algebras, and let $(E_{n})_{n \in \mathbb{N}}$ be the corresponding
family of sorted bundles, then any sorted bundle $E_{n}$ is homeomorphic
onto its image by the $\mathbf{GNS}$ construction on each fibre inside
the Hilbert bundle.
\end{theorem}
\begin{proof}
The proof easily follows from the fact that for any ultrafilter
$\mu $ on the base space $X$, we have the following: if
$\sigma _{\mu}(b)=(c_{x})_{x \in X}$ then
$ \coprod _{x \in X}\mathsf{B}(c_{x},\epsilon ) \in \mu \iff \coprod _{x \in X}
\mathsf{B}(\widehat{c_{x}},\epsilon ) \in \widehat{\mu}$.
\end{proof}
%
%t8.8 #&#
\begin{theorem}
\label{thm8.8}
Let $E$ be a topological $\mathrm{W}^{*}$-bundle, then $E$ is homeomorphic
onto its image by the $\mathbf{GNS}$ on each fibre inside the Hilbert bundle.
\end{theorem}

\begin{proof}
Let $E$ be a topological $W^{*}$ bundle over $X$, the Hilbert bundle is
constructed by showing that the set of sorted bundles $(E_{n})$ corresponds
to a left ultrafunctor from $X$ to the category of tracial von Neumann
algebras, and then by composing with the $\mathbf{GNS}$ left ultrafunctor
and obtaining  a left ultrafunctor from $X$ to
$\mathsf{Hilb}$, which corresponds to a bundle of Hilbert spaces.

Let $(b_{\alpha})$ be a net on $E$ converging to $b$, let us call
$\pi (b)=x$. By the equivalence established by
\cite{evington2016locally}, this bundle can be regarded as a certain inclusion
of $C(X)$ inside the centre of $A$, the $\mathrm{C}^{*}$-algebra of bounded,
continuous sections over the topological bundle, and each fibre in this
case is isomorphic to the quotient of $A$ by $\mathcal{I}_{x}$; All this
implies that there exists a bounded, continuous section $a$, such that
$a(x)=b$. Now the net $(a(\pi (b_{\alpha}))-b_{\alpha})$ converges to
$0_{x}$, this implies that $\pi (a(\pi (b_{\alpha}))-b_{\alpha})$ converges
to $x$, and the norm of $\|(a(\pi (b_{\alpha}))-b_{\alpha}\|_{2}$ converges
to $0$ by axiom $(\mathrm{vi})$ of the topological bundle
axioms; since the $\mathbf{GNS}$ construction on each fibre is $2$-norm
isometric, we get by the axiom $5$ of the definition of Banach bundles,
that $(\widehat{a(\pi (b_{\alpha}))}-\widehat{b_{\alpha}})$ converges to
$\widehat{0_{x}}$. Now, we show that
$(\widehat{a(\pi (b_{\alpha}))})$ converges to $\widehat{b}$; we know that
for sorted bundles the inclusion is a homeomorphism, so it's enough to
show that the net $(a(\pi (b_{\alpha}))$ is operator norm bounded, but
this net is operator norm bounded by $\|a\|$, so we have that.

On the other hand, suppose that $\widehat{(b_{\alpha})}$ converges to
$\widehat{b}$; we wish to show that $(b_{\alpha})$ converges to $b$, in
order to do that we summon a continuous section $a$ to do the reverse of
the last argument, we know that
$(\widehat{a(\pi (b_{\alpha}))}) =
(\widehat{a(\pi ^{\prime}(\widehat{b_{\alpha}}))})$ (here $\pi ^{\prime} $ is the
projection map on the Hilbert bundle) converges to
$\widehat{a(x)}=\widehat{b}$, hence we may deduce that
$(\widehat{a(\pi (b_{\alpha}))}-\widehat{b_{\alpha}})$ converges to
$\widehat{0_{x}}$.

Now $ \  \ \widehat{}$ (the inclusion map)   is a $2$-norm isometry on each fibre and
$\pi (a(\pi (b_{\alpha}))-b_{\alpha})=\pi ^{\prime}(
\widehat{a(\pi (b_{\alpha}))}-\widehat{b_{\alpha}})$ converges to
$x$, this implies that $(a(\pi (b_{\alpha}))-b_{\alpha})$ converges to
$0_{x}$. Now we use the fact that $(a(\pi (b_{\alpha})))$ is operator norm
bounded to deduce that $(a(\pi (b_{\alpha}))$ converges to $a(x)=b$ (since
$(\widehat{a(\pi (b_{\alpha}))})$ converges to
$\widehat{a(x)}=\widehat{b}$, and the bounded bundles $E_{n}$ are homeomorphic
onto their images by the map $\widehat{}$), and hence
$(b_{\alpha})$ converges to $b$.
\end{proof}

So the last two theorems indicate that the two processes we showed first
are inverses of each other, as we have explained in the proof description.
We finish by stating a nice theorem/conclusion to this section:

%t8.9 #&#
\begin{theorem}%
\label{thm8.9}
Let $E$ be a topological $\mathrm{W}^{*}$-bundle, then there exists a Hilbert
bundle whose fibres are the corresponding $ \mathbf{GNS}$ constructions
for each trace, and such that the subspace topology of the Hilbert bundle,
of the subspace which equals on each fibre the image of the tracial von
Neumann algebra by the $\mathbf{GNS}$ construction, is homeomorphic to
the $\mathrm{W}^{*}$-bundle.
\end{theorem}

%s8.5 #&#
\subsection{Note regarding the examples}
\label{sec8.5}

The reader may notice that in the definition of bundles existing in the
literature, we can always define a category of bundles by defining a morphism
of bundles between $(E^{\prime},Y,\pi ^{\prime})$ and $(E,X,\pi )$ to be continuous
maps $f$ and $f^{\prime}$ such that the following diagram commutes:
\begin{equation*}
\begin{tikzcd}
{E{'}} && E
\\
\\
Y && X \arrow["{\pi ^{\prime}}", from=1-1, to=3-1]
\arrow["{f^{\prime}}", from=3-1, to=3-3] \arrow["\pi "', from=1-3, to=3-3]
\arrow["f"', from=1-1, to=1-3]
\end{tikzcd}
\end{equation*}

Now, it's not difficult to see that the equivalence we showed is functorial.
We already know that bundles of models form a topological stack over the
category of compact Hausdorff spaces, with Cartesian lifts given by pullback
in $\mathsf{Top}$ sortwise, this allows us to extend the result to all
bundles above; indeed the argument is easy for Banach, Hilbert and
$\mathrm{C}^{*}$ bundles, since we know that in that case
$E =\varinjlim E_{n}$, where $(E_{n})_{n \geq 1}$ are the sorted bundles,
and taking pullbacks commutes with colimits in the category
$\mathsf{Top}$; for $\mathrm{W}^{*}$-bundles, the argument is more subtle
as usual since it's not true in general $E =\varinjlim E_{n}$.

Our goal is to show the following theorem:

%t8.10 #&#
\begin{theorem}
\label{thm8.10}
Suppose that $(E,X,\pi )$ is a $\mathrm{W}^{*}$-bundle and let
$f: \ Y \rightarrow X$ be a map of compact Hausdorff topological spaces,
then the following are true:%
\begin{itemize}
\item The pullback along $f$ is a $\mathrm{W}^{*}$-bundle.
\item This pullback is the left ultrafunctor associated to
the composition of the left ultrafunctor corresponding to the bundle with
$f$ (regarded as a left ultrafunctor), in other words, it's the Cartesian
lift over $f$.
\end{itemize}
\end{theorem}

\begin{proof}
Let $F$ be the left ultrafunctor corresponding to the bundle $E$; we know
that the Hilbert bundle resulting from composition
$\mathbf{GNS } \circ F \circ f$ is the pullback along $f$ of the Hilbert
bundle resulting from the composition $\mathbf{GNS } \circ F$, let us call
this resulting bundle $H^{\prime}$. Let $E^{\prime}$ be the $\mathrm{W}^{*}$-bundle
corresponding to composition $F \circ f$, then
$E^{\prime}$ is homeomorphic onto its image in $H^{\prime}$, and $E^{\prime}$ as a set
is the pullback of $E$ (the justification of this statement is that at
the level of each sort $E^{\prime}_{n}$ is the pullback of $E_{n}$ as we have
shown), but since pullbacks commute with taking subspaces, then
$E^{\prime}$ is the pullback of $E$ in $\mathsf{Top}$.
\end{proof}

%s9 #&#
\section{Application: another proof of Lurie's result}
%%LEAP%%%\label{sec9}
\label{9}

Now we use this already developed theory to find another proof of Lurie's
result of equivalence between \emph{sheaves of sets} (where the site is
$\mathcal{O} (X)$ where $X$ is compact Hausdorff) and \emph{left ultrafunctors}
from $X$ to $\mathsf{{Set}}$:

%t9.1 #&#
\begin{theorem}
\label{thm9.1}
Let $X$ be a compact Hausdorff space, then there is an equivalence of categories
between $\mathrm{Sh}(X)$ and the category
$\mathrm{Lult}(X, \mathsf{Set})$.
\end{theorem}

Before proving this result, notice that the category $\mathsf{{Set}}$ is equivalent
to the category of discrete metric spaces. Let us axiomatise discrete metric
spaces using continuous model theory: The language of discrete
metric spaces is mono-sorted with an upper bound to distance the constant
$1$, with no function symbols, and no relation symbols (unless you want
to count the distance as a relation symbol). If we call $S$ the single
sort, we get that the set $\mathbb{T}$ of axioms for discrete metric spaces
contains only one sentence
$\mathrm{Sup}_{x \in S} \mathrm{Sup}_{y \in S} \mathrm{min} (d(x,y),|1
- d(x,y)|)$ which translates to the fact that the metric is discrete. We
note also that the equivalence between discrete metric spaces axiomatised
this way and $\mathsf{{Set}}$ is an equivalence of ultracategories (preserves
the ultraproduct).

Now we want to show that the bundle $E_{S}$ is an \'{e}tale space over
$X$, remember that this is equivalent to saying that the diagonal map:
\begin{equation*}
\begin{tikzcd}
{E_{S}} && { E_{S} \times _{X} E_{S}}
\arrow["\Delta ", from=1-1, to=1-3]
\end{tikzcd}
\end{equation*}
is open (in the case where the projection $\pi $ onto the base is open,
which we have by definition of bundles of continuous theories), which is
equivalent to saying that the diagonal of $E_{S} \times _{X} E_{S}$ is
open, since the diagonal map is an embedding. It follows from upper semi-continuity
of the distance function on $E_{S}$, that for any $\epsilon > 0$ the set
$\{ (v,v^{\prime}) \in E\times _{X} E \mid d_{\pi (v)}(v,v^{\prime}) <\epsilon \
\}$ is open, so if we take $\epsilon <1$, we get the diagonal of
$E \times _{X}E$ (since the metric on each fibre is discrete).

On the other hand, suppose that we have an \'{e}tale space
$(E,X,\pi )$ then let us prove that, in this case, we get a bundle of discrete
metric spaces: The three axioms are easily verifiable: for axiom
$(1)$ let $(e,f) \in E \times _{X} E$, the case where $e \neq f$ is trivial,
thus let us suppose that $e=f$, in that case, we know that there exists
a neighbourhood $U$ of $e$, such that $\pi (U) \simeq U$. Now take the
neighbourhood $U \times _{X} U$ of $(e,e)$, for any
$(g,h) \in U \times _{X} U$, we have that $g=h$ and hence $d(g,h)=0$, and
hence distance is upper semi-continuous. For axiom $(2)$ $\pi $ is continuous
by definition and also it's known that $\pi $ is open~(\cite[II.6 Proposition 1]{maclane2012sheaves}).
For axiom $(3)$ let $e \in E$, and let $W$ be a neighbourhood of $e$; since
$e \in E$ there exists some neighbourhood $U$ of $e$, such that
$\pi (U) \simeq U$ (via $\pi |_{U}$); now take $V=U \bigcap W$ and any
$0<\epsilon <1$, we claim that $V_{\epsilon}=V$ and it's easy to see why.

These maps extend to morphisms, since morphisms of \'{e}tale spaces over
$X$, and maps of bundles of discrete metric spaces are defined the same
way.

So, we get an equivalence between \'{e}tale spaces and bundles of sets
(seen as discrete metric spaces). Now we already know that \'{e}tale spaces
are sheaves of sets on $X$, on the other hand, we also know that bundles
of discrete metric spaces are equivalent to left ultrafunctors from
$X$ to the category of discrete metric spaces, which is equivalent to that
of sets.

\begin{note*}
\normalfont Let $E$ be an \'{e}tale space over $X$, then we have a good
description of the left ultrastructure of the associated left ultrafunctor,
since by definition the \'{e}tale space has enough local sections. So if
$\mu $ converges to $x \in X$, then
$\sigma _{\mu}(a) =(f(y))_{y \in U}$, where $f$ is the local section
that hits $a \in E_{x}$.
\end{note*}

%s10 #&#
\section{Bundles of pointed metric spaces}
%%LEAP%%%\label{sec10}
\label{10}

There is a natural notion of ultraproduct of complete pointed metric spaces,
constructed in a similar fashion to the ultraproduct of bounded metric
spaces as follows: suppose $(M_{i},p_{i})_{i \in I}$ is a family of pointed
metric spaces, and suppose that $\mu $ is an ultrafilter on $I$; we define
$\int _{I} M_{i} d\mu $ to be the space of all bounded sequences (with
respect to the point of each space), quotiented by the equivalence relation
$(x_{i}) \sim (y_{i})$ iff $\lim _{\mu}d_{i}(x_{i},y_{i}) =0$, and by taking
as point for the space the equivalence class of $(p_{i})_{i \in I}$. The
same argument as in the bounded case shows that such space is complete,
this allows the definition of an ultracategory of metric spaces with contractions.

%d10.1 #&#
\begin{definition}
\label{defn10.1}
We say that $(E,X,\pi )$ where $X$ and $E$ are topological spaces and
$\pi : \  E \rightarrow X$, such that for every $x$, $\pi^{-1}(x)$ is a complete pointed metric space, defines a bundle of complete pointed metric spaces if it satisfies the following set of axioms:
\begin{itemize}
\item Axiom(1): The global distance function is upper semi-continuous.
\item Axiom(2): $\pi $ is continuous and open.
\item Axiom(3): For every open set $W$ and every $f \in W$ there exists
an  open neighbourhood $V$ of $f$ and $\epsilon >0$, such that
$V \subseteq _{\epsilon} W$.
\item Axiom(4): The point selection function $x \mapsto p(x)$ is continuous.
\label{definition2}
\end{itemize}
\end{definition}
Let $\mathsf{Point}_{1}$ denote the category of complete pointed metric spaces with
contractions.
%
%t10.1 #&#
\begin{theorem}%
\label{thm10.1}
Let $X$ be a compact Hausdorff space, then there exists an equivalence
of categories of bundles of pointed complete metric spaces over $X$, and
left ultrafunctors from $X$ to $\mathsf{Point}_{1}$.
\end{theorem}

We are not going to do the proof of this theorem. The proof is just repeating
the steps of the equivalence between left ultrafunctors from $X$ to
$\mathsf{k} \text{-} \mathrm{CompMet}$ and bundles of metric spaces bounded
by $k$ as was done in \ref{3}.

This equivalence is not part of the examples section, because we don't
know of any axiomatisation of pointed complete metric spaces in continuous
model theory. Indeed, if we try to imitate the Banach space case, and define
a language with a sort intended to be interpreted as the closed ball of
a radius $n$ for every $n \in \mathbb{N}$, and a constant symbol for the
point of the space ($p \in B_{1}$), then dissections of pointed metric
spaces are not an elementary class in this language; to justify this take
the following example: for every $i \in \mathbb{N}$, let
$X_{i}=\{p_{i}, x_{i}\}$ and let
$d(p_{i}, x_{i}) =1 +1/i$; the dissection of such space is
${X_{i}}_{1}= \{p_{i}\}$ and ${X_{i}}_{n}= \{p_{i},x_{i}\}$ for
$n >1$. Let $\mu $ be a non-principal ultrafilter on $\mathbb{N}$, then
$(\int {X_{i}}_{1} d\mu ) = \{ (p_{i})_{i \in \mathbb{N}}\}$, but if
$(\int {X_{i}}_{n} d\mu )_{n \in \mathbb{N}}$ was a ``model of the continuous
theory of pointed complete spaces'', then $(\int {X_{i}}_{1} d\mu )$ should
be $\{ (p_{i})_{i \in \mathbb{N}},(x_{i})_{i \in \mathbb{N}}\}$. So
we may deduce that dissections of pointed spaces are not
axiomatisable in this language.

One final thing to note is that when we showed that Banach (Hilbert,
$\mathrm{C}^{*}$, etc.) bundles are equivalent to left ultrafunctors from
their base spaces to their respective categories, we have
used an approach based on bundles of complete bounded metric spaces as
our main building block. Another viable approach could have been using
bundles of pointed metric spaces as the main ingredient. Indeed, some of
the results may have been easier to show, but our approach has the advantage
of having a notion of bundles that works for any continuous theory.

\section*{Acknowledgment}
\label{sec:ack}
\addcontentsline{toc}{section}{\nameref{sec:ack}}
 This work was done as part of my Doctorate research. I would like to express gratitude to my thesis supervisor Simon Henry for his valuable guidance, expertise and feedback. This work was partially supported by the Natural Sciences and Engineering Research Council of Canada (NSERC), funding reference number RGPIN-2020-067 awarded to Simon Henry, and by the Ontario Ministry of Colleges and Universities through the Ontario Graduate Scholarship and the QEII Graduate Scholarship in Science and Technology.

\bibliographystyle{alphaabbr}
{\footnotesize
\bibliography{bib}}
\end{document}